\documentclass{amsart}
\usepackage{amsmath,amscd}
\usepackage{amssymb}
\usepackage{bbm}
\usepackage{enumerate}
\usepackage{amssymb}
\usepackage{array}
\usepackage{graphicx} 
\usepackage{amsthm}
\usepackage{units}
\usepackage{dsfont}
\usepackage[arrow, matrix, curve]{xy}
\usepackage[left,modulo]{lineno}
\usepackage{tikz-cd}
\usepackage{mathrsfs}
\usepackage{amscd}
  
\newtheorem{theorem}{Theorem}[section]  
\newtheorem{lemma}[theorem]{Lemma}

\newtheorem{conjecture}[theorem]{Conjecture}

\theoremstyle{definition}
\newtheorem{definition}[theorem]{Definition}

\theoremstyle{remark}

\numberwithin{equation}{section}

\newcommand{\op}{\operatorname}
\newcommand\numberthis{\addtocounter{equation}{1}\tag{\theequation}}

\begin{document}

\title{Irreducible Characters and Semisimple Coadjoint Orbits}

\author{Benjamin Harris}
\email{benjaminlharris@outlook.com}

\author{Yoshiki Oshima}
\address{Graduate School of Information Science and Technology,
 Osaka University, 1-5 Yamadaoka, Suita, Osaka 565-0871, Japan}
\email{oshima@ist.osaka-u.ac.jp}

\subjclass[2010]{22E46}

\date{October 27, 2017}


\keywords{Semisimple Orbit, Coadjoint Orbit, The Orbit Method, Kirillov's Character Formula, Cohomological Induction, Parabolic Induction, Reductive Group}

\begin{abstract}
When $G_{\mathbb{R}}$ is a real, linear algebraic group, the orbit method predicts that nearly all of the unitary dual of $G_{\mathbb{R}}$ consists of representations naturally associated to orbital parameters $(\mathcal{O},\Gamma)$. If $G_{\mathbb{R}}$ is a real, reductive group and $\mathcal{O}$ is a semisimple coadjoint orbit, the corresponding unitary representation $\pi(\mathcal{O},\Gamma)$ may be constructed utilizing Vogan and Zuckerman's cohomological induction together with Mackey's real parabolic induction. In this article, we give a geometric character formula for such representations $\pi(\mathcal{O},\Gamma)$. Special cases of this formula were previously obtained by Harish-Chandra and Kirillov when $G_{\mathbb{R}}$ is compact and by Rossmann and Duflo when $\pi(\mathcal{O},\Gamma)$ is tempered.
\end{abstract}

\maketitle

\section{Introduction}
\label{sec:intro}

Let $G$ be a connected, complex reductive algebraic group, let $\sigma$ be an anti-holomorphic involution of $G$, and let $(G^{\sigma})_e\subset G_{\mathbb{R}}\subset G^{\sigma}$ be an open subgroup of the fixed point set of $\sigma$. In this paper, we will call such a group $G_{\mathbb{R}}$ a real, reductive group. Let $\mathfrak{g}$ (resp.\ $\mathfrak{g}_{\mathbb{R}}$) denote the Lie algebra of $G$ (resp.\ $G_{\mathbb{R}}$), let 
\[\mathfrak{g}^*:=\operatorname{Hom}_{\mathbb{C}}(\mathfrak{g},\mathbb{C}),\ \sqrt{-1} \mathfrak{g}_{\mathbb{R}}^*:=\operatorname{Hom}_{\mathbb{R}}(\mathfrak{g}_{\mathbb{R}},\sqrt{-1}\mathbb{R})=(\mathfrak{g}^*)^{-\sigma}\] 
denote the dual space of $\mathfrak{g}$ and the collection of purely imaginary valued linear functionals on $\mathfrak{g}_{\mathbb{R}}$ respectively. If $\xi\in \sqrt{-1}\mathfrak{g}_{\mathbb{R}}^*$, then we denote by $G(\xi)$ (resp.\ $G_{\mathbb{R}}(\xi)$) the stabilizer of $\xi$ in $G$ (resp.\ $G_{\mathbb{R}}$), and by $\mathfrak{g}(\xi)$ (resp.\ $\mathfrak{g}_{\mathbb{R}}(\xi)$) its Lie algebra. If $\xi\in \sqrt{-1}\mathfrak{g}_{\mathbb{R}}^*$, then $G_{\mathbb{R}}(\xi)$ admits a natural Duflo double cover $\widetilde{G_{\mathbb{R}}}(\xi)$ (see Section \ref{subsec:double_covers} for a definition). A one-dimensional, unitary representation of $\widetilde{G_{\mathbb{R}}}(\xi)$ is \emph{genuine} if it does not factor through the quotient $G_{\mathbb{R}}(\xi)$. A coadjoint $G_{\mathbb{R}}$-orbit $\mathcal{O}\subset \sqrt{-1}\mathfrak{g}_{\mathbb{R}}^*$ is called \emph{semisimple} if it is closed.

\begin{definition} \label{def:semisimple_orbital_parameter}
A \emph{semisimple orbital parameter} for $G_{\mathbb{R}}$ is a pair $(\mathcal{O},\Gamma)$ where 
\begin{enumerate}[(a)]
\item $\mathcal{O}\subset \sqrt{-1}\mathfrak{g}_{\mathbb{R}}^*$ is a semisimple, coadjoint orbit
\item for every $\lambda\in \mathcal{O}$, $\Gamma_{\lambda}$ is a genuine, one-dimensional, unitary representation of $\widetilde{G_{\mathbb{R}}}(\lambda)$.
\end{enumerate}
In addition, this pair must satisfy
\begin{enumerate}[(i)]
\item $g\cdot \Gamma_{\lambda}=\Gamma_{\op{Ad}^*(g) \lambda}$ 
for every $g\in G_{\mathbb{R}}$, $\lambda\in \mathcal{O}$
\item $d\Gamma_{\lambda}=\lambda|_{\mathfrak{g}_{\mathbb{R}}(\lambda)}$
for every $\lambda\in \mathcal{O}$.
\end{enumerate}
\end{definition}

Giving a semisimple orbital parameter is equivalent to giving a $G_{\mathbb{R}}$-equivariant Hermitian quantum bundle with equivariant curvature $\lambda\mapsto \mu(\lambda)+\Omega$ on a closed coadjoint orbit (see pages 261, 276-277, 289-290 of \cite{Ve94} for terminology).

Mackey's work \cite{Ma49}, \cite{Ma52} reduces the problem of constructing a unitary representation $\pi(\mathcal{O},\Gamma)$ for every semisimple orbital parameter $(\mathcal{O},\Gamma)$ to the case of elliptic orbital parameters. To every elliptic orbital parameter $(\mathcal{O},\Gamma)$, Zuckerman (unpublished) and Vogan \cite{Vog81}, \cite{Vog84} associate an at most finite sum $\pi(\mathcal{O},\Gamma)$ of irreducible, unitary representations of $G_{\mathbb{R}}$. In the case where $\mathcal{O}$ is in the good range (roughly meaning that $\mathcal{O}$ is sufficiently far from the nilpotent cone, see Section \ref{subsec:VZ_construction} for a precise definition), the representation $\pi(\mathcal{O},\Gamma)$ is irreducible. The special case where $\mathcal{O}$ is of maximal dimension was carried out earlier by Harish-Chandra \cite{HC65c}, \cite{HC66}, and the special case where $G_{\mathbb{R}}$ is compact was understood much earlier by Cartan \cite{Car13} and Weyl \cite{Wey25}, \cite{Wey26a}, \cite{Wey26b}.
\bigskip

If $\pi$ is an irreducible, admissible representation of $G_{\mathbb{R}}$, we denote the Harish-Chandra character of $\pi$ by $\Theta(\pi)\in C^{-\infty}(G_{\mathbb{R}})$, and we denote by $\theta(\pi)\in C^{-\infty}(\mathfrak{g}_{\mathbb{R}})$
the Lie algebra analogue of $\Theta(\pi)$ (see Section \ref{subsec:Rossmann_characters} for definitions). When $(\mathcal{O},\Gamma)$ is a semisimple orbital parameter and $\mathcal{O}$ is in the good range, we write $\theta(\mathcal{O},\Gamma)$ instead of $\theta(\pi(\mathcal{O},\Gamma))$. If $\mathfrak{c}$ is the universal Cartan subalgebra of $\mathfrak{g}$ with Weyl group $W$, then every irreducible, admissible representation $\pi$ of $G$ has an \emph{infinitesimal character}, which is a $W$-orbit in $\mathfrak{c}^*$. Further, there exists a natural fibration
\[q\colon \mathfrak{g}^*\rightarrow \mathfrak{c}^*/W.\]
We denote the inverse image of the infinitesimal character of an irreducible, admissible representation $\pi$ of $G_{\mathbb{R}}$ under $q$ by $\Omega(\pi)$, and we call $\Omega(\pi)$ the \emph{geometric infinitesimal character} of $\pi$ (see Section \ref{subsec:Rossmann_characters} for further discussion). We say $\pi$ has \emph{regular infinitesimal character} if the infinitesimal character of $\pi$ is a $W$-orbit in $\mathfrak{c}^*$ consisting of regular elements or equivalently if $\Omega(\pi)$ is a single closed, coadjoint $G$-orbit. In this case, $\Omega(\pi)$ has a natural $G$-invariant symplectic form $\omega$ (so-called Kirillov-Kostant-Souriau form). 

Let $\Omega\subset \mathfrak{g}^*$ be a regular coadjoint $G$-orbit with complex dimension $2m$, and let $\mathcal{M}\subset \Omega$ be an oriented, middle-dimensional, real analytic submanifold that is bounded in the real direction and satisfies a tempered growth condition in the imaginary direction (see Section \ref{subsec:Rossmann_characters} for precise definitions). If $\mathcal{F}[\mu]$ denotes the Fourier transform of a smooth compactly supported density $\mu\in C_c^{\infty}(\mathfrak{g}_{\mathbb{R}},\mathcal{D}(\mathfrak{g}_{\mathbb{R}}))$ (see Section \ref{subsec:Rossmann_characters} for a precise definition), then we may define the \emph{Fourier transform} of $\mathcal{M}$ to be
\begin{equation}\label{eq:Fourier_transform}
\langle \mathcal{F}[\mathcal{M}], \mu\rangle:=\int_{\mathcal{M}} \mathcal{F}[\mu] \frac{\omega^{\wedge m}}{(2\pi\sqrt{-1})^m m!}.
\end{equation}
\noindent for every $\mu\in C_c^{\infty}(\mathfrak{g}_{\mathbb{R}},\mathcal{D}(\mathfrak{g}_{\mathbb{R}}))$. This Fourier transform appears to depend on a choice of $\sqrt{-1}$, but we give an orientation on $\mathcal{M}$ which depends on a choice of $\sqrt{-1}$ and makes (\ref{eq:Fourier_transform}) canonically defined (see Section \ref{subsec:Rossmann_characters} for details). We define an \emph{admissible contour} to be a finite linear combination of oriented, analytic submanifolds $\mathcal{M}\subset \Omega$ satisfying the above conditions. Rossmann showed that for every irreducible, admissible representation $\pi$ of $G_{\mathbb{R}}$ with regular infinitesimal character, there exists an admissible contour $\mathcal{C}(\pi)\subset \Omega(\pi)$ satisfying 
$$\mathcal{F}[\mathcal{C}(\pi)]=\theta(\pi).$$
In general, the contour $\mathcal{C}(\pi)$ is not unique. Two contours which are homologous in a suitable sense have identical Fourier transforms (see Section \ref{subsec:Rossmann_characters}).

\begin{theorem} [Rossmann \cite{Ro78}]  If $G_{\mathbb{R}}$ is a real, reductive group, $(\mathcal{O},\Gamma)$ is a semisimple orbital parameter, and $\mathcal{O}\subset \sqrt{-1}\mathfrak{g}_{\mathbb{R}}^*$ is of maximal dimension among coadjoint $G_{\mathbb{R}}$-orbits, then 
\begin{equation}\label{eq:character_tempered}
\mathcal{F}[\mathcal{O}]=\theta(\mathcal{O},\Gamma).
\end{equation}
\end{theorem}

In this classical case, the Fourier transform reduces to the classical Fourier transform of a $G_{\mathbb{R}}$-invariant density on $\mathcal{O}$. Rossmann's proof uses a reduction to the case of discrete series found in earlier work of Duflo \cite{Du70}. Additional proofs were provided later by Vergne \cite{Ve79} and Berline-Vergne \cite{BV83b}. In the case where $G_{\mathbb{R}}$ is compact, this result was obtained earlier by Harish-Chandra \cite{HC57a} and Kirillov \cite{Kir68}. In some special cases, analogous results have been found for non-reductive Lie groups \cite{Puk67}, \cite{Kir68},  \cite{Puk69}, \cite{Duf70b}, \cite{Kha82}. However, if $(\mathcal{O},\Gamma)$ is a semisimple orbital parameter for a real, reductive group $G_{\mathbb{R}}$ and $\mathcal{O}$ is not of maximal dimension, then $\mathcal{F}[\mathcal{O}]$ is a tempered distribution on $\mathfrak{g}_{\mathbb{R}}$ and $\theta(\mathcal{O},\Gamma)$ is often not a tempered distribution on $\mathfrak{g}_{\mathbb{R}}$; hence, the two cannot be equal in general. This unfortunate fact was observed in the 1980s; the goal of this paper is to find a suitable remedy.
\bigskip

For every semisimple orbital parameter $(\mathcal{O},\Gamma)$ where $\mathcal{O}$ is in the good range, we wish to construct an admissible contour $\mathcal{C}(\mathcal{O},\Gamma)$ for which
$$\mathcal{F}[\mathcal{C}(\mathcal{O},\Gamma)]=\theta(\mathcal{O},\Gamma).$$
We might first attempt to construct an admissible contour $\mathcal{C}(\mathcal{O},\Gamma)$ by creating a fiber bundle
$$\mathcal{C}(\mathcal{O},\Gamma)\rightarrow \mathcal{O}$$
where the fiber over $\lambda\in \mathcal{O}$ is a character formula for the one-dimensional representation $\Gamma_{\lambda}$ of $\widetilde{G_{\mathbb{R}}}(\lambda)$. One checks that in the case where $\mathcal{O}$ is of maximal dimension, this fiber will be a single point  and we recover $\mathcal{C}(\mathcal{O},\Gamma)=\mathcal{O}$ as desired. However, in general it turns out that such a construction leads to a contour that is not admissible in the sense of Rossmann (see Section \ref{subsec:philosophy} for additional explanation). To alleviate this difficulty, we require a polarization of $\mathcal{O}$.

A (complex) \emph{polarization} of $\mathcal{O}$ is a $G_{\mathbb{R}}$-equivariant choice of an integrable, Lagrangian subspace of the complexified tangent spaces $T_{\lambda}\mathcal{O}\otimes_{\mathbb{R}} \mathbb{C}$ for each $\lambda\in \mathcal{O}$. More concretely, it is a choice of complex parabolic subalgebra $\mathfrak{q}_{\lambda}\subset \mathfrak{g}$ with Levi factor $\mathfrak{g}(\lambda)$ for every $\lambda\in \mathcal{O}$ such that
$$\op{Ad}(g)\mathfrak{q}_{\lambda}=\mathfrak{q}_{\op{Ad}^*(g)\lambda}$$
for every $g\in G_{\mathbb{R}}$. There exist finitely many polarizations of any semisimple coadjoint $G_{\mathbb{R}}$-orbit $\mathcal{O}\subset \sqrt{-1}\mathfrak{g}_{\mathbb{R}}^*$. Among all (complex) polarizations, we consider those which are \emph{admissible}, and among admissible polarizations, we consider those which are \emph{maximally real} (See Section \ref{subsec:VZ_construction} for definitions and details). Let us fix a maximally real admissible polarization $\mathfrak{q}=\{\mathfrak{q}_{\lambda}\}_{\lambda\in \mathcal{O}}$.

Next, we fix a maximal compact subgroup $U=G^{\sigma_c}$ for which $\sigma_c\sigma=\sigma\sigma_c$. Then for each $\lambda\in \mathcal{O}$, there exists a unique $\sigma_c$-stable, Levi subalgebra $\mathfrak{l}_{\lambda,\sigma_c}\subset \mathfrak{q}_{\lambda}$. Notice $\mathfrak{l}_{\lambda,\sigma_c}$ is $G$-conjugate to $\mathfrak{g}(\lambda)$, but in most cases it is not equal to $\mathfrak{g}(\lambda)$. If $\mathfrak{h}\subset \mathfrak{l}_{\lambda,\sigma_c}$ is a Cartan subalgebra, $\Delta(\mathfrak{l}_{\lambda,\sigma_c},\mathfrak{h})$ is the collection of roots of $\mathfrak{l}_{\lambda,\sigma_c}$ with respect to $\mathfrak{h}$, and $\Delta^+\subset \Delta(\mathfrak{l}_{\lambda,\sigma_c},\mathfrak{h})$ is a choice of positive roots, then we put
$$\rho_{\mathfrak{l}_{\lambda,\sigma_c}, \mathfrak{h}, \Delta^+}=\frac{1}{2}\sum_{\alpha\in \Delta^+}\alpha\in \mathfrak{h}^*\subset \mathfrak{l}_{\lambda,\sigma_c}^*.$$
\noindent (We utilize the root space decomposition of $\mathfrak{g}$ with respect to $\mathfrak{h}$ to give inclusions $\mathfrak{h}^*\subset \mathfrak{l}^*_{\lambda,\sigma_c}\subset \mathfrak{g}^*$). Let $L_{\lambda,\sigma_c}$ denote the connected subgroup of $G$ with Lie algebra $\mathfrak{l}_{\lambda,\sigma_c}$, put $U_{\lambda}=L_{\lambda,\sigma_c}\cap U$, and put
$$\mathcal{O}_{\rho}^{U_{\lambda}}:=\bigcup_{\substack{\sigma_c(\mathfrak{h})=\mathfrak{h}\\ \Delta^+\subset\ \Delta(\mathfrak{l}_{\lambda,\sigma_c},\mathfrak{h})}} \rho_{\mathfrak{l}_{\lambda,\sigma_c},\mathfrak{h},\Delta^+}$$
where the union is over Cartan subalgebras $\mathfrak{h}\subset \mathfrak{l}_{\lambda,\sigma_c}$ that are $\sigma_c$-stable and choices of positive roots $\Delta^+\subset \Delta(\mathfrak{l}_{\lambda,\sigma_c},\mathfrak{h})$. 
Note that $\mathcal{O}_{\rho}^{U_{\lambda}}$
 is a single $U_{\lambda}$-orbit.
Define
\begin{equation}\label{eq:contour}
\mathcal{C}(\mathcal{O},\Gamma,\mathfrak{q},\sigma_c)=\bigcup_{\lambda\in \mathcal{O}} (\lambda+\mathcal{O}_{\rho}^{U_{\lambda}})\subset \mathfrak{g}^*.
\end{equation}

If $\mathcal{O}$ is in the good range, then there is a natural projection
$$\mathcal{C}(\mathcal{O},\Gamma,\mathfrak{q},\sigma_c)\rightarrow \mathcal{O}$$
and the fibers are compact and uniformly bounded in the real direction, $\mathfrak{g}_{\mathbb{R}}^*$. The fiber over $\lambda\in \mathcal{O}$ is a contour in $\mathfrak{l}_{\lambda,\sigma_c}^*$ that is conjugate via $G$ to a contour in $\mathfrak{g}(\lambda)^*$ whose Fourier transform is $\theta(\Gamma_{\lambda})$, the Lie algebra analogue of the character of $\Gamma_{\lambda}$ (See Section \ref{subsec:philosophy} for further explanation). 

Our main theorem is:
\begin{theorem} \label{thm:main} Let $G$ be a connected, complex reductive algebraic group, and let $(G^{\sigma})_e\subset G_{\mathbb{R}}\subset G^{\sigma}$ be a real form. Let $(\mathcal{O},\Gamma)$ be a semisimple orbital parameter for $G_{\mathbb{R}}$ with $\mathcal{O}$ in the good range, let $\mathfrak{q}=\{\mathfrak{q}_{\lambda}\}_{\lambda\in \mathcal{O}}$ be a maximally real admissible polarization of $\mathcal{O}$, and choose a maximal compact subgroup $U=G^{\sigma_c}\subset G$ such that $\sigma\sigma_c=\sigma_c\sigma$. Let $\theta(\mathcal{O},\Gamma)$ denote the Lie algebra analogue of the character of the irreducible, unitary representation $\pi(\mathcal{O},\Gamma)$. If the contour $\mathcal{C}(\mathcal{O},\Gamma,\mathfrak{q},\sigma_c)\subset \mathfrak{g}^*$ is defined as in (\ref{eq:contour}) and the Fourier transform $\mathcal{F}$ is defined as in (\ref{eq:Fourier_transform}), then we have
\begin{equation}\label{eq:character_formula}
\mathcal{F}[\mathcal{C}(\mathcal{O},\Gamma,\mathfrak{q},\sigma_c)]=\theta(\mathcal{O},\Gamma).
\end{equation}
\end{theorem}

We will discuss the orientation of the contour $\mathcal{C}(\mathcal{O},\Gamma,\mathfrak{q},\sigma_c)$ in Section \ref{subsec:Rossmann_characters}. Fix $\lambda\in \mathcal{O}$ for which there exists a $\sigma$ and $\sigma_c$-stable, fundamental Cartan subalgebra $\mathfrak{h}\subset \mathfrak{l}:=\mathfrak{g}(\lambda)$. Then we may more concretely define the contour 
\begin{equation}\label{eq:contourII}
\mathcal{C}(\mathcal{O},\Gamma,\mathfrak{q},\sigma_c):=\left\{g\cdot \lambda+u\cdot \rho_{\mathfrak{l}}\mid g\in G_{\mathbb{R}},\ u\in U,\ g\cdot \mathfrak{q}_{\lambda}=u\cdot \mathfrak{q}_{\lambda}\right\}
\end{equation}
where $\rho_{\mathfrak{l}}$ is half the sum of the positive roots of $\mathfrak{l}$ with respect to the Cartan subalgebra $\mathfrak{h}\subset \mathfrak{l}$ and a choice of positive roots. We note that different maximally real admissible polarizations of $\mathcal{O}$ may give different character formulas (\ref{eq:character_formula}) which are not obviously equivalent. This is analogous to the fact that different admissible polarizations of $\mathcal{O}$ may give different models of the representation $\pi(\mathcal{O},\Gamma)$ which are not obviously equivalent (See Section \ref{subsec:conjectures} for further discussion).
\bigskip

To every irreducible, admissible representation $\pi$ of $G_{\mathbb{R}}$, one associates an object in each of finitely many twisted, $G_{\mathbb{R}}$-equivariant derived categories of sheaves of complex vector spaces on the flag variety $X$ of $G$. Schmid and Vilonen give geometric character formulas for an arbitrary irreducible, admissible representation $\pi$ of $G_{\mathbb{R}}$ with regular infinitesimal character as the Fourier transform of the pushforward via Rossmann's twisted momentum map of the characteristic cycle of each object on $X$ (see (1.8) on page 4 of \cite{SV98}). When $\pi=\pi(\mathcal{O},\Gamma)$ and $\mathcal{O}$ is of maximal dimension, Schmid and Vilonen's formula \emph{does not} reduce to the classical formula of Rossmann (\ref{eq:character_tempered}). However, their contour may be obtained from the coadjoint orbit $\mathcal{O}$ by a suitable homotopy, and, in fact, Schmid and Vilonen prove their formula in this case by reducing it to the classical formula of Rossmann (\ref{eq:character_tempered}) via this homotopy (see Section 7 of \cite{SV98}). Schmid and Vilonen then prove the general case by showing that the characteristic cycle behaves well with respect to coherent continuation, parabolic induction, and certain Cayley transforms on the flag variety (see Sections 8-10 of \cite{SV98}).

Our proof of (\ref{eq:character_formula}) involves taking a formula of Schmid and Vilonen in the case $\pi=\pi(\mathcal{O},\Gamma)$ where $(\mathcal{O},\Gamma)$ is an arbitrary semisimple orbital parameter with $\mathcal{O}$ in the good range, and performing a homotopy which, in the special case where $\mathcal{O}$ is of maximal dimension, is the inverse of the one used by Schmid and Vilonen to establish their formula.
\bigskip

The reader may ask why we are not content with the more general character formula of Schmid and Vilonen. There are three reasons why we wish to write down the formula (\ref{eq:character_formula}). First, while the character formula of Schmid and Vilonen is natural from the point of view of algebraic analysis on the flag variety, the authors desire a formula that is natural from the point of view of the orbit method. Second, the formula of Schmid and Vilonen involves the characteristic cycle of a certain twisted, $G_{\mathbb{R}}$-equivariant sheaf on the flag variety. In general, this object has a complex combinatorial structure; therefore, it is meaningful to write down an explicit formula in a special case. Third, a recent series of papers \cite{HHO16}, \cite{Har}, \cite{HW} give applications of the classical formula of Rossmann (\ref{eq:character_tempered}) to abstract harmonic analysis questions and branching problems. To generalize these results, one requires a character formula for $\pi(\mathcal{O},\Gamma)$ that has certain analytic properties that are uniform in the parameter $\mathcal{O}$. Our formula (\ref{eq:character_formula}) has this property and is suitable for such applications.
\bigskip

Finally, Barbasch and Vogan proved that the first term in the asymptotic expansion of the character $\theta(\mathcal{O},\Gamma)$ is the same (up to scaling) as the first order term in the asymptotic expansion of the Fourier transform of $\mathcal{O}$ \cite{BV83}. Our formula implies this classical result and may be viewed as a further exploration of the relationship between $\theta(\mathcal{O},\Gamma)$ and $\mathcal{O}$ first observed by Barbasch and Vogan.
\bigskip

\noindent \textbf{Acknowledgements:} B. Harris would like to thank D. Vogan for introducing him to the orbit method several years ago when B. Harris was a young graduate student at MIT. The more he learns, the more B. Harris appreciates this wondrous gift. B. Harris would like to thank M. Duflo and H. He for several helpful conversations about the orbit method and character theory. In addition, B. Harris would like to thank L. Barchini, R. Zierau, and P. Trapa for helpful conversations about the deep and technical work of Schmid-Vilonen. B. Harris was supported by an AMS-Simons travel grant while this research was conducted.
Y. Oshima is supported by JSPS KAKENHI Grant No.\ 16K17562.  

\section{Semisimple Orbital Parameters and Representations of Real, Reductive Groups}
\label{sec:representations}

In this section, we define the Duflo double cover, completing the definition of a semisimple orbital parameter $(\mathcal{O},\Gamma)$, we recall the Vogan-Zuckerman construction of the Harish-Chandra module of $\pi(\mathcal{O},\Gamma)$, and we recall the Schmid-Wong construction of the maximal globalization of $\pi(\mathcal{O},\Gamma)$. This section provides notation and recalls key concepts to be used in the sequel.

\subsection{Duflo Double Covers}
\label{subsec:double_covers}

If $\mathcal{O}\subset \sqrt{-1}\mathfrak{g}_{\mathbb{R}}^*$ is a coadjoint orbit for $G_{\mathbb{R}}$, then the Kirillov-Kostant symplectic form on $\mathcal{O}$ is defined by 
\begin{equation}\label{eq:Kirillov-Kostant_form}
\omega_{\lambda}(\op{ad}^*(X)\lambda,\op{ad}^*(Y)\lambda)=\lambda([X,Y])
\end{equation}
for each $\lambda\in \mathcal{O}$. This action of $G_{\mathbb{R}}$ on $\mathcal{O}$ preserves $\omega$ (see for instance Chapter 1 of \cite{Kir04}). In particular, for any $\lambda\in \mathcal{O}$, we may form the stabilizer $G_{\mathbb{R}}(\lambda)$ of $\lambda$, and this stabilizer acts on the tangent space $T_{\lambda}\mathcal{O}$ of $\mathcal{O}$ and preserves the symplectic form on this vector space. In particular, we obtain a map
$$\psi\colon G_{\mathbb{R}}(\lambda)\rightarrow \op{Sp}(T_{\lambda}\mathcal{O},\omega_{\lambda})$$
for each $\lambda\in \mathcal{O}$ where $\op{Sp}(T_{\lambda}\mathcal{O},\omega_{\lambda})$ denotes the symplectic group of linear automorphisms of $T_{\lambda}\mathcal{O}$ that preserve $\omega_{\lambda}$. 

Recall that there exists a unique two fold cover 
$$\varpi\colon \op{Mp}(T_{\lambda}\mathcal{O},\omega_{\lambda})\rightarrow \op{Sp}(T_{\lambda}\mathcal{O},\omega_{\lambda})$$ called the metaplectic group. The existence and uniqueness of the double cover may be seen by checking that the fundamental group of the symplectic group is $\mathbb{Z}$ and therefore has a unique index two subgroup (see for instance page 173 of \cite{Fo89} for an exposition). However, for many applications it is necessary to have an explicit construction of the metaplectic group. One can use the Segal-Shale-Weil projective representation of the symplectic group or the Maslov index construction of Lion (see for instance \cite{LV80} for an exposition).

We form the \emph{Duflo double cover} of $G_{\mathbb{R}}(\lambda)$ by
\begin{equation}\label{eq:Duflo_double_cover}
\widetilde{G_{\mathbb{R}}}(\lambda):=\left\{(g,s)\in G_{\mathbb{R}}(\lambda)\times \op{Mp}(T_{\lambda}\mathcal{O},\omega_{\lambda})\Bigl\lvert \psi(g)=\varpi(s)\right\}.
\end{equation}
It was first introduced on page 106 of \cite{Duf72}. This definition completes the definition of a semisimple orbital parameter (Definition \ref{def:semisimple_orbital_parameter}).

Let $\mathcal{O}\subset \sqrt{-1}\mathfrak{g}_{\mathbb{R}}^*$ be a semisimple coadjoint orbit, fix $\lambda\in \mathcal{O}$, and let $L:=G(\lambda)$ (resp.\ $L_{\mathbb{R}}:=G_{\mathbb{R}}(\lambda))$ be the stabilizer of $\lambda$ in $G$ (resp.\ $G_{\mathbb{R}}$) with Lie algebra $\mathfrak{l}:=\mathfrak{g}(\lambda)$ (resp.\ $\mathfrak{l}_{\mathbb{R}}:=\mathfrak{g}_{\mathbb{R}}(\lambda))$. Fix a fundamental (that is, maximally compact) Cartan subalgebra $\mathfrak{h}_{\mathbb{R}}\subset \mathfrak{l}_{\mathbb{R}}$, and decompose $\mathfrak{g}$ into root spaces for $\mathfrak{h}:=\mathfrak{h}_{\mathbb{R}}\otimes_{\mathbb{R}}\mathbb{C}$
$$\mathfrak{g}=\mathfrak{h}\oplus\sum_{\alpha\in \Delta(\mathfrak{g},\mathfrak{h})} \mathfrak{g}_{\alpha}.$$
If $\mathfrak{q}_{\lambda}=\mathfrak{l}\oplus \mathfrak{n}$ is a parabolic subalgebra of $\mathfrak{g}$ with Levi factor $\mathfrak{l}$, then we may write 
$$\mathfrak{n}=\sum_{\alpha\in \Delta(\mathfrak{n},\mathfrak{h})}\mathfrak{g}_{\alpha}$$
as a sum of root spaces of $\mathfrak{g}$ with respect to $\mathfrak{h}$. If we define $\mathfrak{q}_{\op{Ad}^*(g) \lambda}:=\op{Ad}(g) \mathfrak{q}_{\lambda}$, then $\mathfrak{q}=\{\mathfrak{q}_{\lambda}\}_{\lambda\in \mathcal{O}}$ is a (complex) polarization of $\mathcal{O}$. In particular, giving a (complex) polarization of $\mathcal{O}$ is equivalent to giving a parabolic subalgebra $\mathfrak{q}_{\lambda}\subset \mathfrak{g}$ with Levi factor $\mathfrak{l}=\mathfrak{g}(\lambda)$. 

Fix a polarization $\mathfrak{q}=\{\mathfrak{q}_{\lambda}\}_{\lambda\in \mathcal{O}}$ of $\mathcal{O}$ with Levi decomposition $\mathfrak{q}_{\lambda}=\mathfrak{l}\oplus \mathfrak{n}$. Define 
\begin{equation}\label{eq:def_2rho}
2\rho(\mathfrak{n})=\sum_{\alpha\in \Delta(\mathfrak{n},\mathfrak{h})} \alpha\in \mathfrak{h}^*.
\end{equation}
Define $H:=Z_G(\mathfrak{h})\subset G$ to be the Cartan subgroup of $G$ with Lie algebra $\mathfrak{h}$, and write $e^{2\rho(\mathfrak{n})}$ for the character of $H$ whose differential is $2\rho(\mathfrak{n})$. One can also define this character as the scalar by which $H$ acts on $\bigwedge^{\text{top}}\mathfrak{n}$. Now, if $Q_{\lambda}:=N_G(\mathfrak{q}_{\lambda})$ is the normalizer of $\mathfrak{q}_{\lambda}$ in $G$, then the parabolic subgroup $Q_{\lambda}$ has Levi decomposition $Q_{\lambda}=LN$ with $N=\exp(\mathfrak{n})$. Since $N$ is the nilradical of $Q_{\lambda}$, $L$ acts on $N$ by conjugation, and $L$ acts on $\bigwedge^{\text{top}}\mathfrak{n}$. In particular, $e^{2\rho(\mathfrak{n})}$ extends to a character of $L$ and then may be restricted to a character of $L_{\mathbb{R}}:=G_{\mathbb{R}}(\lambda)$.

Define the $\rho(\mathfrak{n})$ double cover of $G_{\mathbb{R}}(\lambda)$ to be 
\begin{equation}\label{eq:rho_n_double_cover}
\widetilde{G_{\mathbb{R}}}(\lambda)^{\rho(\mathfrak{n})}=\{(g,z)\in G_{\mathbb{R}}(\lambda)\times \mathbb{C}\mid e^{2\rho(\mathfrak{n})}(g)=z^2\}.
\end{equation} 

Now, if $W\subset T_{\lambda}\mathcal{O}\otimes_{\mathbb{R}}\mathbb{C}\subset \mathfrak{g}^*$ is a complex Lagrangian subspace, denote by $\op{Sp}(T_{\lambda}\mathcal{O},\omega_{\lambda})_W$ the stabilizer of $W$ in $\op{Sp}(T_{\lambda}\mathcal{O},\omega_{\lambda})$, and denote by $\op{Mp}(T_{\lambda}\mathcal{O},\omega_{\lambda})_W$ the inverse image of $\op{Sp}(T_{\lambda}\mathcal{O},\omega_{\lambda})_W$ in the metaplectic group $\op{Mp}(T_{\lambda}\mathcal{O},\omega_{\lambda})$. In Sections 1.6-1.9 of \cite{Du82a}, Duflo defines a one-dimensional representation of $\op{Mp}(T_{\lambda}\mathcal{O},\omega_{\lambda})_W$, which we will denote by $e^{\rho_W}$, for each complex Lagrangian subspace $W$. We note that such a space is given by
$$W:=\{\op{ad}^*(X)\lambda \mid X\in \mathfrak{n}\}.$$
Since $G_{\mathbb{R}}(\lambda)$ stabilizes $W$, by (\ref{eq:Duflo_double_cover}), if $(g,s)\in \widetilde{G_{\mathbb{R}}}(\lambda)$, then $s\in \op{Mp}(T_{\lambda}\mathcal{O},\omega_{\lambda})_W$ and $e^{\rho_W}(s)$ is well-defined.  Duflo shows that if $s\in \op{Sp}(T_{\lambda}\mathcal{O},\omega_{\lambda})_W$, then (see (28) on page 149 of \cite{Du82a})
\begin{equation}\label{eq:duflo_equation}
e^{2\rho_W}(s)=\det(s|_W).
\end{equation}
Now, (\ref{eq:def_2rho}) and (\ref{eq:duflo_equation}) together imply $e^{2\rho_W}(s)=e^{2\rho(\mathfrak{n})}(g)$ whenever $\psi(g)=\varpi(s)$. Therefore, viewing (\ref{eq:Duflo_double_cover}) and (\ref{eq:rho_n_double_cover}), there is a natural map
$$\varphi\colon \widetilde{G_{\mathbb{R}}}(\lambda)\longrightarrow \widetilde{G_{\mathbb{R}}}(\lambda)^{\rho(\mathfrak{n})}$$
by
$$(g,s)\mapsto (g,e^{\rho_W}(s)).$$

Now, there exists a unique element $\epsilon\in \op{Mp}(T_{\lambda}\mathcal{O},\omega_{\lambda})$ with $\epsilon\neq e$ not the identity in $\op{Mp}(T_{\lambda}\mathcal{O},\omega_{\lambda})$ but $\varpi(\epsilon)=e$ the identity in $\op{Sp}(T_{\lambda}\mathcal{O},\omega_{\lambda})$. Duflo proves that $e^{\rho_W}(\epsilon)=-1$ (see (28) on page 149 of \cite{Du82a}), and since $\op{ker}\varpi=\{1,\epsilon\}$, we deduce that $\varphi$ is injective. Since this group homomorphism commutes with projection onto $G_{\mathbb{R}}(\lambda)$, $\varphi$ must be surjective as well. We deduce 
\begin{equation}\label{eq:double_cover}
\widetilde{G_{\mathbb{R}}}(\lambda)\simeq \widetilde{G_{\mathbb{R}}}(\lambda)^{\rho(\mathfrak{n})}
\end{equation}
and the Duflo double cover $\widetilde{G_{\mathbb{R}}}(\lambda)$ is isomorphic to the $\rho(\mathfrak{n})$ double cover of $G_{\mathbb{R}}(\lambda)$.

\subsection{Maximally Real Admissible Polarizations}
\label{subsec:parabolics}

Let $(\mathcal{O},\Gamma)$ be a semisimple orbital parameter for $G_{\mathbb{R}}$ with $(G^{\sigma})_e\subset G_{\mathbb{R}}\subset G^{\sigma}$, and fix $\lambda\in \mathcal{O}$. Let $\mathfrak{h}_{\mathbb{R}}\subset \mathfrak{g}_{\mathbb{R}}(\lambda)$ be a fundamental (maximally compact) Cartan subalgebra with complexification $\mathfrak{h}$, and let $H\subset L$ be the corresponding $\sigma$-stable Cartan subgroup. If $\mathfrak{q}=\{\mathfrak{q}_{\lambda}\}_{\lambda\in \mathcal{O}}$ is a (complex) polarization of $\mathcal{O}$, let $\mathfrak{q}_{\lambda}=\mathfrak{l}\oplus \mathfrak{n}$ be the Levi decomposition of $\mathfrak{q}_{\lambda}$ with Levi factor $\mathfrak{l}=\mathfrak{g}(\lambda)$. Whenever $\alpha\in \Delta(\mathfrak{g},\mathfrak{h})\subset \mathfrak{h}^*$, denote by $\alpha^{\vee}\in \mathfrak{h}$ the corresponding coroot. We say $\mathfrak{q}_{\lambda}$ is \emph{admissible} if 
\[\langle \lambda, \alpha^{\vee}\rangle\in \mathbb{R}_{>0}\implies \alpha\in \Delta(\mathfrak{n},\mathfrak{h}).\] Let $\sigma$ denote complex conjugation of $\mathfrak{g}$ with respect to $\mathfrak{g}_{\mathbb{R}}$. An admissible polarization is \emph{maximally real} if $\sigma(\mathfrak{q}_{\lambda})\cap \mathfrak{q}_{\lambda}$ is of maximal dimension among all admissible polarizations. 

Let us construct an example of a maximally real admissible polarization. Fix a square root of $-1$ and call it $i$. If $\gamma=a+ib$ is a complex number, define $\op{Re}\gamma:=a$ to be the real part of $\gamma$ and $\op{Im}\gamma:=b$ to be the imaginary part of $\gamma$. Define $\Delta(\mathfrak{n},\mathfrak{h})$ to be the collection of roots $\alpha\in \Delta(\mathfrak{g},\mathfrak{h})$ satisfying
$$\op{Im}\langle \lambda, \alpha^{\vee}\rangle>0$$
or
$$\op{Im}\langle \lambda,\alpha^{\vee}\rangle=0\ \text{and}\ \op{Re}\langle \lambda,\alpha^{\vee}\rangle>0.$$
Define
$$
\mathfrak{n}=\sum_{\alpha\in \Delta(\mathfrak{n},\mathfrak{h})}\mathfrak{g}_{\alpha},\quad \mathfrak{q}_{\lambda}=\mathfrak{l}\oplus \mathfrak{n}.
$$
Observe that $\mathfrak{q}_{\lambda}\subset \mathfrak{g}$ is a parabolic subalgebra of $\mathfrak{g}$ depending on the choice of $\lambda\in \mathcal{O}$. Notice $\mathfrak{q}=\{\mathfrak{q}_{\lambda}\}_{\lambda\in \mathcal{O}}$ defines an admissible polarization of $\mathcal{O}$. One checks 
$$\sigma(\mathfrak{q}_{\lambda})\cap \mathfrak{q}_{\lambda}=\mathfrak{l}\oplus \sum_{\substack{\alpha\in \Delta(\mathfrak{n},\mathfrak{h})\\ \op{Im}\langle \lambda, \alpha^{\vee}\rangle>0}}\mathfrak{g}_{\alpha};$$
therefore,
\begin{equation}\label{eq:maximally_real1}
\dim(\sigma(\mathfrak{q}_{\lambda})\cap \mathfrak{q}_{\lambda})=\dim \mathfrak{q}_{\lambda}-\#\{\alpha\in \Delta(\mathfrak{g},\mathfrak{h}) \mid \op{Im}\langle \lambda,\alpha^{\vee}\rangle=0\ \& \op{Re}\langle \lambda,\alpha^{\vee}\rangle> 0\}.
\end{equation}
Since $\sigma(\mathfrak{g}_{\alpha})=\mathfrak{g}_{\sigma(\alpha)}$ and $\langle \lambda,\sigma(\alpha)^{\vee}\rangle<0$ whenever $\langle\lambda, \alpha^{\vee} \rangle>0$, we deduce
\begin{equation}\label{eq:maximally_real2}
\dim(\sigma(\widetilde{\mathfrak{q}}_{\lambda})\cap \widetilde{\mathfrak{q}}_{\lambda})\leq \dim \widetilde{\mathfrak{q}}_{\lambda}-\#\{\alpha\in \Delta(\mathfrak{g},\mathfrak{h})\mid
 \langle \lambda,\alpha^{\vee}\rangle\in \mathbb{R}_{>0}\}
\end{equation}
whenever $\widetilde{\mathfrak{q}}_{\lambda}$ defines an admissible polarization. In particular, the admissible polarization $\mathfrak{q}=\{\mathfrak{q}_{\lambda}\}$ defined above is maximally real among all admissible polarizations.

Finally, viewing (\ref{eq:maximally_real1}) and (\ref{eq:maximally_real2}), we note that an admissible polarization $\mathfrak{q}=\mathfrak{l}\oplus \mathfrak{n}$ is maximally real among admissible polarizations if, and only if 
\begin{equation}\label{eq:maximally_real_criterion}
\alpha\in \Delta(\mathfrak{n},\mathfrak{h}) \implies \sigma\alpha\in \Delta(\mathfrak{n},\mathfrak{h})\ \text{or}\ \langle \lambda,\alpha^{\vee}\rangle \in \mathbb{R}_{>0}.
\end{equation}
\bigskip

Given a semisimple orbital parameter $(\mathcal{O},\Gamma)$ and a maximally real admissible polarization $\mathfrak{q}=\{\mathfrak{q}_{\lambda}\}$, fix $\lambda\in\mathcal{O}$ and 
fix a maximal compact subgroup $U:=G^{\sigma_c}\subset G$ defined by an anti-holomorphic involution $\sigma_c$ of $G$ that commutes with $\sigma$ and fixes $L=G(\lambda)$ and $H$. 
Then we wish to define a unitary representation $\pi=\pi(\mathcal{O},\Gamma,\mathfrak{q},\sigma_c)$ of $G_{\mathbb{R}}$ depending on the datum $(\mathcal{O},\Gamma,\mathfrak{q},\sigma_c)$. It will be easy to see that the isomorphism class of $\pi$ does not depend on $\sigma_c$. Later, we will check that it is also independent of the choice of $\mathfrak{q}$ (See Section \ref{sec:remarks}). Therefore, one may denote this representation simply as $\pi(\mathcal{O},\Gamma)$. In order to give a construction, we must utilize the datum $(\mathcal{O},\Gamma,\mathfrak{q},\sigma_c)$ to construct additional objects. We finish this section by constructing these objects. 

Define $\theta:=\sigma\sigma_c$, let $K:=G^{\theta}$, and let $K_{\mathbb{R}}:=K^{\sigma}$. The holomorphic involution $\theta$ commuting with $\sigma$ and $\sigma_c$ is called a \emph{Cartan involution}. Decompose $\mathfrak{h}=\mathfrak{h}^{\theta}\oplus \mathfrak{h}^{-\theta}$ into $+1$ and $-1$ eigenspaces of $\theta$, and let 
$$\mathfrak{h}^*=(\mathfrak{h}^{\theta})^*\oplus (\mathfrak{h}^{-\theta})^*$$
be the dual decomposition. That is, let $(\mathfrak{h}^{\theta})^*$ be the collection of linear functionals on $\mathfrak{h}$ vanishing on $\mathfrak{h}^{-\theta}$ and let $(\mathfrak{h}^{-\theta})^*$ be the collection of linear functionals on $\mathfrak{h}$ vanishing on $\mathfrak{h}^{\theta}$. We may decompose $\lambda\in \sqrt{-1}\mathfrak{h}_{\mathbb{R}}^*\subset \sqrt{-1}\mathfrak{g}_{\mathbb{R}}^*$ as $\lambda_{c}+\lambda_{n}$ in a unique way with $\lambda_{c}\in (\mathfrak{h}^{\theta})^*$ and $\lambda_{n}\in (\mathfrak{h}^{-\theta})^*$. Since $\lambda\in \sqrt{-1}\mathfrak{g}_{\mathbb{R}}^*$ and $\theta$ commutes with $\sigma$, we observe that $\lambda_{c}$ and $\lambda_{n}$ are contained in $\sqrt{-1}(\mathfrak{h}_{\mathbb{R}}^{\theta})^*$ and $\sqrt{-1}(\mathfrak{h}_{\mathbb{R}}^{-\theta})^*$ respectively. 

Let $G(\lambda_n)$ be the stabilizer of $\lambda_n$ in $G$
 with Lie algebra $\mathfrak{g}(\lambda_n)$.
Further, define $\Delta(\mathfrak{n}_{\mathfrak{p}},\mathfrak{h})\subset \Delta(\mathfrak{n},\mathfrak{h})$ to be the collection of roots $\alpha$ of $\mathfrak{n}$ with respect to $\mathfrak{h}$ for which $\langle \lambda_n,\alpha^{\vee}\rangle\neq 0$ (or equivalently $\langle \lambda,\alpha^{\vee}\rangle\notin \mathbb{R}$). In particular, using (\ref{eq:maximally_real_criterion}), we deduce 
\begin{equation}\label{eq:sigma_stable}
\alpha\in \Delta(\mathfrak{n}_{\mathfrak{p}},\mathfrak{h})\Longleftrightarrow \sigma\alpha\in \Delta(\mathfrak{n}_{\mathfrak{p}},\mathfrak{h}).
\end{equation}
Define
\[\mathfrak{n}_{\mathfrak{p}}=\sum_{\alpha\in \Delta(\mathfrak{n}_{\mathfrak{p}},\mathfrak{h})}\mathfrak{g}_{\alpha},\qquad 
\mathfrak{p}=\mathfrak{g}(\lambda_n)\oplus \mathfrak{n}_{\mathfrak{p}}.\]
We claim $\mathfrak{p}\subset \mathfrak{g}$ is a subalgebra. 
The statement
 $[\mathfrak{n}_{\mathfrak{p}},\mathfrak{n}_{\mathfrak{p}}]
 \subset \mathfrak{g}(\lambda_n)+\mathfrak{n}_{\mathfrak{p}}$
 follows directly from the definitions. 
To check
 $[\mathfrak{g}(\lambda_n),\mathfrak{n}_{\mathfrak{p}}]
 \subset \mathfrak{n}_{\mathfrak{p}}$,
 choose $\alpha\in \Delta(\mathfrak{n}_{\mathfrak{p}},\mathfrak{h})$
 and $\beta\in \Delta(\mathfrak{g}(\lambda_n),\mathfrak{h})$
 such that $\alpha+\beta\in \Delta(\mathfrak{g},\mathfrak{h})$
 and note $\langle \beta^{\vee}, \lambda_n\rangle=0$,
 $\langle \alpha^{\vee}, \lambda_n\rangle\neq 0$ together imply
\begin{equation}\label{eq:non_zero}
\langle (\alpha+\beta)^{\vee}, \lambda_n\rangle\neq 0.
\end{equation}
Further, if $\beta\in \Delta(\mathfrak{n},\mathfrak{h})$
 or $\beta\in \Delta(\mathfrak{l},\mathfrak{h})$,
 then $\alpha+\beta\in \Delta(\mathfrak{n},\mathfrak{h})$
 and combining with (\ref{eq:non_zero}),
 we obtain that
 $\alpha+\beta\in \Delta(\mathfrak{n}_{\mathfrak{p}},\mathfrak{h})$. 
Alternately, if $\beta\notin \Delta(\mathfrak{n},\mathfrak{h})
 \cup \Delta(\mathfrak{l},\mathfrak{h})$,
 then $-\beta\in \Delta(\mathfrak{n}\cap\mathfrak{g}(\lambda_n),\mathfrak{h})$;  therefore, $\langle \beta^{\vee},\lambda\rangle<0$
 which implies $\langle \sigma(\beta)^{\vee},\lambda \rangle>0$
 and we deduce
 $\sigma(\beta)\in
 \Delta(\mathfrak{n}\cap\mathfrak{g}(\lambda_n),\mathfrak{h})$. 
By repeating the previous argument, we conclude 
\[\sigma(\alpha+\beta)=\sigma(\alpha)+\sigma(\beta)
 \in \Delta(\mathfrak{n}_{\mathfrak{p}},\mathfrak{h}).\]
A second use of (\ref{eq:sigma_stable}) yields
 $\alpha+\beta\in \Delta(\mathfrak{n}_{\mathfrak{p}},\mathfrak{h})$. 
We conclude $\mathfrak{p}\subset \mathfrak{g}$ is a subalgebra;
 it is a parabolic subalgebra
 since $\mathfrak{p}\supset \mathfrak{q}_{\lambda}$. 
Moreover, by (\ref{eq:sigma_stable}) $\mathfrak{p}$ is $\sigma$-stable
 with real points $\mathfrak{p}_{\mathbb{R}}$. 

Let $P=N_G(\mathfrak{p})$
 (resp.\ $P_{\mathbb{R}}=N_{G_{\mathbb{R}}}(\mathfrak{p}_{\mathbb{R}})$)
 denote the normalizer of $\mathfrak{p}$ in $G$
 (resp.\ $\mathfrak{p}_{\mathbb{R}}$ in $G_{\mathbb{R}}$). 
Let $P_{\mathbb{R}}=M_{\mathbb{R}}A_{\mathbb{R}}(N_{P})_{\mathbb{R}}$ and
 $\mathfrak{p}_{\mathbb{R}}
 =\mathfrak{m}_{\mathbb{R}}\oplus \mathfrak{a}_{\mathbb{R}}
 \oplus (\mathfrak{n}_{\mathfrak{p}})_{\mathbb{R}}$
 be the Langlands decomposition, where
 $\mathfrak{m}_{\mathbb{R}} \oplus \mathfrak{a}_{\mathbb{R}}
 = \mathfrak{g}_{\mathbb{R}}(\lambda_n)$.
Define $M$ to be the intersection of preimages of $\{\pm 1\}$ for
 all algebraic characters $\chi: G(\lambda_n)\to \mathbb{C}^*$
 defined over $\mathbb{R}$ (namely, $\chi(\sigma(g))=\overline{\chi(g)}$).
Then $M$ is an algebraic subgroup of $G$ and
 $M_{\mathbb{R}}=G_{\mathbb{R}}\cap M$.

Observe $M$ is $\sigma$-stable, $\theta$-stable, and $\sigma_c$-stable. 
Note $\lambda_c\in \sqrt{-1}\mathfrak{m}_{\mathbb{R}}^*$, and note 
\[\mathcal{O}^{M_{\mathbb{R}}}
 :=\mathcal{O}^{M_{\mathbb{R}}}_{\lambda_c}
 =M_{\mathbb{R}}\cdot \lambda_c\subset \sqrt{-1}\mathfrak{m}_{\mathbb{R}}^*\]
 is a semisimple coadjoint orbit for $M_{\mathbb{R}}$. 
Let $\Delta(\mathfrak{m},\mathfrak{h}\cap \mathfrak{m})$
 denote the collection of roots
 of $\mathfrak{m}$ with respect to $\mathfrak{h}\cap \mathfrak{m}$, and write
\[\mathfrak{m}=(\mathfrak{h}\cap\mathfrak{m})\oplus
 \sum_{\alpha\in \Delta(\mathfrak{m},\mathfrak{h}\cap \mathfrak{m})}
 \mathfrak{g}_{\alpha}.\]
We define $\Delta(\mathfrak{n}_{\mathfrak{m}},\mathfrak{h}\cap \mathfrak{m})$
 to be the collection of
 $\alpha\in \Delta(\mathfrak{m},\mathfrak{h}\cap \mathfrak{m})$
 for which
\[\langle \lambda_{c},\alpha^{\vee}\rangle>0.\]
Then
\begin{equation}\label{eq:imaginary_parabolic}
\mathfrak{q}_{\mathfrak{m}}
=(\mathfrak{l}\cap\mathfrak{m})\oplus \mathfrak{n}_{\mathfrak{m}}
\end{equation}
is a $\theta$-stable parabolic subalgebra of $\mathfrak{m}$.

Note $\Gamma_{\lambda}$ is a genuine, one-dimensional, unitary representation of $\widetilde{G_{\mathbb{R}}}(\lambda)$. 
Moreover, $\widetilde{G_{\mathbb{R}}}(\lambda)\simeq \widetilde{L_{\mathbb{R}}}^{\rho(\mathfrak{n})}$ by (\ref{eq:double_cover}). 
We need to construct a genuine, one-dimensional,
 unitary representation $\Gamma^{M_{\mathbb{R}}}_{\lambda_c}$
 of $\widetilde{M_{\mathbb{R}}}(\lambda_c)$.  Define
\[2\rho(\mathfrak{n}_{\mathfrak{m}}):=
\sum_{\alpha\in \Delta(\mathfrak{n}_{\mathfrak{m}},
 \mathfrak{h}\cap \mathfrak{m})} \alpha,
\qquad 2\rho(\mathfrak{n}_{\mathfrak{p}})
:=\sum_{\alpha\in \Delta(\mathfrak{n}_{\mathfrak{p}},\mathfrak{a})}\alpha.\]
Again by (\ref{eq:double_cover}),
 $\widetilde{M_{\mathbb{R}}}(\lambda_c)\simeq
 \widetilde{(L_{\mathbb{R}}\cap M_{\mathbb{R}})}^{
  \rho(\mathfrak{n}_{\mathfrak{m}})}$.
Observe 
\[2\rho(\mathfrak{n})
=2\rho(\mathfrak{n}_{\mathfrak{m}})+2\rho(\mathfrak{n}_{\mathfrak{p}}),\]
and note $e^{2\rho(\mathfrak{n}_{\mathfrak{p}})}$
 is a character of $L_{\mathbb{R}}$ taking values in $\mathbb{R}^{\times}$
 since $\mathfrak{p}$ is $\sigma$-stable. 
Thus, if $z\in \mathbb{C}$ and
 $z^2=e^{2\rho(\mathfrak{n}_{\mathfrak{p}})}(g)\in \mathbb{R}^{\times}$
 for some $g\in L_{\mathbb{R}}$, then $z/|z|\in \{\pm 1, \pm \sqrt{-1}\}$. 
Now, if $\Gamma_{\lambda}$ is a one-dimensional, genuine,
 unitary representation of $\widetilde{L_{\mathbb{R}}}^{\rho(\mathfrak{n})}$,
 then we define
\begin{equation}\label{eq:character_transfer}
\Gamma^{M_{\mathbb{R}}}_{\lambda_c}(g,z)
 :=\frac{zz_1^{-1}}{|zz_1^{-1}|}\Gamma_{\lambda}(g,z_1)
\end{equation}
for $(g,z_1)\in \widetilde{L_{\mathbb{R}}}^{\rho(\mathfrak{n})}$. 
One notes that (\ref{eq:character_transfer}) is independent of the choice of
 $z_1$ so long as $(g,z_1)\in \widetilde{L_{\mathbb{R}}}^{\rho(\mathfrak{n})}$. 
Therefore, $\Gamma^{M_{\mathbb{R}}}_{\lambda_c}$ is a well-defined genuine,
 one-dimensional, unitary representation of
 $\widetilde{(L_{\mathbb{R}}\cap M_{\mathbb{R}})}^{
 \rho(\mathfrak{n}_{\mathfrak{m}})}$. 
One may translate $\Gamma_{\lambda_c}^{M_{\mathbb{R}}}$ by elements of
 $M_{\mathbb{R}}$ to create a semisimple orbital parameter
 $(\mathcal{O}^{M_{\mathbb{R}}},\Gamma^{M_{\mathbb{R}}})$ which satisfies
 parts (a) and (b) of Definition \ref{def:semisimple_orbital_parameter}.

An element $\lambda\in \sqrt{-1}\mathfrak{g}_{\mathbb{R}}^*$
 is \emph{elliptic} if there exists a Cartan involution $\theta$
 for which $\theta(\lambda)=\lambda$. 
A coadjoint orbit $\mathcal{O}\subset \sqrt{-1}\mathfrak{g}_{\mathbb{R}}^*$
 is \emph{elliptic} if it consists of elliptic elements. 
A semisimple orbital parameter $(\mathcal{O},\Gamma)$ for $G_{\mathbb{R}}$ is
 called an \emph{elliptic orbital parameter}
 if $\mathcal{O}\subset \sqrt{-1}\mathfrak{g}_{\mathbb{R}}^*$
 is an elliptic coadjoint orbit. 
If $(\mathcal{O},\Gamma)$ is a semisimple orbital parameter for
 $G_{\mathbb{R}}$, one observes
 $(\mathcal{O}^{M_{\mathbb{R}}}, \Gamma^{M_{\mathbb{R}}})$
 is an elliptic orbital parameter for $M_{\mathbb{R}}$.

\subsection{The Vogan-Zuckerman Construction of Representations for Elliptic Orbital Parameters}
\label{subsec:VZ_construction}

In this subsection, we recall the construction of the unitary representation for elliptic orbital parameters.  We do this by utilizing the cohomological induction construction of Vogan and Zuckerman, and we deal with the possibility of disconnected groups. 
Later, we will apply this to $\pi(\mathcal{O}^{M_{\mathbb{R}}},\Gamma^{M_{\mathbb{R}}})$ of $M_{\mathbb{R}}$ defined in the previous subsection. Our discussion follows Knapp-Vogan \cite{KV95}.

A \emph{pair} $(\mathfrak{g},K_{\mathbb{R}})$ is a finite-dimensional,
 complex Lie algebra $\mathfrak{g}$ and a compact Lie group $K_{\mathbb{R}}$ such that
\begin{enumerate}[(i)]
\item The complexified Lie algebra $\mathfrak{k}:=\mathfrak{k}_{\mathbb{R}}\otimes_{\mathbb{R}}\mathbb{C}$ of $K_{\mathbb{R}}$ is a subalgebra of $\mathfrak{g}$.
\item There is a fixed action of $K_{\mathbb{R}}$ on $\mathfrak{g}$ which, when restricted to $\mathfrak{k}$, yields the adjoint action of $K_{\mathbb{R}}$ on $\mathfrak{k}$.
\item The differential of the fixed action of $K_{\mathbb{R}}$ on $\mathfrak{g}$ yields the adjoint action of $\mathfrak{k}_{\mathbb{R}}$ on $\mathfrak{g}$.
\end{enumerate}

Now, if $(\mathfrak{g},K_{\mathbb{R}})$ is a pair, then a {\it $(\mathfrak{g},K_{\mathbb{R}})$-module} is a complex vector space $V$ with a group action of $K_{\mathbb{R}}$ and a Lie algebra action of $\mathfrak{g}$ satisfying

\begin{enumerate}[(i)]
\item The representation of $K_{\mathbb{R}}$ on $V$ is locally $K_{\mathbb{R}}$-finite. 
\item The differential of the $K_{\mathbb{R}}$ action on $V$ is the restriction of the action of $\mathfrak{g}$ on $V$ to an action of $\mathfrak{k}_{\mathbb{R}}$ on $V$.
\item We have $(\op{Ad}(k)X)v=k(X(k^{-1}v))$ for $k\in K_{\mathbb{R}}$, $X\in \mathfrak{g}$, and $v\in V$.
\end{enumerate}

Suppose $G$ is a connected, complex algebraic group with Lie algebra $\mathfrak{g}$, $G_{\mathbb{R}}$ is a real form of $G$, and $K_{\mathbb{R}}=G_{\mathbb{R}}^{\theta}$ is a maximal compact subgroup of $G_{\mathbb{R}}$ defined by a Cartan involution $\theta$. 
Then $(\mathfrak{g},K_{\mathbb{R}})$ is naturally a pair. 
A $(\mathfrak{g},K_{\mathbb{R}})$ module $V$ is a \emph{Harish-Chandra module} if $V$ is finitely generated over $\mathcal{U}(\mathfrak{g})$, the universal enveloping algebra of $\mathfrak{g}$, and every irreducible representation of $K_{\mathbb{R}}$ has at most finite multiplicity in $V$. 

If $(\mathfrak{g},K_{\mathbb{R}})$ is a pair, then one may form the Hecke algebra $R(\mathfrak{g},K_{\mathbb{R}})$ of this pair. In the setting of the previous example, one may set $R(\mathfrak{g},K_{\mathbb{R}})$ equal to the algebra of bi-$K_{\mathbb{R}}$-finite distributions on $G_{\mathbb{R}}$ that are supported on $K_{\mathbb{R}}$ under the operation convolution. Let $R(K_{\mathbb{R}})$ denotes the algebra of bi-$K_{\mathbb{R}}$-finite distributions on $K_{\mathbb{R}}$. Then we may obtain an
element of $R(\mathfrak{g},K_{\mathbb{R}})$ by applying a left invariant differential operator corresponding to an element of $\mathcal{U}(\mathfrak{g})$, restricting to $K_{\mathbb{R}}$, and then pairing with an element of $R(K_{\mathbb{R}})$. In fact, every element of $R(\mathfrak{g},K_{\mathbb{R}})$ may be written as a finite sum of such distributions and one has the isomorphism of vector spaces
\[R(\mathfrak{g},K_{\mathbb{R}})\simeq R(K_{\mathbb{R}})\otimes_{\mathcal{U}(\mathfrak{k})}\mathcal{U}(\mathfrak{g}).\]

Motivated by this special case, one can define an algebra structure on the vector space $R(K_{\mathbb{R}})\otimes_{\mathcal{U}(\mathfrak{k})}\mathcal{U}(\mathfrak{g})$ for any pair $(\mathfrak{g},K_{\mathbb{R}})$. This algebra is then written $R(\mathfrak{g},K_{\mathbb{R}})$ and called the Hecke algebra of the pair (See Sections I.5, I.6 of \cite{KV95}). For every pair $(\mathfrak{g},K_{\mathbb{R}})$, one obtains an equivalence of categories between the category of $(\mathfrak{g},K_{\mathbb{R}})$-modules and the category of approximately unital left $R(\mathfrak{g},K_{\mathbb{R}})$-modules (See Theorem 1.117 on page 90 of \cite{KV95}).

If $(\mathfrak{h},B_{\mathbb{R}})$ and $(\mathfrak{g},K_{\mathbb{R}})$ are two pairs, then we write
\[
\iota_{\mathfrak{b}_{\mathbb{R}}}\colon \mathfrak{b}_{\mathbb{R}}\hookrightarrow \mathfrak{h},
\quad 
\iota_{\mathfrak{k}_{\mathbb{R}}}\colon \mathfrak{k}_{\mathbb{R}}\hookrightarrow \mathfrak{g}
\]
for the two inclusions. A map of pairs $\phi\colon (\mathfrak{h},B_{\mathbb{R}})\rightarrow (\mathfrak{g},K_{\mathbb{R}})$ is a pair of maps
\[
\phi_{\text{alg}}\colon \mathfrak{h}\rightarrow \mathfrak{g},\quad
 \phi_{\text{gp}}\colon B_{\mathbb{R}}\rightarrow K_{\mathbb{R}}
\]
satisfying 
\begin{enumerate}[(i)]
\item $\phi_{\text{alg}}\circ \iota_{\mathfrak{b}_{\mathbb{R}}}=\iota_{\mathfrak{k}_{\mathbb{R}}}\circ d\phi_{\text{gp}}$
\item $\phi_{\text{alg}}\circ \op{Ad}_{B_{\mathbb{R}}}(b)=\op{Ad}_{K_{\mathbb{R}}}(\phi_{\text{gp}}(b))\circ \phi_{\text{alg}}$ for all $b\in B_{\mathbb{R}}$.
\end{enumerate}

Given a map of pairs $\phi\colon (\mathfrak{h},B_{\mathbb{R}})\rightarrow (\mathfrak{g},K_{\mathbb{R}})$, Zuckerman, Bernstein, and others defined functors from the category of $R(\mathfrak{h},B_{\mathbb{R}})$-modules to the category of $R(\mathfrak{g},K_{\mathbb{R}})$-modules. A strategy to find a non-trivial, interesting $R(\mathfrak{g},K_{\mathbb{R}})$-module is to take a trivial $R(\mathfrak{h},B_{\mathbb{R}})$-module and apply a suitable functor.

If $\phi\colon (\mathfrak{h},B_{\mathbb{R}})\rightarrow (\mathfrak{g},K_{\mathbb{R}})$ is a map of pairs, then $\phi$ gives $R(\mathfrak{g},K_{\mathbb{R}})$ the structure of an approximately unital right $R(\mathfrak{h},B_{\mathbb{R}})$-module (See page 104 of \cite{KV95}). We can then define the functor
$$V\mapsto P_{\mathfrak{h},B_{\mathbb{R}}}^{\mathfrak{g},K_{\mathbb{R}}}(V):=R(\mathfrak{g},K_{\mathbb{R}})\otimes_{R(\mathfrak{h},B_{\mathbb{R}})} V$$
from the category of approximately unital, left $R(\mathfrak{h},B_{\mathbb{R}})$-modules to the category of approximately unital, left $R(\mathfrak{g},K_{\mathbb{R}})$-modules. In addition, we have the functor
\[V\mapsto I_{\mathfrak{h},B_{\mathbb{R}}}^{\mathfrak{g},K_{\mathbb{R}}}(V)
:=\op{Hom}_{R(\mathfrak{h},B_{\mathbb{R}})}(R(\mathfrak{g},K_{\mathbb{R}}),V)_{K_{\mathbb{R}}}\]
from the category of approximately unital, left $R(\mathfrak{h},B_{\mathbb{R}})$-modules to the category of approximately unital, left $R(\mathfrak{g},K_{\mathbb{R}})$-modules where the subscript $K_{\mathbb{R}}$ means we take the submodule of $K_{\mathbb{R}}$-finite vectors. The functor $I$ is a generalization of a functor of Zuckerman and the functor $P$ is a generalization of a functor of Bernstein, which was inspired by Zuckerman's functor (See Section I.1 of \cite{KV95}).  

If $(\mathfrak{g},K_{\mathbb{R}})$ is a pair, then the category of approximately unital, left $R(\mathfrak{g},K_{\mathbb{R}})$-modules has enough projectives and injectives (See Corollary 2.26 and Corollary 2.37 on page 115 of \cite{KV95}). The functor $P_{\mathfrak{h},B_{\mathbb{R}}}^{\mathfrak{g},K_{\mathbb{R}}}$ is right exact and the functor $I_{\mathfrak{h},B_{\mathbb{R}}}^{\mathfrak{g},K_{\mathbb{R}}}$ is left exact; therefore, we have a sequence of left derived functors $(P_{\mathfrak{h},B_{\mathbb{R}}}^{\mathfrak{g},K_{\mathbb{R}}})_p$ and a sequence of right derived functors $(I_{\mathfrak{h},B_{\mathbb{R}}}^{\mathfrak{g},K_{\mathbb{R}}})^p$ for $p=0,1,2,\ldots$ (See Sections II.2 and II.6 of \cite{KV95}). If $V$ is a $(\mathfrak{h},B_{\mathbb{R}})$-module, then it has the structure of an approximately unital, left $R(\mathfrak{h},B_{\mathbb{R}})$-module, and we may apply the functors $(P_{\mathfrak{h},B_{\mathbb{R}}}^{\mathfrak{g},K_{\mathbb{R}}})_p(V)$ and $(I_{\mathfrak{h},B_{\mathbb{R}}}^{\mathfrak{g},K_{\mathbb{R}}})^p(V)$ to obtain approximately unital, left $R(\mathfrak{g},K_{\mathbb{R}})$-modules. These, in turn, may be viewed as $(\mathfrak{g},K_{\mathbb{R}})$-modules.
\bigskip

Next, we construct a unitary representation for an elliptic orbital parameter
 $(\mathcal{O},\Gamma)$ utilizing the functor above.
We fix $\lambda\in \mathcal{O}$ such that $\theta(\lambda)=\lambda$.
For the elliptic case, an admissible polarization
 $\mathfrak{q}_{\lambda}=\mathfrak{l}+\mathfrak{n}$ is uniquely determined.
We let $\mathfrak{q}=\mathfrak{q}_{\lambda}$ in this subsection to simplify the notation.


Observe $e^{\rho(\mathfrak{n})}$
 and $\Gamma_{\lambda}$ are genuine, one-dimensional unitary representation of
 $\widetilde{L_{\mathbb{R}}}^{\rho(\mathfrak{n})}$ ($e^{\rho(\mathfrak{n})}$ is unitary
 since all of the roots in
 $\Delta(\mathfrak{n},\mathfrak{h})$ are imaginary). 
Hence, their tensor product
 $\Gamma_{\lambda}\otimes
 e^{\rho(\mathfrak{n})}$ descends to a one-dimensional
 unitary representation of $L_{\mathbb{R}}$. 
Differentiating and complexifying the action of $L_{\mathbb{R}}$,
 we may view $\Gamma_{\lambda}\otimes e^{\rho(\mathfrak{n})}$ as a one-dimensional
 $(\mathfrak{l}, L_{\mathbb{R}}\cap K_{\mathbb{R}})$-module. 
Observe
 $\overline{\mathfrak{q}}:=\sigma(\mathfrak{q})$
 is the opposite parabolic to $\mathfrak{q}$
 with nilradical $\overline{\mathfrak{n}}
 :=\sigma(\mathfrak{n})$.
Letting $\overline{\mathfrak{n}}$ act trivially,
 we may view $\Gamma_{\lambda} \otimes e^{\rho(\mathfrak{n})}$
 as a one-dimensional
 $(\overline{\mathfrak{q}}, L_{\mathbb{R}}\cap K_{\mathbb{R}})$-module.

Then we may form the
 $(\mathfrak{g}, K_{\mathbb{R}})$-modules
\[(I_{\mathfrak{q}, L_{\mathbb{R}}\cap K_{\mathbb{R}}}^{\mathfrak{g}, K_{\mathbb{R}}})^p
 (\Gamma_{\lambda} \otimes e^{\rho(\mathfrak{n})}),\quad 
 (P_{\overline{\mathfrak{q}}, L_{\mathbb{R}}\cap K_{\mathbb{R}}}^{\mathfrak{g},
 K_{\mathbb{R}}})_p
(\Gamma_{\lambda}\otimes e^{\rho(\mathfrak{n})}).\]
By Theorem 5.35 of \cite{KV95}, these vanish when
 $p>s:=\dim_{\mathbb{C}}(\mathfrak{n}\cap \mathfrak{k})$. Since we assumed $\mathfrak{q}$ is admissible, by Theorem 5.109 of \cite{KV95}, they vanish when $p<s$ and they are isomorphic as $(\mathfrak{g},K_{\mathbb{R}})$-modules when $p=s$.

Let $(G_{\mathbb{R}})_e\subset G_{\mathbb{R}}$ (resp.\ $(K_{\mathbb{R}})_e\subset K_{\mathbb{R}}$)
 denote the identity component of $G_{\mathbb{R}}$ (resp.\ $K_\mathbb{R}$). 
For technical reasons, we will work with representations of the connected group $(G_{\mathbb{R}})_e$ rather than the potentially disconnected group $G_{\mathbb{R}}$ in the sequel. 
Therefore, we begin by considering the $(\mathfrak{g}, (K_{\mathbb{R}})_e)$-module 
\[ (I_{\mathfrak{q},L_{\mathbb{R}}\cap (K_{\mathbb{R}})_e}^{\mathfrak{g},(K_{\mathbb{R}})_e})^s
 ((\Gamma_{\lambda}\otimes e^{\rho(\mathfrak{n})})|_{L_{\mathbb{R}}\cap (G_{\mathbb{R}})_e})\]
where $s=\dim_{\mathbb{C}}(\mathfrak{n}\cap \mathfrak{k})$. 
This $(\mathfrak{g},(K_{\mathbb{R}})_e)$-module is isomorphic to 
\[(P_{\overline{\mathfrak{q}},L_{\mathbb{R}}\cap (K_{\mathbb{R}})_e}^{
  \mathfrak{g},(K_{\mathbb{R}})_e})_s
 ((\Gamma_{\lambda}\otimes e^{\rho(\mathfrak{n})})|_{L_{\mathbb{R}}\cap (G_{\mathbb{R}})_e}).\]

Let $\mathcal{ZU}(\mathfrak{g})$ denote the center of
 the universal enveloping algebra of $\mathfrak{g}$. 
Then $\mathcal{ZU}(\mathfrak{g})$ acts by a scalar on this module
 (see Theorem 5.25 of \cite{KV95}); in the language of Harish-Chandra,
 it is a \emph{quasi-simple} module. 
Choose a system of positive roots $\Delta^+(\mathfrak{l},\mathfrak{h})$
 of $\mathfrak{l}$ with respect to $\mathfrak{h}$ and define
\[\rho_{\mathfrak{l}}=\frac{1}{2}\sum_{\alpha\in \Delta^+(\mathfrak{l},\mathfrak{h})} \alpha.\]
We say $\lambda$ (resp.\ $\mathcal{O}^{(G_{\mathbb{R}})_e}:=(G_{\mathbb{R}})_e\cdot \lambda$,
 $\mathcal{O}$) 
 is in the \emph{good range} for $(G_{\mathbb{R}})_e$ or $G_{\mathbb{R}}$
 (resp.\ $(G_{\mathbb{R}})_e$, $G_{\mathbb{R}}$) if
\begin{equation}\label{eq:good_range}
\op{Re}\langle \lambda+\rho_{\mathfrak{l}}, \alpha^{\vee}\rangle>0
\end{equation}
for all $\alpha\in \Delta(\mathfrak{n},\mathfrak{h})$.
 If $\mathcal{O}^{(G_{\mathbb{R}})_e}$ is in the good range, then 
\[(I_{\mathfrak{q},L_{\mathbb{R}}\cap (K_{\mathbb{R}})_e}^{\mathfrak{g},(K_{\mathbb{R}})_e})^s
 ((\Gamma_{\lambda}\otimes e^{\rho(\mathfrak{n})})|_{L_{\mathbb{R}}\cap (G_{\mathbb{R}})_e})\]
is an irreducible $(\mathfrak{g},(K_{\mathbb{R}})_e)$-module
 (see Zuckerman (unpublished) and Vogan \cite{Vog81} for the original references;
 see Theorem 8.2 of \cite{KV95} for an exposition). 
Since it is an irreducible, quasi-simple $(\mathfrak{g},(K_{\mathbb{R}})_e)$-module,
 by a result of Harish-Chandra (Theorem 4 on page 63 of \cite{HC54a}),
 it is the collection of $(K_{\mathbb{R}})_e$-finite vectors of a (possibly non-unitary)
 continuous representation of $(G_{\mathbb{R}})_e$ on a Hilbert space. 
Vogan proved that the $(\mathfrak{g},(K_{\mathbb{R}})_e)$-module 
\[(I_{\mathfrak{q},L_{\mathbb{R}}\cap (K_{\mathbb{R}})_e}^{\mathfrak{g},(K_{\mathbb{R}})_e})^s
 ((\Gamma_{\lambda}\otimes e^{\rho(\mathfrak{n})})|_{L_{\mathbb{R}}\cap (G_{\mathbb{R}})_e})\]
has an invariant inner product \cite{Vog84} (see Theorem 9.1 on page 598 of \cite{KV95} for an exposition). 
A second result of Harish-Chandra implies that this representation is the collection of $(K_{\mathbb{R}})_e$-finite vectors of a unitary representation of $(G_{\mathbb{R}})_e$ (Theorem 9 on page 233 of \cite{HC53}). We denote this unitary representation by $\pi(\mathcal{O}^{(G_{\mathbb{R}})_e},\Gamma^{(G_{\mathbb{R}})_e})$ where $\mathcal{O}^{(G_{\mathbb{R}})_e}:=(G_{\mathbb{R}})_e\cdot \lambda$ and $\Gamma^{(G_{\mathbb{R}})_e}:=\Gamma|_{\mathcal{O}^{(G_{\mathbb{R}})_e}}$.
\bigskip

In order to define $\pi(\mathcal{O},\Gamma)$, a unitary representation of $G_{\mathbb{R}}$,
 we must deal with several subtleties. 
We define $G_{\mathbb{R}}^{\lambda}$
 to be the subgroup of $G_{\mathbb{R}}$ generated by
 $(G_{\mathbb{R}})_e$ and $L_{\mathbb{R}}$,
 and define $K_{\mathbb{R}}^{\lambda}:=G_{\mathbb{R}}^{\lambda}\cap K_{\mathbb{R}}$.
We note
 $(G_{\mathbb{R}})_e\subset G_{\mathbb{R}}^{\lambda}\subset G_{\mathbb{R}}$. 
Let $V$ be the Hilbert space on which
 $\pi(\mathcal{O}^{(G_{\mathbb{R}})_e},\Gamma^{(G_{\mathbb{R}})_e})$ acts. 
The collection of $(K_{\mathbb{R}})_e$-finite vectors, $V_{(K_{\mathbb{R}})_e}$, is given by
\[(I_{\mathfrak{q},L_{\mathbb{R}}\cap (K_{\mathbb{R}})_e}^{\mathfrak{g},(K_{\mathbb{R}})_e})^s
((\Gamma_{\lambda}\otimes e^{\rho(\mathfrak{n})})|_{L_{\mathbb{R}}\cap (G_{\mathbb{R}})_e})\]
as an $(\mathfrak{g},(K_{\mathbb{R}})_e)$-module. 
Recall that one can compute the above derived functor by applying the
 $\op{Hom}_{R(\mathfrak{q},L_{\mathbb{R}}\cap (K_{\mathbb{R}})_e)}(R(\mathfrak{g}, (K_{\mathbb{R}})_e),\cdot)_{(K_{\mathbb{R}})_e}$ functor to the standard resolution of $(\Gamma_{\lambda}\otimes e^{\rho(\mathfrak{n})})|_{L_{\mathbb{R}}\cap (G_{\mathbb{R}})_e}$ and taking cohomology (See Section II.7 of \cite{KV95}). 
Analogously, we may explicitly write down the $(\mathfrak{g},K_{\mathbb{R}}^{\lambda})$-module 
\[(I_{\mathfrak{q},L_{\mathbb{R}}\cap K_{\mathbb{R}}}^{\mathfrak{g},K_{\mathbb{R}}^{\lambda}})^s
 (\Gamma_{\lambda}\otimes e^{\rho(\mathfrak{n})})\]
by applying the
 $\op{Hom}_{R(\mathfrak{q},L_{\mathbb{R}}\cap K_{\mathbb{R}}^\lambda)}
 (R(\mathfrak{g},K_{\mathbb{R}}^{\lambda}), \cdot)_{K_{\mathbb{R}}^{\lambda}}$ functor
 to the standard resolution of $\Gamma_{\lambda}\otimes e^{\rho(\mathfrak{n})}$ and taking cohomology. One notes that these standard resolutions are identical except that the additional action of $(K_{\mathbb{R}})_e$ is extended to an action of $K_{\mathbb{R}}^{\lambda}$ in the latter case. Analogous remarks hold for the $P$ functor. We wish to extend this action of $K_{\mathbb{R}}^{\lambda}$ on $V_{(K_{\mathbb{R}})_e}$ to an action of $G_{\mathbb{R}}^{\lambda}$ on $V$.

First, recall that we have the decomposition $G_{\mathbb{R}}^{\lambda}=K_{\mathbb{R}}^{\lambda}(G_{\mathbb{R}})_e$. This follows from the decomposition $G_{\mathbb{R}}= K_{\mathbb{R}}(G_{\mathbb{R}})_e$ which follow from equations (12.74) and (12.75) on page 468 of \cite{Kna86}. We wish to define an action $\pi$ of $G_{\mathbb{R}}^{\lambda}$ on $V$ by 
$\pi(kg)=\pi(k)\pi(g)$ for $k\in K_{\mathbb{R}}^{\lambda}$ and $g\in (G_{\mathbb{R}})_e$. Since $V$ is a unitary representation of $(G_{\mathbb{R}})_e$, $\pi(g)$ is well-defined. Note $\pi(k)$ is well-defined on $V_{(K_{\mathbb{R}})_e}$ by the remarks in the above paragraph. Moreover, the inner product on $V_{(K_{\mathbb{R}})_e}$ used to complete this space to $V$ is $K_{\mathbb{R}}^{\lambda}$-invariant (Theorem 9.1 of \cite{KV95}); hence, this action commutes with limits and is well-defined on $V$. Moreover, the action of $(K_{\mathbb{R}})_e$ coincides with the restriction of the actions of $K_{\mathbb{R}}^\lambda$ and $(G_{\mathbb{R}})_e$ on $V$; hence, $\pi$ is a well-defined, continuous map
\[G_{\mathbb{R}}^{\lambda}\times V\rightarrow V.\]
We must check that it is a group action. Writing out the algebraic expression for a group action, one sees that it is enough to check 
\[\pi(kgk^{-1})=\pi(k)\pi(g)\pi(k^{-1})\]
for $g\in (G_{\mathbb{R}})_e$ and $k\in K_{\mathbb{R}}^{\lambda}$. Since $(G_{\mathbb{R}})_e$ is generated by the image of the exponential map, it is enough to check
\begin{equation}\label{eq:KM_identity}
\pi(\exp(\op{Ad}(k)X))=\pi(k)\pi(\exp X)\pi(k^{-1})
\end{equation}
for $X\in \mathfrak{g}_{\mathbb{R}}$ and $k\in K_{\mathbb{R}}^{\lambda}$. Since all of these operators are unitary and therefore commute with limits, it is enough to check this identity when applied to $v\in V_{(K_{\mathbb{R}})_e}$. Harish-Chandra proved that every $v\in V_{(K_{\mathbb{R}})_e}$ is an analytic vector for $V$, that is, the map $(G_{\mathbb{R}})_e\rightarrow V$ by $g\mapsto \pi(g)v$ is analytic \cite{HC53}. In particular, one observes $\pi(\exp(X))=\exp(\pi(X))$ for $X\in \mathfrak{g}_{\mathbb{R}}$. Expanding out both sides of (\ref{eq:KM_identity}) into power series and utilizing part (iii) of the definition of an $(\mathfrak{g},K_{\mathbb{R}}^{\lambda})$-module term by term, one verifies (\ref{eq:KM_identity}). Hence, we have checked that we may extend the action of $(G_{\mathbb{R}})_e$ on $V$ to obtain a unitary representation of $G_{\mathbb{R}}^{\lambda}$ on $V$. We call this second unitary representation $\pi(\mathcal{O}^{G_{\mathbb{R}}^{\lambda}},\Gamma^{G_{\mathbb{R}}^{\lambda}})$, and we note
\begin{equation}\label{eq:M_extension_I}
\pi(\mathcal{O}^{G_{\mathbb{R}}^{\lambda}},\Gamma^{G_{\mathbb{R}}^{\lambda}})|_{(G_{\mathbb{R}})_e}\simeq \pi(\mathcal{O}^{(G_{\mathbb{R}})_e},\Gamma^{(G_{\mathbb{R}})_e}).
\end{equation}

Finally, we must obtain the unitary representation $\pi(\mathcal{O},\Gamma)$. 
We wish for the Harish-Chandra module of this representation to be 
\[(I_{\mathfrak{q},L_{\mathbb{R}}\cap K_{\mathbb{R}}}^{\mathfrak{g}, K_{\mathbb{R}}})^s
 (\Gamma_{\lambda}\otimes e^{\rho(\mathfrak{n})})\]
for $s=\dim(\mathfrak{n}\cap \mathfrak{k})$. 
Utilizing Proposition 5.150 of \cite{KV95}, we observe
\[(I_{\mathfrak{q},L_{\mathbb{R}}\cap K_{\mathbb{R}}}^{\mathfrak{g},K_{\mathbb{R}}})^s
 (\Gamma_{\lambda}\otimes e^{\rho(\mathfrak{n})})
\simeq (I_{\mathfrak{g},K_{\mathbb{R}}^{\lambda}}^{\mathfrak{g},K_{\mathbb{R}}})^0
 (I_{\mathfrak{q},L_{\mathbb{R}}\cap K_{\mathbb{R}}}^{\mathfrak{g},K_{\mathbb{R}}^{\lambda}})^s
 (\Gamma_{\lambda}\otimes e^{\rho(\mathfrak{n})}).\]
Moreover, by Proposition 2.77 on page 136 of \cite{KV95}, the functor
 $(I_{\mathfrak{g},K_{\mathbb{R}}^{\lambda}}^{\mathfrak{g}, K_{\mathbb{R}}})^0$
 is the classical induction functor from a subgroup of finite index. Therefore, 
\begin{equation}\label{eq:M_induction_HC_module}
(I_{\mathfrak{q},L_{\mathbb{R}}\cap K_{\mathbb{R}}}^{\mathfrak{g}, K_{\mathbb{R}}})^s
 (\Gamma_{\lambda}\otimes e^{\rho(\mathfrak{n})})|_{G_{\mathbb{R}}^{\lambda}}
\simeq \bigoplus_{w\in G_{\mathbb{R}}/G_{\mathbb{R}}^{\lambda}} w\cdot (I_{\mathfrak{q},L_{\mathbb{R}}\cap K_{\mathbb{R}}}^{\mathfrak{g},K_{\mathbb{R}}^{\lambda}})^s
 (\Gamma_{\lambda}\otimes e^{\rho(\mathfrak{n})}).
\end{equation}
Mimicking this construction globally, we consider the vector bundle 
\[G_{\mathbb{R}}\times_{G_{\mathbb{R}}^{\lambda}} \pi(\mathcal{O}^{G_{\mathbb{R}}^{\lambda}},\Gamma^{G_{\mathbb{R}}^{\lambda}})\rightarrow G_{\mathbb{R}}/G_{\mathbb{R}}^{\lambda}\]
and we define the unitary representation $\pi(\mathcal{O},\Gamma)$ to be the space of sections
\[\Gamma(G_{\mathbb{R}}/G_{\mathbb{R}}^{\lambda}, G_{\mathbb{R}}\times_{G_{\mathbb{R}}^{\lambda}} \pi(\mathcal{O}^{G_{\mathbb{R}}^{\lambda}},\Gamma^{G_{\mathbb{R}}^{\lambda}})).\]
One observes that the collection of $K_{\mathbb{R}}$-finite vectors of $\pi(\mathcal{O},\Gamma)$ is isomorphic to $(I_{\mathfrak{q},L_{\mathbb{R}}\cap K_{\mathbb{R}}}^{\mathfrak{g}, K_{\mathbb{R}}})^s(\Gamma_{\lambda}\otimes e^{\rho(\mathfrak{n})})$ as a $(\mathfrak{g}, K_{\mathbb{R}})$-module and 
\begin{equation}\label{eq:M_extension_II}
\pi(\mathcal{O},\Gamma)|_{G_{\mathbb{R}}^{\lambda}}\simeq \bigoplus_{w\in G_{\mathbb{R}}/G_{\mathbb{R}}^{\lambda}} w\cdot \pi(\mathcal{O}^{G_{\mathbb{R}}^{\lambda}},\Gamma^{G_{\mathbb{R}}^{\lambda}})
\end{equation}
as a unitary representation of $G_{\mathbb{R}}^{\lambda}$. Combining (\ref{eq:M_extension_I}) and (\ref{eq:M_extension_II}), we obtain
\begin{equation}\label{eq:M_extension}
\pi(\mathcal{O},\Gamma)|_{(G_{\mathbb{R}})_e}\simeq \bigoplus_{w\in G_{\mathbb{R}}/G_{\mathbb{R}}^{\lambda}} w\cdot \pi(\mathcal{O}^{(G_{\mathbb{R}})_e},\Gamma^{(G_{\mathbb{R}})_e}).
\end{equation}

Moreover, if $\mathcal{O}^{G_{\mathbb{R}}}$ is in the good range, then $\pi(\mathcal{O},\Gamma)$ is irreducible since its Harish-Chandra module $(I_{\mathfrak{q},L_{\mathbb{R}}\cap K_{\mathbb{R}}}^{\mathfrak{g}, K_{\mathbb{R}}})^s(\Gamma_{\lambda}\otimes e^{\rho(\mathfrak{n})})$ is irreducible (see Theorem 8.2 of \cite{KV95}).

\subsection{The Schmid-Wong Construction of Maximal Globalizations}
\label{subsec:max_globalization}

Let $G$ be a connected, complex reductive algebraic group, and let $G_{\mathbb{R}}$ be the identity component of $G^{\sigma}\subset G$. In this section, we consider an elliptic orbital parameter $(\mathcal{O},\Gamma)$ for $G_{\mathbb{R}}$ with $\mathcal{O}$ in the good range, and we review the construction of Schmid and Wong of the maximal globalization of the Harish-Chandra module of $\pi(\mathcal{O},\Gamma)$. We will need this maximal globalization construction in Section \ref{sec:KSV}.

Giving a $G_{\mathbb{R}}$-invariant almost complex structure on $\mathcal{O}$ is equivalent to writing 
\[T_{\lambda}\mathcal{O} \otimes_{\mathbb{R}}\mathbb{C}\simeq W\oplus \overline{W}\] 
where $W$, $\overline{W}$ are $G_{\mathbb{R}}(\lambda)$-invariant complex conjugate subspaces of $T_{\lambda}\mathcal{O} \otimes_{\mathbb{R}}\mathbb{C}$. Such a decomposition yields a composition
\[T_{\lambda}\mathcal{O}\hookrightarrow T_{\lambda}\mathcal{O}\otimes_{\mathbb{R}} \mathbb{C}
 \twoheadrightarrow W\]
which defines an isomorphism of real vector spaces $T_{\lambda}\mathcal{O}\simeq W$ and gives $T_{\lambda}\mathcal{O}$ the structure of a complex vector space. In our case, identifying $T_{\lambda}\mathcal{O}\simeq \mathfrak{g}_{\mathbb{R}}/\mathfrak{l}_{\mathbb{R}}$, we may define $W:= \overline{\mathfrak{q}}_{\lambda}/\mathfrak{l}$ in order to give $\mathcal{O}$ a $G_{\mathbb{R}}$-invariant almost complex structure.

Further, Newlander-Nirenburg \cite{NN57} showed that an almost complex structure is integrable (arises from a complex structure) if, and only if $d=\partial+\overline{\partial}$ (see pages 141-142 of \cite{Hor66} for a simpler proof). One checks that this is equivalent to the vanishing of the Nijenhuis tensor and it is equivalent to $T^{(1,0)}\mathcal{O}$ (or $T^{(0,1)}\mathcal{O}$) being closed under the Lie bracket, $[\cdot,\cdot]$. One then checks that, in our context, this condition is satisfied if, and only if the subspace $\mathfrak{q}_{\lambda}\subset \mathfrak{g}$ is a Lie subalgebra. If $(\mathcal{O},\Gamma)$ is an elliptic orbital parameter for $G_{\mathbb{R}}$ and $\theta$ is a Cartan involution on $\mathfrak{g}_{\mathbb{R}}$, then any $\theta$-stable parabolic subalgebra $\mathfrak{q}_{\lambda}\subset \mathfrak{g}$ with Levi factor $\mathfrak{g}(\lambda)$ defines a polarization of $\mathcal{O}$. Therefore, it induces a complex structure on $\mathcal{O}$. 

Next, suppose $\mathfrak{n}$ is the nilradical of $\mathfrak{q}_{\lambda}$. Giving a $G_{\mathbb{R}}$-equivariant, holomorphic line bundle on $\mathcal{O}$ is equivalent to giving a one-dimensional $(\mathfrak{q}_{\lambda},G_{\mathbb{R}}(\lambda)\cap K_{\mathbb{R}})$-module (see Theorem 3.6 on page 17 of \cite{TW70}). In particular, if we differentiate $\Gamma_{\lambda}\otimes e^{\rho(\mathfrak{n})}$, and we extend by zero on $\mathfrak{n}$, then we obtain a one-dimensional $(\mathfrak{q}_{\lambda},G_{\mathbb{R}}(\lambda)\cap K_{\mathbb{R}})$-module. Hence,
$$\mathcal{L}_{\Gamma}:=G_{\mathbb{R}}\times_{G_{\mathbb{R}}(\lambda)} (\Gamma_{\lambda}\otimes e^{\rho(\mathfrak{n})})$$
is a $G_{\mathbb{R}}$-equivariant, holomorphic line bundle on $\mathcal{O}$.
\bigskip

Let $T_{\lambda}\mathcal{O}\otimes_{\mathbb{R}}\mathbb{C}\simeq \mathfrak{g}/\mathfrak{g}(\lambda)$ denote the complexified tangent space to $\mathcal{O}$ at $\lambda$ with dual space $T^*_{\lambda}\mathcal{O}$, let $T^{1,0}_{\lambda}\mathcal{O}\simeq \overline{\mathfrak{q}}_{\lambda}/\mathfrak{g}(\lambda)$ denote the holomorphic tangent space with dual space $(T^{1,0}_{\lambda}\mathcal{O})^*$, and let $T^{0,1}_{\lambda}\mathcal{O}\simeq \mathfrak{q}_{\lambda}/\mathfrak{g}(\lambda)$ denote the anti-holomorphic tangent space to $\mathcal{O}$ at $\lambda$ with dual space $(T^{0,1}_{\lambda}\mathcal{O})^*$. Let 
$$\Omega^m(\mathcal{O},\mathcal{L}_{\Gamma}):=C^{\infty}\left(\mathcal{O},\mathcal{L}_{\Gamma}\otimes \bigwedge^mT^*\mathcal{O}\right)$$ denote the space of smooth, complex-valued $m$-forms on $\mathcal{O}$ with values in $\mathcal{L}_{\Gamma}$, and let 
$$\Omega^{p,q}(\mathcal{O},\mathcal{L}_{\Gamma}):=C^{\infty}\left(\mathcal{O},\mathcal{L}_{\Gamma}\otimes \bigwedge^p(T^{1,0}\mathcal{O})^*\wedge \bigwedge^q(T^{0,1}\mathcal{O})^*\right)$$
denote the space of smooth, complex-valued $(p,q)$-forms on $\mathcal{O}$ with values in $\mathcal{L}_{\Gamma}$. These are both Fr\'{e}chet spaces with respect to the topology of uniform convergence of all derivatives on compact subsets of $\mathcal{O}$. The action of $G_{\mathbb{R}}$ on $\mathcal{L}_{\Gamma}\rightarrow \mathcal{O}$ induces a continuous action on both topological vector spaces. Observe we have a direct sum decomposition 
$$\Omega^m(\mathcal{O},\mathcal{L}_{\Gamma})\simeq \bigoplus_{p+q=m} \Omega^{p,q}(\mathcal{O},\mathcal{L}_{\Gamma})$$
of topological vector spaces. Let $d\colon \Omega^m(\mathcal{O}, \mathcal{L}_{\Gamma})\rightarrow \Omega^{m+1}(\mathcal{O}, \mathcal{L}_{\Gamma})$ denote the exterior derivative. If we restrict $d$ to $\Omega^{p,q}$ and then take the component of the image in $\Omega^{p,q+1}$, then the resulting operator is written $\overline{\partial}^{p,q}$ or simply $\overline{\partial}$. Since $d\circ d=0$, we deduce $\overline{\partial}\circ \overline{\partial}=0$. One defines the $q$th Dolbeault cohomology of $\mathcal{O}$ with values in $\mathcal{L}_{\Gamma}$ by
$$H^q(\mathcal{O},\mathcal{L}_{\Gamma}):=\op{ker}\overline{\partial}^{0,q}/\op{Im}\overline{\partial}^{0,q-1}.$$
Wong showed $\op{Im}\overline{\partial}^{0,q-1}$ is a closed subspace of $\op{ker}\overline{\partial}^{0,q}$ in the induced topology \cite{Won95}. This fact implies $H^q(\mathcal{O},\mathcal{L}_{\Gamma})$ is a Fr\'{e}chet space on which $G_{\mathbb{R}}$ acts continuously.
\bigskip

Let $(\pi,V)$ be an irreducible, unitary representation of a real, reductive group $G_{\mathbb{R}}$. Let $V^{\infty}\subset V$ denote the subset of vectors $v\in V$ for which $g\mapsto \pi(g)v$ is a smooth map $G_{\mathbb{R}}\rightarrow V$. We may differentiate the action of $G_{\mathbb{R}}$ to obtain an action of $\mathfrak{g}$ on $V^{\infty}$. If $K_{\mathbb{R}}$ is a maximal compact subgroup of $G_{\mathbb{R}}$, then the collection of $K_{\mathbb{R}}$-finite vectors, $V_{K_{\mathbb{R}}}\subset V$, is an admissible representation of $K_{\mathbb{R}}$ \cite{HC53}. This means that every irreducible representation of $K_{\mathbb{R}}$ occurs with finite multiplicity in $V_{K_{\mathbb{R}}}$. It follows that $V^{\infty}\supset V_{K_{\mathbb{R}}}$ has the structure of a $(\mathfrak{g},K_{\mathbb{R}})$-module. A \emph{globalization} of this Harish-Chandra module is a continuous action of $G_{\mathbb{R}}$ on a complete, locally convex, Hausdorff topological space $W$ for which $W^{\infty}\supset W_{K_{\mathbb{R}}}$ and $W_{K_{\mathbb{R}}}\simeq V_{K_{\mathbb{R}}}$ as a $(\mathfrak{g},K_{\mathbb{R}})$-module. 

If $(\pi,V)$ is an irreducible, unitary representation of $G_{\mathbb{R}}$, then there is a unique (up to isomorphism) \emph{minimal globalization} $W_{\operatorname{min}}$ of $V_{K_{\mathbb{R}}}$. The globalization $W_{\operatorname{min}}$ has the universal property that for every other globalization $W$ of $V_{K_{\mathbb{R}}}$, there is a unique $G_{\mathbb{R}}$-equivariant injection $W_{\operatorname{min}}\hookrightarrow W$. A vector $v\in V$ is \emph{analytic} if $g\mapsto \pi(g)v$ is an analytic function $G_{\mathbb{R}}\rightarrow V$. The collection of analytic vectors $V^{\omega}$ is a globalization isomorphic to $W_{\operatorname{min}}$. In addition, there is a unique (up to isomorphism) \emph{maximal globalization} $W_{\operatorname{max}}$ of $V_{K_{\mathbb{R}}}$. The globalization $W_{\operatorname{max}}$ has the universal property that for every other globalization $W$ of $V_{K_{\mathbb{R}}}$, there is a unique $G_{\mathbb{R}}$-equivariant injection $W\hookrightarrow W_{\operatorname{max}}$. The space $V^{-\omega}:=((V^*)^{\omega})^*$ is the collection of hyperfunction vectors of $V$; it is a globalization isomorphic to $W_{\operatorname{max}}$. The notions of minimal and maximal globalizations were first introduced by Schmid \cite{Sch85}.
\bigskip

Let $(\mathcal{O},\Gamma)$ be an elliptic orbital parameter with $\mathcal{O}$ in the good range. It was proved by Wong \cite{Won95} that, if $s=\dim_{\mathbb{C}}(\mathfrak{n}\cap \mathfrak{k})$, then $H^s(\mathcal{O},\mathcal{L}_{\Gamma})$ is the maximal globalization of the $(\mathfrak{g},K_{\mathbb{R}})$-module $$(I_{\mathfrak{q}_{\lambda},G_{\mathbb{R}}(\lambda)\cap K_{\mathbb{R}}}^{\mathfrak{g},K_{\mathbb{R}}})^s(\Gamma_{\lambda}\otimes e^{\rho(\mathfrak{n})}).$$
Therefore, $H^s(\mathcal{O},\mathcal{L}_{\Gamma})$ is the maximal globalization of the Harish-Chandra module of $\pi(\mathcal{O},\Gamma)$. 

\subsection{Mackey's Parabolic Induction}
\label{subsec:parabolic_induction}
We finish this section by completing the construction of
 the unitary representation $\pi(\mathcal{O},\Gamma)$
 for a semisimple orbital parameter $(\mathcal{O},\Gamma)$. 
By applying the construction in Section \ref{subsec:VZ_construction}
 to the group $M_{\mathbb{R}}$, we obtain a unitary representation
 $\pi(\mathcal{O}^{M_{\mathbb{R}}},\Gamma^{M_{\mathbb{R}}})$ of $M_{\mathbb{R}}$. 
We extend the representation
 $\pi(\mathcal{O}^{M_{\mathbb{R}}},\Gamma^{M_{\mathbb{R}}})
 \boxtimes e^{\lambda_n}$ of $M_{\mathbb{R}}A_{\mathbb{R}}$
 trivially on $(N_P)_{\mathbb{R}}$ to a representation of $P_{\mathbb{R}}$. 
This unitary representation defines a (possibly infinite dimensional)
 $G_{\mathbb{R}}$-equivariant Hermitian vector bundle
\[{\mathcal V}:=
G_{\mathbb{R}}\times_{P_{\mathbb{R}}} 
 (\pi(\mathcal{O}^{M_{\mathbb{R}}},\Gamma^{M_{\mathbb{R}}})
 \boxtimes  e^{\lambda_n})
\longrightarrow G_{\mathbb{R}}/P_{\mathbb{R}}.\]
Then we form the representation
\begin{align*}
\pi(\mathcal{O},\Gamma)
:=L^2(G_{\mathbb{R}}/P_{\mathbb{R}},{\mathcal V}\otimes \mathcal{D}^{1/2})
\end{align*}
where $\mathcal{D}^{1/2}\rightarrow G_{\mathbb{R}}/P_{\mathbb{R}}$ denotes the bundle of half densities on $G_{\mathbb{R}}/P_{\mathbb{R}}$. This defines a unitary representation of $G_{\mathbb{R}}$ which we associate to the parameter $(\mathcal{O},\Gamma)$. We say $\lambda$ (or $\mathcal{O}$) is in the \emph{good range} if $\mathcal{O}^{M_{\mathbb{R}}}$ is in the good range for $M_{\mathbb{R}}$. 

If $\mathcal{O}$ is in the good range, then the representation $\pi(\mathcal{O},\Gamma)$ is irreducible (see Theorem 14.93 on page 618 of \cite{Kna86}).

Finally, we remark that the representation $\pi(\mathcal{O},\Gamma)$ does not depend on the choice of maximally real, admissible polarization $\{\mathfrak{q}_{\lambda}\}_{\lambda\in \mathcal{O}}$ even though we used this polarization to define $\pi(\mathcal{O},\Gamma)$. Indeed, the parabolic $\mathfrak{q}_{\mathfrak{m}}\subset \mathfrak{m}$ defined in (\ref{eq:imaginary_parabolic}) does not depend on the choice of maximally real, admissible polarization. Therefore, the representation $\pi(\mathcal{O}^{M_{\mathbb{R}}},\Gamma^{M_{\mathbb{R}}})$ is independent of the choice of maximally real, admissible polarization. Now, a glance at the induced character formula (see for instance page 352 of \cite{Kna86}) shows that the Harish-Chandra character of the parabolically induced representation $\pi(\mathcal{O},\Gamma)$ does not depend on the choice of real parabolic subgroup $P_{\mathbb{R}}\subset G_{\mathbb{R}}$. Since the Harish-Chandra character determines the representation \cite{HC54b}, $\pi(\mathcal{O},\Gamma)$ is independent of the choice of maximally real, admissible polarization.

\section{Reduction to the Elliptic Case}
\label{sec:reduction_elliptic}

In this section, we reduce the proof of Theorem \ref{thm:main} to the case where $(\mathcal{O},\Gamma)$ is an elliptic orbital parameter for $G_{\mathbb{R}}$. First, we must discuss the work of Rossmann \cite{Ro84}, \cite{Ro90}.

\subsection{Rossmann's Work on Character Formulas and Contours}
\label{subsec:Rossmann_characters}

Suppose $\pi$ is an irreducible, admissible representation of a real, reductive group $G_{\mathbb{R}}$ and let $\Theta(\pi)$ denote the Harish-Chandra character of $\pi$. The Harish-Chandra character is a generalized function on $G_{\mathbb{R}}$ given by integration against an analytic, locally $L^1$ function on $G_{\mathbb{R}}'\subset G_{\mathbb{R}}$, the subset of regular semisimple elements which we will, by an abuse of notation, also call $\Theta(\pi)$ \cite{HC65b}. Let $|dg|$ be a non-zero, invariant density on $G_{\mathbb{R}}$, and let $|dX|$ be a non-zero, translation invariant density on $\mathfrak{g}_{\mathbb{R}}$. We may pull back the density $|dg|$ under the exponential map 
$$\exp\colon \mathfrak{g}_{\mathbb{R}}\longrightarrow G_{\mathbb{R}}$$
to obtain an equality
$$\exp^*|dg|=j_{G_{\mathbb{R}}}(X)|dX|$$
for an $\op{Ad}(G_{\mathbb{R}})$-invariant analytic function $j_{G_{\mathbb{R}}}$ on $\mathfrak{g}_{\mathbb{R}}$. We normalize the densities $|dg|$ and $|dX|$ so that $j_{G_{\mathbb{R}}}(0)=1$, and one notes that there exists a unique analytic function $j_{G_{\mathbb{R}}}^{1/2}$ satisfying $(j_{G_{\mathbb{R}}}^{1/2})^2=j_{G_{\mathbb{R}}}$ and $j_{G_{\mathbb{R}}}^{1/2}(0)=1$ \cite{HC65b}.

If we pull back the analytic function $\Theta(\pi)$ to the Lie algebra $\mathfrak{g}_{\mathbb{R}}$ and multiply by $j_{G_{\mathbb{R}}}^{1/2}$, then we obtain
\[\theta(\pi):=\exp^*\Theta(\pi)\cdot j_{G_{\mathbb{R}}}^{1/2}.\]
The function $\theta(\pi)$ is analytic on a dense subset of $\mathfrak{g}_{\mathbb{R}}$, it defines a distribution on $\mathfrak{g}_{\mathbb{R}}$, and it is called the Lie algebra analogue of the character of $\pi$. The distribution $\Theta(\pi)$ is an eigendistribution for the algebra of invariant constant coefficient linear differential operators on $G_{\mathbb{R}}$. Multiplication by $j_{G_{\mathbb{R}}}^{1/2}$ is necessary to make $\theta(\pi)$ into an invariant eigendistribution on $\mathfrak{g}_{\mathbb{R}}$ \cite{HC65b}. Let $\op{Diff}(\mathfrak{g}_{\mathbb{R}})^{G_{\mathbb{R}}}$ denote the algebra of $\op{Ad}^*(G_{\mathbb{R}})$-invariant constant coefficient linear differential operators on $\mathfrak{g}_{\mathbb{R}}$. An \emph{invariant eigendistribution} on $\mathfrak{g}_{\mathbb{R}}$ is a $G_{\mathbb{R}}$-invariant distribution $\theta$ on $\mathfrak{g}_{\mathbb{R}}$ for which there exists an algebra homomorphism 
$$
\chi_{\theta}\colon \op{Diff}(\mathfrak{g}_{\mathbb{R}})^{G_{\mathbb{R}}}\rightarrow \mathbb{C}
$$
such that
$$D\theta=\chi_{\theta}(D)\theta.$$
The algebra homomorphism $\chi_{\theta}$ is called the \emph{infinitesimal character} of the invariant eigendistribution $\theta$. By abuse of notation, we will also call a certain Weyl group orbit on $\mathfrak{c}^*$, the dual space of the universal Cartan subalgebra of $\mathfrak{g}$, associated to $\theta$ (and defined below) the infinitesimal character of $\theta$.

If $D\in \op{Diff}(\mathfrak{g}_{\mathbb{R}})$ is a constant coefficient linear differential operator on $\mathfrak{g}$, then there exists a unique (holomorphic) polynomial $p\in \op{Pol}(\mathfrak{g}^*)$ such that 
$$\mathcal{F}[D\varphi]=p\cdot \mathcal{F}[\varphi]$$
for all $\varphi\in C_c^{\infty}(\mathfrak{g}_{\mathbb{R}})$. In this case, we write $\mathcal{F}[D]=p$ and we say that $p$ is the Fourier transform of $D$. Then the Fourier transform yields an isomorphism
$$\mathcal{F}\colon \op{Diff}(\mathfrak{g}_{\mathbb{R}})^{G_{\mathbb{R}}}\stackrel{\sim}\longrightarrow \op{Pol}(\mathfrak{g}^*)^G$$ 
between $\op{Ad}^*(G_{\mathbb{R}})$-invariant constant coefficient linear differential operators on $\mathfrak{g}_{\mathbb{R}}$ and $\op{Ad}^*(G)$-invariant polynomials on $\mathfrak{g}^*$. If $\theta$ is an invariant eigendistribution on $\mathfrak{g}_{\mathbb{R}}$, then the \emph{geometric infinitesimal character} of $\theta$ is defined to be the $\op{Ad}(G)$-invariant subset of $\mathfrak{g}^*$ given by
\begin{equation}\label{eq:Inf_Ch}
\Omega(\theta):=\left\{\lambda\in \mathfrak{g}^* \mid \mathcal{F}[D](\lambda)=\chi_{\theta}(D)\ \text{for\ all}\ D\in \op{Diff}(\mathfrak{g}_{\mathbb{R}})^{G_{\mathbb{R}}}\right\}.
\end{equation}
If $\theta=\theta(\pi)$ is the Lie algebra analogue of the character of an irreducible, admissible representation $\pi$, then we call $\Omega(\pi):=\Omega(\theta)$ the geometric infinitesimal character of $\pi$. If $\mathfrak{h}\subset \mathfrak{g}$ is a Cartan subalgebra, then we may use the root space decomposition of $\mathfrak{g}$ with respect to $\mathfrak{h}$ to embed $\mathfrak{h}^*\subset \mathfrak{g}^*$ and we may restrict $\op{Ad}^*(G)$-invariant polynomials on $\mathfrak{g}^*$ to Weyl group invariant polynomials on $\mathfrak{h}^*$. Taking the direct limit, we obtain a restriction map
\[\op{Pol}(\mathfrak{g}^*)^G\rightarrow \op{Pol}(\mathfrak{c}^*)^W\]
denoted $p\mapsto p|_{\mathfrak{c}^*}$ where $W$ denotes the Weyl group of the universal Cartan subalgebra $\mathfrak{c}$.
Then we obtain a natural fibration
\begin{equation}\label{eq:fibration}
q\colon \mathfrak{g}^*\rightarrow \mathfrak{c}^*/W
\end{equation}
where $\xi\in \mathfrak{g}^*$ maps to the unique $W$-orbit of points $\eta\in \mathfrak{c}^*$ with
\[p(\xi)=p|_{\mathfrak{c}^*}(\eta)\] 
for all $p\in \op{Pol}(\mathfrak{g}^*)^G$. Then $\Omega(\theta)$ is the inverse image under $q$ of the \emph{infinitesimal character} of $\theta$; in fact one can take this to be the definition of the infinitesimal character. One deduces that $\Omega(\theta)$ is a finite union of $\op{Ad}^*(G)$-orbits in $\mathfrak{g}^*$. If the infinitesimal character of $\theta$ is a Weyl group orbit in $\mathfrak{c}^*$ consisting of regular elements or equivalently, $\Omega(\theta)$ consists entirely of regular elements, then we say that $\theta$ has \emph{regular infinitesimal character}. In this case, $\Omega(\theta)$ is a closed $G_{\mathbb{R}}$-orbit of maximal dimension. When $\theta=\theta(\pi)$ is the Lie algebra analogue of the character of an irreducible admissible representation of $\pi$, we say $\pi$ has regular infinitesimal character if $\theta(\pi)$ has regular infinitesimal character. The infinitesimal character of the representation $\pi(\mathcal{O},\Gamma)$ is the Weyl group orbit through $\lambda+\rho_{\mathfrak{l}}\in \mathfrak{h}^*\simeq \mathfrak{c}^*$ (see Corollary 5.25 of \cite{KV95}). If $\mathcal{O}$ is in the good range, then $\lambda+\rho_{\mathfrak{l}}$ is regular and $\pi(\mathcal{O},\Gamma)$ has regular infinitesimal character.
\bigskip

Suppose $\Omega$ is a regular, semisimple complex coadjoint orbit for $G$ in $\mathfrak{g}^*$. Write $\mathfrak{g}^*=\mathfrak{g}_{\mathbb{R}}^*\oplus \sqrt{-1}\mathfrak{g}_{\mathbb{R}}^*$, write $\xi=\op{Re}\xi+\op{Im}\xi$ for the corresponding decomposition of $\xi\in \mathfrak{g}^*$, and let $|\cdot|$ be a fixed norm on $\mathfrak{g}^*$. Rossmann defined a locally finite $\mathbb{Z}$-linear combination 
$$\gamma=\sum c_k \sigma_k$$
of continuous maps $\sigma_k$ from a $d$-simplex to $\Omega$ to be an \emph{admissible d-chain} if 
\begin{enumerate}[(a)]
\item There exists a constant $C$ for which $|\op{Re}\sigma_k|\leq C$ for all $k$.
\item The sum
$$\sum_k |c_k| \max_{x\in \op{supp}\sigma_k}\frac{\op{vol}(\sigma_k)}{1+|x|^N}$$
is finite for some natural number $N>0$ where $\op{vol}(\sigma_k)$ denotes the Euclidean volume of the image of $\sigma_k$ in $\Omega$.
\item The boundary $\partial \gamma$ satisfies (a) and (b).
\end{enumerate}  

Then Rossmann formed a theory of homology by defining $_{\prime}\mathcal{H}_d(\Omega)$ to be the space of closed, admissible $d$-chains modulo the space of exact, admissible $d$-chains for every nonnegative integer $d$ (see page 264 of \cite{Ro90}). Let 
$$n=\dim_{\mathbb{C}}\Omega=\frac{1}{2}\dim_{\mathbb{R}}\Omega.$$
An oriented, closed, $n$-dimensional, real-analytic submanifold $\mathcal{M}\subset \Omega$, which admits an admissible triangulation defines an element of $_{\prime}\mathcal{H}_n(\Omega)$. Moreover, any two admissible triangulations of $\mathcal{M}$ yield the same element of $_{\prime}\mathcal{H}_n(\Omega)$. Further, any element of $_{\prime}\mathcal{H}_n(\Omega)$ can be represented as a finite linear combination of oriented, $n$-dimensional, real-analytic submanifolds $\mathcal{M}\subset \Omega$ which admit admissible triangulations (see Remark 1.1 on pages 265-266 of \cite{Ro90}). An \emph{admissible contour} in $\Omega$ is a finite linear combination of oriented, $n$-dimensional, real-analytic submanifolds $\mathcal{M}\subset \Omega$ which admit admissible triangulations.
\bigskip

If $\mu\in C_c^{\infty}(\mathfrak{g}_{\mathbb{R}},\mathcal{D}({\mathfrak{g}_{\mathbb{R}}}))$ is a smooth, compactly supported density on $\mathfrak{g}$, then we define the \emph{Fourier transform} of $\mu$ to be
$$\mathcal{F}[\mu](\xi):=\int_{\mathfrak{g}_{\mathbb{R}}} e^{\langle \xi,X\rangle} d\mu(X).$$
By the Payley-Wiener Theorem, $\mathcal{F}[\mu]$ is a holomorphic function on $\mathfrak{g}^*$ that decays rapidly when restricted to any strip $\{\xi\in \mathfrak{g}^* \mid |\op{Re}\xi|\leq C\}$ with $C>0$. If $\mathcal{C}\subset \Omega$ is an admissible contour, then we define the \emph{Fourier transform} of $\mathcal{C}$ to be 
\begin{equation}\label{eq:Fourier_transform_2}
\langle \mathcal{F}[\mathcal{C}], \mu\rangle:=\int_{\mathcal{C}} \mathcal{F}[\mu] \frac{\omega^{\wedge m}}{(2\pi\sqrt{-1})^m m!}
\end{equation}
\noindent where $\omega$ is the Kirillov-Kostant form on $\Omega$ (see (\ref{eq:Kirillov-Kostant_form}) and replace $G_{\mathbb{R}}$ with the complex group $G$) and $n=2m$. Notice that the conditions (a), (b) guarantee the convergence of (\ref{eq:Fourier_transform_2}). By Stokes' Theorem, the above integral only depends on the homology class $[\mathcal{C}]\in$ $_{\prime}\mathcal{H}_n(\Omega)$. Rossmann proved that every Lie algebra analogue of the character $\theta(\pi)$ of an irreducible admissible representation $\pi$ of $G_{\mathbb{R}}$ with geometric infinitesimal character $\Omega$ can be written as $\mathcal{F}[\mathcal{C}(\pi)]$ for a unique homology class $[\mathcal{C}(\pi)]\in$ $_{\prime}\mathcal{H}_n(\Omega)$. Moreover, these elements span the finite dimensional complex vector space $_{\prime}\mathcal{H}_n(\Omega)$ (see Theorem 1.4 on page 268 of \cite{Ro90}).
\bigskip

We now discuss an orientation on the manifold
 $\mathcal{C}(\mathcal{O},\Gamma,\mathfrak{q},\sigma_c)$
 to define its Fourier transform
 $\mathcal{F}[\mathcal{C}(\mathcal{O},\Gamma,\mathfrak{q},\sigma_c)]$.
Utilizing (\ref{eq:contour}), we have a fiber bundle structure
$$\varpi\colon \mathcal{C}(\mathcal{O},\Gamma,\mathfrak{q},\sigma_c)\rightarrow \mathcal{O}.$$
Put $k:=\frac{1}{2}\dim\mathcal{O}$.
The canonical symplectic form $\omega$ is purely imaginary on $\mathcal{O}$.
Therefore, the top-dimensional form
\[\frac{\omega^{\wedge k}}{(\sqrt{-1})^k}\] 
gives an orientation on $\mathcal{O}$.
Next, let $l$ be half the dimension of a fiber of $\varpi$. 
Recall from (\ref{eq:contour}) that
 $\varpi^{-1}(\lambda)=\lambda+\mathcal{O}_{\rho}^{U_{\lambda}}$. 
Here $\mathcal{O}_{\rho}^{U_{\lambda}}$ is a coadjoint orbit in
 $\sqrt{-1}\mathfrak{u}_{\lambda}^*$ for the compact subgroup $U_{\lambda}$ of $L$. 
Therefore, we get an orientation on $\varpi^{-1}(\lambda)$
 by using the canonical symplectic form on
 the coadjoint orbit $\mathcal{O}_{\rho}^{U_{\lambda}}$
 as we did for $\mathcal{O}$.
They define an orientations on the total space 
 $\mathcal{C}(\mathcal{O},\Gamma,\mathfrak{q},\sigma_c)$.

We remark that the above definition depends on a choice of $\sqrt{-1}$.
The other choice of $\sqrt{-1}$ reverses the orientation on
 $\mathcal{C}(\mathcal{O},\Gamma,\mathfrak{q},\sigma_c)$ exactly when
 $m(=k+l)$ is odd.
However, the definition of the Fourier transform \eqref{eq:Fourier_transform_2}
 also involves $\sqrt{-1}$ and one can see that 
 $\mathcal{F}[\mathcal{C}(\mathcal{O},\Gamma,\mathfrak{q},\sigma_c)]$
 does not depend on the choice of $\sqrt{-1}$.
Therefore, our formula (\ref{eq:character_formula})
 has canonical meaning independent of a choice of $\sqrt{-1}$.

\subsection{Reduction to the Elliptic Case}
\label{subsec:elliptic_reduction}

We return to the notation of Section~\ref{subsec:parabolics}. As in (\ref{eq:contourII}), we define the contour

\[\mathcal{C}(\mathcal{O},\Gamma,\mathfrak{q},\sigma_c):=\left\{g\cdot \lambda+u\cdot \rho_{\mathfrak{l}} \mid g\in G_{\mathbb{R}},\ u\in U,\ g\cdot \mathfrak{q}_{\lambda}=u\cdot \mathfrak{q}_{\lambda}\right\}\]
where $\rho_{\mathfrak{l}}$ is half the sum of the positive roots of $\mathfrak{l}$ with respect to a $\sigma$ and $\sigma_c$-stable Cartan subalgebra $\mathfrak{h}\subset \mathfrak{l}$ and a choice of positive roots. In addition, we define the contour
\begin{align*}
&\mathcal{C}(\mathcal{O}^{(M_{\mathbb{R}})_e},\Gamma^{(M_{\mathbb{R}})_e},\mathfrak{q}_{\mathfrak{m}},\sigma_c|_{M_e})\\
&:=\left\{m\cdot \lambda_c+u\cdot \rho_{\mathfrak{l}\cap\mathfrak{m}} \mid m\in (M_{\mathbb{R}})_e,\ u\in U\cap M_e,\ m\cdot \mathfrak{q}_{\mathfrak{m}}=u\cdot \mathfrak{q}_{\mathfrak{m}}\right\}
\end{align*}
where $\rho_{\mathfrak{l}\cap\mathfrak{m}}$ is half the sum of the positive roots of $\mathfrak{l}\cap\mathfrak{m}$ which are restrictions of positive roots of $\mathfrak{l}$ to $\mathfrak{h}\cap\mathfrak{m}$.
Note that
 $\rho_{\mathfrak{l}}|_{\mathfrak{h}\cap\mathfrak{m}}=\rho_{\mathfrak{l}\cap\mathfrak{m}}$
 and $\rho_{\mathfrak{l}}|_{\mathfrak{a}}=0$.

In the case where $\pi=\pi(\mathcal{O},\Gamma)$, we let $\theta(\mathcal{O},\Gamma):=\theta(\pi)$ denote the Lie algebra analogue of the character.

\begin{lemma}\label{lem:elliptic_reduction}
Let $G$ be a connected complex reductive algebraic group, and let $(G^{\sigma})_e\subset G_{\mathbb{R}}\subset G^{\sigma}$ be a real form of $G$. Let $(\mathcal{O},\Gamma)$ be a semisimple orbital parameter for $G_{\mathbb{R}}$, fix $\lambda\in \mathcal{O}$, and fix a maximal compact subgroup $U=G^{\sigma_c}\subset G$ for which the anti-holomorphic involution $\sigma_c$ commutes with $\sigma$ and stabilizes $L$. Define $M$, $M_{\mathbb{R}}$, $(\mathcal{O}^{M_{\mathbb{R}}},\Gamma^{M_{\mathbb{R}}})$, and $\mathfrak{q}_{\mathfrak{m}}$ as in Section~\ref{subsec:parabolics}. 
If the character formula
\begin{equation} \label{eq:elliptic_character_formula}
\theta(\mathcal{O}^{(M_{\mathbb{R}})_e},\Gamma^{(M_{\mathbb{R}})_e})=\mathcal{F}[\mathcal{C}(\mathcal{O}^{(M_{\mathbb{R}})_e},\Gamma^{(M_{\mathbb{R}})_e},\mathfrak{q}_{\mathfrak{m}},\sigma_c|_{M_e})]
\end{equation} 
holds for $(M_{\mathbb{R}})_e$ and $\mathcal{O}$ is in the good range, then the character formula
\begin{equation} \label{eq:general_character_formula}
\theta(\mathcal{O},\Gamma)=\mathcal{F}[\mathcal{C}(\mathcal{O},\Gamma,\mathfrak{q},\sigma_c)]
\end{equation} 
holds for $G_{\mathbb{R}}$.
\end{lemma}

\begin{proof} We prove the Lemma in two steps. Define 
\begin{align*}
&\mathcal{C}(\mathcal{O}^{M_{\mathbb{R}}},\Gamma^{M_{\mathbb{R}}},\mathfrak{q}_{\mathfrak{m}},\sigma_c|_{M_e})\\
&:=\left\{m\cdot \lambda_c+u\cdot \rho_{\mathfrak{l}\cap\mathfrak{m}}\mid m\in M_{\mathbb{R}},\ u\in U\cap M,\ m\cdot \mathfrak{q}_{\mathfrak{m}}=u\cdot \mathfrak{q}_{\mathfrak{m}}\right\}.
\end{align*}
First, we will prove that (\ref{eq:elliptic_character_formula}) implies 
\begin{equation}\label{eq:disconnected_elliptic_character_formula}
\theta(\mathcal{O}^{M_{\mathbb{R}}},\Gamma^{M_{\mathbb{R}}})=\mathcal{F}[\mathcal{C}(\mathcal{O}^{M_{\mathbb{R}}},\Gamma^{M_{\mathbb{R}}},\mathfrak{q}_{\mathfrak{m}},\sigma_c|_{M})].
\end{equation}
Then we will prove (\ref{eq:disconnected_elliptic_character_formula}) implies (\ref{eq:general_character_formula}). For the first step, by the hypothesis (\ref{eq:elliptic_character_formula}), we have
$$\theta(w\cdot \mathcal{O}^{(M_{\mathbb{R}})_e},w\cdot \Gamma^{(M_{\mathbb{R}})_e})=\mathcal{F}[\mathcal{C}(w\cdot \mathcal{O}^{(M_{\mathbb{R}})_e},w\cdot \Gamma^{(M_{\mathbb{R}})_e},w\cdot \mathfrak{q}_{\mathfrak{m}}, (w\cdot \sigma_c)|_{M_e})]$$
for $w\in M_{\mathbb{R}}/M_{\mathbb{R}}^{\lambda}$. 
Here, $M_{\mathbb{R}}^{\lambda}$ is the subgroup of $M_{\mathbb{R}}$ generated
 by $(M_{\mathbb{R}})_e$ and $L_{\mathbb{R}}\cap M_{\mathbb{R}}$
 as in Section~\ref{subsec:VZ_construction}.
Further, one checks
\begin{align*}
&\mathcal{C}(w\cdot \mathcal{O}^{(M_{\mathbb{R}})_e},w\cdot \Gamma^{(M_{\mathbb{R}})_e},w\cdot \mathfrak{q}_{\mathfrak{m}},(w\cdot \sigma_c)|_{M_e})\\
&=\{m\cdot w\cdot \lambda_c+u\cdot w\cdot \rho_{\mathfrak{l}\cap\mathfrak{m}}\mid m\in (M_{\mathbb{R}})_e,\ u\in (wUw^{-1})\cap M_e,\ m\cdot w\cdot\mathfrak{q}_{\mathfrak{m}}=u\cdot w\cdot \mathfrak{q}_{\mathfrak{m}} \}\\
&=\{w\cdot m\cdot \lambda_c+w\cdot u\cdot \rho_{\mathfrak{l}\cap\mathfrak{m}}\mid m\in (M_{\mathbb{R}})_e,\ u\in U\cap M_e,\ w\cdot m\cdot\mathfrak{q}_{\mathfrak{m}}=w\cdot u\cdot \mathfrak{q}_{\mathfrak{m}} \}
\end{align*}
for $w\in M_{\mathbb{R}}/M_{\mathbb{R}}^{\lambda}$. Next, observe that $G(\lambda_n)$ is connected since it is a Levi subgroup of $G$.
If $A$ denotes the connected torus with Lie algebra $\mathfrak{a}$, then $MA=G(\lambda_n)$,
 showing that the adjoint action of $M$ is inner.  
Hence for any $w\in M_{\mathbb{R}}/M_{\mathbb{R}}^{\lambda}$ and $u\in U\cap M_e$, there exists
 $u'\in U\cap M_e$ such that
 $w\cdot u\cdot \rho_{\mathfrak{l}\cap\mathfrak{m}}= u' \cdot \rho_{\mathfrak{l}\cap\mathfrak{m}}$
 and $w\cdot u\cdot \mathfrak{q}_{\mathfrak{m}}=u'\cdot \mathfrak{q}_{\mathfrak{m}}$.
Therefore, we deduce
\begin{align*}
&\mathcal{C}(w\cdot \mathcal{O}^{(M_{\mathbb{R}})_e},w\cdot \Gamma^{(M_{\mathbb{R}})_e},w\cdot \mathfrak{q}_{\mathfrak{m}},(w\cdot \sigma_c)|_{M_e})\\
&=\{w\cdot m\cdot \lambda_c+u\cdot \rho_{\mathfrak{l}\cap\mathfrak{m}}\mid m\in (M_{\mathbb{R}})_e,\ u\in U\cap M_e,\ w\cdot m\cdot\mathfrak{q}_{\mathfrak{m}}=u\cdot \mathfrak{q}_{\mathfrak{m}} \}
\end{align*}
for $w\in M_{\mathbb{R}}/M_{\mathbb{R}}^{\lambda}$, which immediately implies 
\begin{equation}\label{eq:contour_M_equality}
\mathcal{C}(\mathcal{O}^{M_{\mathbb{R}}},\Gamma^{M_{\mathbb{R}}},\mathfrak{q}_{\mathfrak{m}},\sigma_c|_{M})=\sum_{w\in M_{\mathbb{R}}/M_{\mathbb{R}}^{\lambda}}\mathcal{C}(w\cdot \mathcal{O}^{(M_{\mathbb{R}})_e},w\cdot \Gamma^{(M_{\mathbb{R}})_e},w\cdot \mathfrak{q}_{\mathfrak{m}},(w\cdot \sigma_c)|_{M_e}).
\end{equation}

Next, recall from (\ref{eq:M_extension}), we have
$$\pi(\mathcal{O}^{M_{\mathbb{R}}},\Gamma^{M_{\mathbb{R}}})|_{(M_{\mathbb{R}})_e}\simeq \bigoplus_{w\in M_{\mathbb{R}}/M_{\mathbb{R}}^{\lambda}} \pi(w\cdot \mathcal{O}^{(M_{\mathbb{R}})_e},w\cdot \Gamma^{(M_{\mathbb{R}})_e}).$$
In particular,
\begin{equation}\label{eq:character_M_equation}
\theta(\mathcal{O}^{M_{\mathbb{R}}},\Gamma^{M_{\mathbb{R}}})=\sum_{w\in M_{\mathbb{R}}/M_{\mathbb{R}}^{\lambda}} \theta(w\cdot \mathcal{O}^{(M_{\mathbb{R}})_e},w\cdot \Gamma^{(M_{\mathbb{R}})_e}).
\end{equation}
Combining (\ref{eq:contour_M_equality}) and (\ref{eq:character_M_equation}), we obtain (\ref{eq:disconnected_elliptic_character_formula}).

For the second step, recall from Section \ref{subsec:parabolic_induction}
\[\pi(\mathcal{O},\Gamma)
=\op{Ind}_{P_{\mathbb{R}}}^{G_{\mathbb{R}}}
(\pi(\mathcal{O}^{M_{\mathbb{R}}},\Gamma^{M_{\mathbb{R}}})\boxtimes e^{\lambda_n})
\]
is obtained via (unitary) parabolic induction. We have assumed that $\mathcal{O}$ is in the good range, which implies that the infinitesimal character of $\pi(\mathcal{O}^{M_{\mathbb{R}}},\Gamma^{M_{\mathbb{R}}})$ is regular. Moreover, $\mathcal{C}(\mathcal{O}^{M_{\mathbb{R}}},\Gamma^{M_{\mathbb{R}}},\mathfrak{q}_{\mathfrak{m}},\sigma_c|_{M})$ is an admissible contour.
Hence, we may utilize the Lemma on page 377 of \cite{Ro84} together with the assumption (\ref{eq:elliptic_character_formula}) and the fact that $\mathcal{O}$ in the good range implies that $\pi(\mathcal{O},\Gamma)$ has regular infinitesimal character to deduce that $\theta(\mathcal{O},\Gamma)$ can be written as the Fourier transform of the contour
\[\widetilde{\mathcal{C}}(\mathcal{O},\Gamma,\mathfrak{q},\sigma_c)
:=K_{\mathbb{R}}\cdot
 [\mathcal{C}(\mathcal{O}^{M_{\mathbb{R}}},\Gamma^{M_{\mathbb{R}}},\mathfrak{q}_{\mathfrak{m}},\sigma_c|_{M})+\lambda_n+\sqrt{-1}(\mathfrak{g}_{\mathbb{R}}/\mathfrak{p}_{\mathbb{R}})^*].\]

Let us pause for a moment to discuss the orientation on
 $\widetilde{\mathcal{C}}(\mathcal{O},\Gamma,\mathfrak{q},\sigma_c)$. 
The orientation on 
$\mathcal{C}(\mathcal{O}^{M_{\mathbb{R}}},\Gamma^{M_{\mathbb{R}}},
 \mathfrak{q}_{\mathfrak{m}},\sigma_c|_{M})$
 is defined as in Section \ref{subsec:Rossmann_characters}.
We have a pairing
\begin{equation*}
\langle \cdot , \cdot \rangle :
\mathfrak{k}_{\mathbb{R}}
 /(\mathfrak{k}_{\mathbb{R}}\cap \mathfrak{m}_{\mathbb{R}})
\otimes (\mathfrak{g}_{\mathbb{R}}/\mathfrak{p}_{\mathbb{R}})^*
\rightarrow \mathbb{R}
\end{equation*}
by utilizing the isomorphism
 $\mathfrak{k}_{\mathbb{R}}/(\mathfrak{k}_{\mathbb{R}}\cap \mathfrak{m}_{\mathbb{R}})
 \simeq \mathfrak{g}_{\mathbb{R}}/\mathfrak{p}_{\mathbb{R}}$ 
 and then we define a symplectic form $\gamma$ on
$\mathfrak{k}_{\mathbb{R}}/(\mathfrak{k}_{\mathbb{R}}\cap \mathfrak{m}_{\mathbb{R}}) \oplus (\mathfrak{g}_{\mathbb{R}}/\mathfrak{p}_{\mathbb{R}})^*$ by
\begin{align*}
&\gamma(\mathfrak{k}_{\mathbb{R}}/(\mathfrak{k}_{\mathbb{R}}\cap \mathfrak{m}_{\mathbb{R}}),
 \mathfrak{k}_{\mathbb{R}}/(\mathfrak{k}_{\mathbb{R}}\cap \mathfrak{m}_{\mathbb{R}}))=0,\qquad
\gamma((\mathfrak{g}_{\mathbb{R}}/\mathfrak{p}_{\mathbb{R}})^*,
 (\mathfrak{g}_{\mathbb{R}}/\mathfrak{p}_{\mathbb{R}})^*)=0,\\
&\gamma(X,Y)=\langle X, Y\rangle
 \text{ for } X\in \mathfrak{k}_{\mathbb{R}}/(\mathfrak{k}_{\mathbb{R}}\cap \mathfrak{m}_{\mathbb{R}}),\ Y\in (\mathfrak{g}_{\mathbb{R}}/\mathfrak{p}_{\mathbb{R}})^*.
\end{align*}
If $d:=\dim \mathfrak{g}_{\mathbb{R}}/\mathfrak{p}_{\mathbb{R}}$,
 then $\bigwedge^d \gamma$ gives an orientation on 
 $\mathfrak{k}_{\mathbb{R}}/(\mathfrak{k}_{\mathbb{R}}\cap \mathfrak{m}_{\mathbb{R}})
 \times (\mathfrak{g}_{\mathbb{R}}/\mathfrak{p}_{\mathbb{R}})^*$.
Combining these two, we obtain an orientation on 
\[\widetilde{\mathcal{C}}(\mathcal{O},\Gamma,\mathfrak{q},\sigma_c)
 = K_{\mathbb{R}}\times_{K_{\mathbb{R}}\cap M_{\mathbb{R}}}
 [\mathcal{C}(\mathcal{O}^{M_{\mathbb{R}}},\Gamma^{M_{\mathbb{R}}},
  \mathfrak{q}_{\mathfrak{m}},\sigma_c|_{M})
 +\lambda_n+\sqrt{-1}(\mathfrak{g}_{\mathbb{R}}/\mathfrak{p}_{\mathbb{R}})^*]\]
 and with respect to this orientation
\[\mathcal{F}[\widetilde{\mathcal{C}}(\mathcal{O},\Gamma,\mathfrak{q},\sigma_c)]
=\theta(\mathcal{O},\Gamma)\]
 by the Lemma on page 377
 and the Pfaffian argument on page 379 of \cite{Ro84}.

Next, recall 
$$\mathcal{C}(\mathcal{O}^{M_{\mathbb{R}}},\Gamma^{M_{\mathbb{R}}},\mathfrak{q}_{\mathfrak{m}},\sigma_c|_{M})=\left\{m\cdot \lambda_c+u\cdot \rho_{\mathfrak{l}\cap\mathfrak{m}}\mid m\cdot \mathfrak{q}_{\mathfrak{m}}=u\cdot \mathfrak{q}_{\mathfrak{m}},\ m\in M_{\mathbb{R}},\ u\in U\cap M\right\}.$$
Note $A_{\mathbb{R}}(N_P)_{\mathbb{R}}\cdot \lambda=\lambda+\sqrt{-1}(\mathfrak{g}_{\mathbb{R}}/\mathfrak{p}_{\mathbb{R}})^*$. 
Further, if $m\in M_{\mathbb{R}}$, then $m\cdot [\lambda+\sqrt{-1}(\mathfrak{g}_{\mathbb{R}}/\mathfrak{p}_{\mathbb{R}})^*]=m\cdot \lambda+\sqrt{-1}(\mathfrak{g}_{\mathbb{R}}/\mathfrak{p}_{\mathbb{R}})^*$.
Hence, 
\[
\widetilde{\mathcal{C}}(\mathcal{O},\Gamma,\mathfrak{q},\sigma_c)
=\bigcup_{\substack{k\in K_{\mathbb{R}},\ m\in M_{\mathbb{R}},\ u\in U\cap M\\ m\cdot \mathfrak{q}_{\mathfrak{m}}=u\cdot \mathfrak{q}_{\mathfrak{m}}}}k\cdot [m\cdot A_{\mathbb{R}}(N_P)_{\mathbb{R}}\cdot \lambda+u\cdot \rho_{\mathfrak{l}}].
\]
Note $m\cdot \mathfrak{q}_{\mathfrak{m}}=u\cdot \mathfrak{q}_{\mathfrak{m}}$ implies $u^{-1}m\in N_M(\mathfrak{q}_{\mathfrak{m}})=Q_M$ the corresponding parabolic subgroup. Since $Q_M\subset Q:=N_G(\mathfrak{q}_{\lambda})$, we note $(kmn)\cdot \mathfrak{q}_{\lambda}=(ku)\cdot \mathfrak{q}_{\lambda}$ whenever $k\in K_{\mathbb{R}}$, $m\in M_{\mathbb{R}}$, $u\in U\cap M$, $n\in A_{\mathbb{R}}(N_P)_{\mathbb{R}}$, and $m\cdot \mathfrak{q}_{\mathfrak{m}}=u\cdot \mathfrak{q}_{\mathfrak{m}}$. Further, any $g\in G_{\mathbb{R}}$ can be written as $g=kmn$ with $k\in K_{\mathbb{R}}$, $m\in M_{\mathbb{R}}$, and $n\in A_{\mathbb{R}}(N_P)_{\mathbb{R}}$. Finally, if $g\cdot \mathfrak{q}_{\lambda}=u\cdot \mathfrak{q}_{\lambda}=u_0\cdot \mathfrak{q}_{\lambda}$ with $g\in G_{\mathbb{R}}$ and $u,u_0\in U$, then $u_0\in u(U\cap L)$. In particular, every such $u_0$ is in $K_{\mathbb{R}}\cdot (U\cap M)$. Putting all of these remarks together, one deduces
$$\widetilde{\mathcal{C}}(\mathcal{O},\Gamma,\mathfrak{q},\sigma_c)=\bigcup_{\substack{g\in G_{\mathbb{R}},\ u\in U\\ g \cdot \mathfrak{q}_{\lambda}=u\cdot \mathfrak{q}_{\lambda}}} \left(g\cdot \lambda+u\cdot \rho_{\mathfrak{l}}\right)=\mathcal{C}(\mathcal{O},\Gamma,\mathfrak{q},\sigma_c).$$

To complete the argument, we must check that the orientation defined above on $\widetilde{\mathcal{C}}(\mathcal{O},\Gamma,\mathfrak{q},\sigma_c)$ agrees with the orientation on $\mathcal{C}(\mathcal{O},\Gamma,\mathfrak{q},\sigma_c)$ defined in Section \ref{subsec:Rossmann_characters}. To check this, observe that the projection 
$$\varpi_{M_{\mathbb{R}}}\colon \mathcal{C}(\mathcal{O}^{M_{\mathbb{R}}},\Gamma^{M_{\mathbb{R}}},\mathfrak{q}_{\mathfrak{m}},\sigma_c|_M)\rightarrow \mathcal{O}^{M_{\mathbb{R}}}$$
gives rise to a projection
\[\begin{CD}
\widetilde{\mathcal{C}}(\mathcal{O},\Gamma,\mathfrak{q},\sigma_c)\simeq K_{\mathbb{R}}\times_{K_{\mathbb{R}}\cap M_{\mathbb{R}}} [\mathcal{C}(\mathcal{O}^{M_{\mathbb{R}}},\Gamma^{M_{\mathbb{R}}},\mathfrak{q}_{\mathfrak{m}},\sigma_c|_M)+\lambda_n+\sqrt{-1}(\mathfrak{g}_{\mathbb{R}}/\mathfrak{p}_{\mathbb{R}})^*]\\
@V\op{id}\times \varpi_{M_{\mathbb{R}}}\times \op{id} VV\\
K_{\mathbb{R}}\times_{K_{\mathbb{R}}\cap M_{\mathbb{R}}} [\mathcal{O}^{M_{\mathbb{R}}}+\lambda_n+\sqrt{-1}(\mathfrak{g}_{\mathbb{R}}/\mathfrak{p}_{\mathbb{R}})^*]\simeq \mathcal{O}.
\end{CD}\]
In fact, one sees that after composing with the isomorphism $\widetilde{\mathcal{C}}(\mathcal{O},\Gamma,\mathfrak{q},\sigma_c)\simeq \mathcal{C}(\mathcal{O},\Gamma,\mathfrak{q},\sigma_c)$, this is precisely the projection $\varpi\colon \mathcal{C}(\mathcal{O},\Gamma,\mathfrak{q},\sigma_c)\rightarrow \mathcal{O}$ that was utilized in Section \ref{subsec:Rossmann_characters} to defined the orientation on 
$\mathcal{C}(\mathcal{O},\Gamma,\mathfrak{q},\sigma_c)$. Utilizing $\op{id}\times \varpi\times \op{id}$, we observe that the orientation on $\widetilde{\mathcal{C}}(\mathcal{O},\Gamma,\mathfrak{q},\sigma_c)$ is induced from the orientation defined on the fibers of 
$$\varpi_{M_{\mathbb{R}}}\colon \mathcal{C}(\mathcal{O}^{M_{\mathbb{R}}},\Gamma^{M_{\mathbb{R}}},\mathfrak{q}_{\mathfrak{m}},\sigma_c|_M)\longrightarrow \mathcal{O}^{M_{\mathbb{R}}}$$
and the orientation on the base $K_{\mathbb{R}}\times_{K_{\mathbb{R}}\cap M_{\mathbb{R}}} [\mathcal{O}^{M_{\mathbb{R}}}+\lambda_n+\sqrt{-1}(\mathfrak{g}_{\mathbb{R}}/\mathfrak{p}_{\mathbb{R}})^*]\simeq \mathcal{O}$, which in turn arises from the orientation on the coadjoint orbit $\mathcal{O}^{M_{\mathbb{R}}}$ together with the orientation on $\mathfrak{k}_{\mathbb{R}}/(\mathfrak{k}_{\mathbb{R}}\cap \mathfrak{m}_{\mathbb{R}})
 \times (\mathfrak{g}_{\mathbb{R}}/\mathfrak{p}_{\mathbb{R}})^*$
 given above. 
The fibers of $\varpi$ and $\varpi_{M_{\mathbb{R}}}$ are naturally isomorphic and the corresponding orientations were defined analogously. Hence, it is enough to check that orientation on
\[K_{\mathbb{R}}\times_{K_{\mathbb{R}}\cap M_{\mathbb{R}}} [\mathcal{O}^{M_{\mathbb{R}}}+\lambda_n+\sqrt{-1}(\mathfrak{g}_{\mathbb{R}}/\mathfrak{p}_{\mathbb{R}})^*]\]
defined by the orientations on $\mathcal{O}^{M_{\mathbb{R}}}$ and $\mathfrak{k}_{\mathbb{R}}/(\mathfrak{k}_{\mathbb{R}}\cap \mathfrak{m}_{\mathbb{R}})
 \times (\mathfrak{g}_{\mathbb{R}}/\mathfrak{p}_{\mathbb{R}})^*$ agrees with the orientation on the coadjoint orbit $\mathcal{O}$ defined in Section \ref{subsec:Rossmann_characters}. In fact, it follows from the argument on page 379 of \cite{Ro84} that the canonical top-dimensional form on $\mathcal{O}$ is equal to the top-dimensional form on $K_{\mathbb{R}}\times_{K_{\mathbb{R}}\cap M_{\mathbb{R}}} [\mathcal{O}^{M_{\mathbb{R}}}+\lambda_n+\sqrt{-1}(\mathfrak{g}_{\mathbb{R}}/\mathfrak{p}_{\mathbb{R}})^*]$ induced from the forms on $\mathcal{O}^{M_{\mathbb{R}}}$ and $\mathfrak{k}_{\mathbb{R}}/(\mathfrak{k}_{\mathbb{R}}\cap \mathfrak{m}_{\mathbb{R}})
 \times (\mathfrak{g}_{\mathbb{R}}/\mathfrak{p}_{\mathbb{R}})^*$. The lemma follows.
\end{proof}

\section{The Kashiwara-Schmid Correspondence and the Character Formula of Schmid-Vilonen}
\label{sec:KSV}

In this section, we recall a correspondence conjectured by Kashiwara \cite{Kas87} and proved by Kashiwara-Schmid \cite{KS94} between a bounded equivariant derived category of sheaves of complex vector spaces on the flag variety and a bounded derived category of nuclear Fr\'{e}chet $G_{\mathbb{R}}$-representations. By comparing a result of Wong (see Section \ref{subsec:max_globalization}) with this correspondence, we will identify an object of the bounded equivariant derived category that corresponds to our representation $\pi(\mathcal{O}^{(M_{\mathbb{R}})_e},\Gamma^{(M_{\mathbb{R}})_e})$ studied in Section \ref{sec:representations}. Then we recall the character formula of Schmid-Vilonen \cite{SV98} which relates the character of the representation to the characteristic cycle of this sheaf.

\subsection{The Bounded Equivariant Derived Category}
\label{subsec:equivariant_derived}

Before recalling the notion of an equivariant derived category, we recall the notion of an equivariant sheaf. Suppose $X$ is a smooth manifold with a continuous action of a Lie group $G$. Intuitively, an equivariant sheaf on $X$ is a sheaf $\mathcal{F}$ together with a collection of sheaf isomorphisms $l_g\colon \mathcal{F}(\cdot)\rightarrow \mathcal{F}(g\cdot)$ that continuously vary in $g\in G$ satisfying $l_{gh}=l_gl_h$, $l_e=\op{Id}$, and for which sections of $\mathcal{F}$ are locally $l_g$-invariant for $g\in G$ near the identity. (For a more precise definition, see page 2 of \cite{BL94}).

We say $X$ is a \emph{free} $G$-space if the stabilizer of every point $x\in X$ is the identity $e\in G$ and if $q\colon X\rightarrow \overline{X}:=X/G$ is a locally trivial $G$-equivariant fibration with fibers equivariantly isomorphic to $G$. The second condition means that for every $x\in X$, there exists a $G$-invariant, open subset $x\in U\subset X$ and a $G$-equivariant isomorphism $$G\times \overline{U}\stackrel{\sim}{\longrightarrow} U$$ 
where $\overline{U}:=U/G$. 

Bernstein and Luntz point out (see Lemma 0.3 on page 3 of \cite{BL94}) that the inverse image functor gives a natural correspondence between sheaves of $\mathbb{C}$-vector spaces on $\overline{X}=X/G$ and $G$-equivariant sheaves of $\mathbb{C}$-vector spaces on $X$.
Let $\op{Sh}(\overline{X})$ denote the category of sheaves of complex vector spaces on $\overline{X}$, and let $\op{Sh}_G(X)$ denote the category of $G$-equivariant sheaves on $X$. If $X$ is a free space, then we may identify $D^b(\op{Sh}_G(X))$ with $D^b(\op{Sh}(\overline{X}))$. However, if $X$ is a $G$-space that is not free, then we may not have such a correspondence between equivariant sheaves on $X$ and sheaves on $\overline{X}:=X/G$. 

Let $I=[a,b]\subset \mathbb{Z}$ be an interval. If $X$ is any topological space, denote by $D^I(X)$ the full subcategory of the derived category of sheaves of complex vector spaces on $X$ consisting of objects $M$ satisfying $H^i(M)=0$ for all $i\notin I$. If $X$ is a smooth manifold with a smooth $G$-action and $Y$ is a smooth manifold and a free $G$-space, then $X\times Y$ is a free $G$-space. Consider the projections 
$$p\colon X\times Y\rightarrow X,\quad q\colon X\times Y\rightarrow \overline{X\times Y}.$$ 
Motivated by the above situation, Bernstein and Luntz make the following definition (see Section 2 of \cite{BL94}). The category $D_G^I(X,X\times Y)$ is the collection of triples $(\mathcal{F}_X,\overline{\mathcal{F}},\beta)$ for which $\mathcal{F}_X\in D^I(X)$, $\overline{\mathcal{F}}\in D^I(\overline{X\times Y})$, and 
$$\beta\colon p^{-1}\mathcal{F}_X\stackrel{\sim}\longrightarrow q^{-1}\overline{\mathcal{F}}.$$
A morphism $\gamma=(\gamma_X,\overline{\gamma})$ between $(\mathcal{F}_X,\overline{\mathcal{F}},\beta)$ and $(\mathcal{G}_X,\overline{\mathcal{G}},\alpha)$ is a pair of maps
$$\gamma_X\colon \mathcal{F}_X\rightarrow \mathcal{G}_X,\quad \overline{\gamma}\colon \overline{\mathcal{F}}\rightarrow \overline{\mathcal{G}}$$
satisfying
$$\alpha \circ p^{-1}\gamma_X=q^{-1}\overline{\gamma}\circ \beta.$$
There is a natural functor
$$D^I(\op{Sh}_G(X))\rightarrow D^I_G(X,X\times Y)$$
by $\mathcal{F}\mapsto (\mathcal{F},q_*^Gp^{-1}\mathcal{F},\beta)$ where $\beta$ is the natural isomorphism. A topological space $Y$ is called \emph{n-acyclic} if $H^0(Y,\mathbb{C})=\mathbb{C}$, $H^i(Y,\mathbb{C})=0$ if $i=1,\ldots,n$, and $H^i(Y,\mathbb{C})$ is finite-dimensional for all $i$. Bernstein and Luntz show that if $X$ is a free $G$-space, and $Y$ is $n$-acyclic with $n\geq |I|$, then the above map yields an equivalence (Proposition 2.1.1 on page 18 of \cite{BL94})
$$D^I(\op{Sh}_G(X))\simeq D^I_G(X,X\times Y).$$
Even when $X$ is not free, the categories $D^I_G(X,X\times Y)$ are naturally equivalent for all $Y$ $n$-acyclic. 

Let $G_{\mathbb{R}}$ be a component group of a real, linear algebraic group with connected complexification $G$. Then there exists an $n$-acyclic complex algebraic $G$-space $Y$ for every $n$ that is also a free $G_{\mathbb{R}}$-space (see Lemma 3.1 on page 34 of \cite{BL94}). We thus define $D^I_{G_{\mathbb{R}}}(X)$ to be the category $D^I_{G_\mathbb{R}}(X, X\times Y)$ for $Y$ $n$-acyclic with $n\geq |I|$, which is independent of $Y$ up to equivalence. Taking the direct limit over the indices $I$ yields the equivariant bounded derived category of sheaves of complex vector spaces on $X$, $D_{G_{\mathbb{R}}}^b(X)$ (see page 19 of \cite{BL94}; also see the descriptions on page 41 of \cite{Kas08} and pages 12-13 of \cite{KS94}).

There exists a functor
$$D^b(\op{Sh}_{G_{\mathbb{R}}}(X))\rightarrow D^b_{G_{\mathbb{R}}}(X).$$
To define it, take $\mathcal{F}\in D^b(\op{Sh}_{G_{\mathbb{R}}}(X))$, find $I\subset \mathbb{Z}$ with $\mathcal{F}\in D^I(\op{Sh}_{G_{\mathbb{R}}}(X))$, and for every $J\supset I$, choose $Y$ a $|J|$-acyclic complex algebraic $G$-space that is a free $G_{\mathbb{R}}$-space, apply the map
$$D^J(\op{Sh}_{G_{\mathbb{R}}}(X))\rightarrow D^J_{G_{\mathbb{R}}}(X,X\times Y),$$
and then take the limit to obtain an object in $D_{G_{\mathbb{R}}}^b(X)$. Moreover, there is the forgetful functor
$$D^b_{G_{\mathbb{R}}}(X)\rightarrow D^b(X)$$
obtained by taking an object of $D^b_{G_{\mathbb{R}}}(\op{Sh}(X))$, representing it by 
$$(\mathcal{F}_X,\overline{\mathcal{F}},\beta)\in D^I_{G_{\mathbb{R}}}(X,X\times Y)$$
 for some $I$, and then mapping it to $\mathcal{F}_X$.

If $f\colon X\rightarrow Y$ is a real algebraic, $G_{\mathbb{R}}$-equivariant map between real algebraic $G_{\mathbb{R}}$-spaces, then we have functors (see page 35 of \cite{BL94}) 
$$f_*, f_!\colon D^b_{G_{\mathbb{R}}}(X)\rightarrow D^b_{G_{\mathbb{R}}}(Y),\quad f^*,f^!\colon D^b_{G_{\mathbb{R}}}(Y)\rightarrow D^b_{G_{\mathbb{R}}}(X).$$
We work with $D^b_{G_{\mathbb{R}}}(X)$ instead of $D^b(\op{Sh}_{G_{\mathbb{R}}}(X))$ due to the existence of these functors.
All of these functors commute with the forgetful functor (see pages 35-36 of \cite{BL94}). In addition, if $H_{\mathbb{R}}\subset G_{\mathbb{R}}$ is a real, linear algebraic normal subgroup acting freely on $X$, then there is a natural equivalence of categories (see pages 26-27 of \cite{BL94})
\begin{equation}\label{eq:equivariant_equivalence}
D^b_{G_{\mathbb{R}}}(X)\stackrel{\sim}\longrightarrow D^b_{G_{\mathbb{R}}/H_{\mathbb{R}}}(H_{\mathbb{R}}\backslash X).
\end{equation}
\bigskip

Let $G$ be a connected, complex reductive algebraic group with Lie algebra $\mathfrak{g}$, and let $X$ be the flag variety of Borel subalgebras of $\mathfrak{g}$. A sheaf $\mathcal{F}$ on $X$ is $\mathbb{R}$-\emph{constructible} if 
\begin{enumerate}[(a)]
\item $\mathcal{F}_x$ is finite-dimensional for all $x\in X$
\item There exists a finite family of locally closed subsets $\{Z_{\alpha}\}$ of $X$ such that 
\begin{enumerate}[(i)]
\item $X=\sqcup_{\alpha} Z_{\alpha}$
\item $Z_{\alpha}$ is a subanalytic subset of $X$ (see page 327 of \cite{KS02} for a definition of subanalytic)
\item $\mathcal{F}|_{Z_{\alpha}}$ is a locally constant sheaf of finite rank on $Z_{\alpha}$ and every $\alpha$.  
\end{enumerate}
\end{enumerate}
The $\{Z_{\alpha}\}$ are called a \emph{stratification} of $X$. A natural stratification arises by taking a component group of  a real form $G_{\mathbb{R}}\subset G$ and letting $\{Z_{\alpha}\}$ denote the $G_{\mathbb{R}}$-orbits on $X$. In practice, this will be the only stratification considered in the sequel.

An object $\mathcal{F}\in D^b(X)$ is $\mathbb{R}$-constructible if the cohomology sheaves $H^j(\mathcal{F})$ are all $\mathbb{R}$-constructible. An object $\mathcal{F}\in D^b_{G_{\mathbb{R}}}(X)$ is $\mathbb{R}$-constructible if the image of the forgetful functor $\mathcal{F}_X\in D^b(X)$ is $\mathbb{R}$-constructible. We denote by $D^b_{G_{\mathbb{R}}, \mathbb{R}\mathchar`-c}(X)$ the full subcategory of $D^b_{G_{\mathbb{R}}}(X)$ consisting of $\mathbb{R}$-constructible objects.

\subsection{The Kashiwara-Schmid Correspondence}
\label{subsec:Kashiwara}

For each $x\in X$, let $B_x$ (resp.\ $\mathfrak{b}_x$) denote the corresponding Borel subgroup (resp.\ Borel subalgebra). Let $C$ (resp.\ $\mathfrak{c}$) denote the universal Cartan subgroup (resp.\ Cartan subalgebra) of $G$ (resp.\  $\mathfrak{g}$). In particular, there are natural isomorphisms $C\simeq B_x/[B_x,B_x]$ and $\mathfrak{c}\simeq \mathfrak{b}_x/[\mathfrak{b}_x,\mathfrak{b}_x]$ for every $x\in X$. For each $e^{\mu} \in \op{Hom}_{\mathbb{C}^{\times}}(C,\mathbb{C}^{\times})$ with differential $\mu\in \mathfrak{c}^*$, define $\mathcal{L}_{\mu}\rightarrow X$ to be the holomorphic line bundle on $X$ on which the action of the isotropy group $B_x$ on the fiber over $x\in X$ factors through $B_x/[B_x,B_x]\simeq C$ where it acts by $e^{\mu}$. Let $\Delta$ denote the collection of roots of $\mathfrak{c}$, and let $\Delta^+$ denote the collection of positive roots of $\mathfrak{c}$ determined by the root spaces in $[\mathfrak{b}_x,\mathfrak{b}_x]$. Let
$$\rho=\frac{1}{2}\sum_{\alpha\in \Delta^+}\alpha.$$
We call $\mu\in \mathfrak{c}^*$ \emph{integral} if it is the differential of a character in $\op{Hom}(C,\mathbb{C}^{\times})$. If $\mu\in \mathfrak{c}^*$ and $\mu+\rho$ is integral, then we define $\mathcal{O}_X(\mu)$ to be the sheaf of holomorphic sections of the holomorphic line bundle $\mathcal{L}_{\mu+\rho}\rightarrow X$.

Let $\mathbf{FN}_{G_{\mathbb{R}}}$ denote the category of Fr\'{e}chet, nuclear spaces with continuous $G_{\mathbb{R}}$-actions. This category is quasi-abelian and one may form the bounded, derived category $D^b(\mathbf{FN}_{G_{\mathbb{R}}})$ of $\mathbf{FN}_{G_{\mathbb{R}}}$ in much the same way as one forms the bounded, derived category of an abelian category (see Section 2 of \cite{Kas08} for details).  If $\mu\in \mathfrak{h}^*$ and $\mu+\rho$ is integral, then Kashiwara-Schmid give a functor
$$\mathbf{R}\operatorname{Hom}^{\text{top}}_{\mathbb{C}_X}(\ \cdot\ ,\mathcal{O}_X(\mu))\colon D^b_{G_{\mathbb{R}},\mathbb{R}\mathchar`-c}(X)\longrightarrow D^b(\mathbf{FN}_{G_{\mathbb{R}}}).$$
The complexes in $D^b(\mathbf{FN}_{G_{\mathbb{R}}})$ in the image of this functor are strict, meaning that the images of the differential maps are closed. The cohomology of this complex at degree $p$ is then denoted 
$$\op{Ext}^p(\ \cdot \ ,\mathcal{O}_X(\mu)).$$
Forgetting the topology and the $G_{\mathbb{R}}$-action, these functors are the usual $\mathbf{R}\operatorname{Hom}$ and $\op{Ext}^p$ functors composed with the forgetful functor from $D^b_{G_{\mathbb{R}}, \mathbb{R}\mathchar`-c}(X)$ into $D^b(X)$. To make these vector spaces into topological vector spaces we replace $\mathcal{O}_X(\mu)$ by $C^{\infty}$ Dolbeault complex (see Section 5 of \cite{KS94} and Sections 5, 9 of \cite{Kas08}). The equivariance of sheaves gives continuous $G_{\mathbb{R}}$-actions on vector spaces. We refer the reader to Kashiwara-Schmid \cite{KS94} and Kashiwara \cite{Kas08} for the technical details. 
In fact, Kashiwara-Schmid treat non-integral parameters using twisted sheaves, but we will avoid this more general case in our paper.
\bigskip

Let $(\mathcal{O},\Gamma)$ be an elliptic orbital parameter with $\mathcal{O}$ in the good range. We wish to relate the maximal globalization $\pi^{-\omega}(\mathcal{O},\Gamma)$ to the Kashiwara-Schmid correspondence above. To do this, we follow Wong \cite{Won99}. 
Adopting the notation of Section \ref{subsec:parabolics} and Section \ref{subsec:VZ_construction}, fix a maximally real admissible polarization $\{\mathfrak{q}_{\lambda}\}$ of $\mathcal{O}$. Let $S'$ be the $G_{\mathbb{R}}$-orbit consisting of all parabolics $\mathfrak{q}_{\lambda}$ with $\lambda\in \mathcal{O}$, and let $Y$ be the partial flag variety of all complex parabolics $G$-conjugate to some $\mathfrak{q}_{\lambda}$. For every Borel subalgebra $\mathfrak{b}\subset \mathfrak{g}$, it follows from standard arguments regarding root data and Weyl groups that there exists a unique parabolic $\mathfrak{q}\in Y$ containing $\mathfrak{b}$. The corresponding map $\mathfrak{b}\mapsto \mathfrak{q}$ yields a natural, $G$-equivariant fibration
\[\varpi\colon X\rightarrow Y.\]
We may put $S=\varpi^{-1}(S')$. More concretely, $S$ is the collection of Borel subalgebras $\mathfrak{b}\subset \mathfrak{g}$ such that $\mathfrak{b}\subset \mathfrak{q}_{\lambda}$ for some $\lambda\in \mathcal{O}$. Fix $\lambda\in \mathcal{O}$, let $\mathfrak{l}_{\mathbb{R}}=\mathfrak{g}_{\mathbb{R}}(\lambda)$, and find a fundamental real Cartan subalgebra $\mathfrak{h}_{\mathbb{R}}\subset \mathfrak{l}_{\mathbb{R}}\subset \mathfrak{g}_{\mathbb{R}}$ such that $\lambda\in \sqrt{-1}\mathfrak{h}_{\mathbb{R}}^*$. Let $L:=G(\lambda)$, $L_{\mathbb{R}}:=G_{\mathbb{R}}(\lambda)$, and let $X_L$ denote the flag variety for $L$. We must make an integrality assumption on the elliptic orbital parameter $(\mathcal{O},\Gamma)$. Let $\mathfrak{h}:=\mathfrak{h}_{\mathbb{R}}\otimes_{\mathbb{R}}\mathbb{C}$ be the complexification of $\mathfrak{h}_{\mathbb{R}}$ and assume 
\begin{equation}\label{eq:integral}
\lambda+\rho(\mathfrak{n})\ \text{is\ integral.}
\end{equation}

Next, to relate $\pi^{-\omega}(\mathcal{O},\Gamma)$ to the Kashiwara-Schmid correspondence, we first recall that Wong defines a map (see page 11 of \cite{Won99})
$$\op{Ind}_{L_{\mathbb{R}}}^{G_{\mathbb{R}}}\colon D^b_{L_{\mathbb{R}}, \mathbb{R}\mathchar`-c}(X_L)\rightarrow D^b_{G_{\mathbb{R}}, \mathbb{R}\mathchar`-c}(X)$$
in the following way. The fiber $\varpi^{-1}(\mathfrak{q}_{\lambda})$ can naturally be identified with $X_L$. This gives us an inclusion 
$$\iota\colon X_L\hookrightarrow S$$
and we have a corresponding functor 
\begin{equation}\label{eq:Wong_functor_1}
\iota_*\colon D^b_{L_{\mathbb{R}},\mathbb{R}\mathchar`-c}(X_L)\rightarrow D^b_{L_{\mathbb{R}},\mathbb{R}\mathchar`-c}(S).
\end{equation}
Next, we wish to map $L_{\mathbb{R}}$-equivariant objects on $S$ to $G_{\mathbb{R}}$-equivariant objects on $S$. This is done by considering the following diagram
$$S\stackrel{pr}\leftarrow G_{\mathbb{R}}\times S\stackrel{q}\rightarrow G_{\mathbb{R}}\times_{L_{\mathbb{R}}} S\stackrel{a}\rightarrow S$$ 
where $pr$ denotes the projection onto the second factor, $q$ denotes the quotient map, and $a$ denotes the action map.
We let $L_{\mathbb{R}}$ act on $S$ by letting $L_{\mathbb{R}}\subset G_{\mathbb{R}}$ act in the usual way on Borel subalgebras of $\mathfrak{g}$, and we let $G_{\mathbb{R}}\times L_{\mathbb{R}}$ act on $G_{\mathbb{R}}\times S$ by 
$$(g,l)\cdot (x,\mathfrak{b}):=(gxl^{-1},\op{Ad}(l)\mathfrak{b}).$$
This action allows us to identify $$S\simeq (G_{\mathbb{R}}\times S)/G_{\mathbb{R}}$$
and utilize (\ref{eq:equivariant_equivalence}) to give an equivalence of categories
\begin{equation}\label{eq:Wong_functor_2.1}
D^b_{L_{\mathbb{R}}, \mathbb{R}\mathchar`-c}(S)\stackrel{\sim}\rightarrow D^b_{G_{\mathbb{R}}\times L_{\mathbb{R}}, \mathbb{R}\mathchar`-c}(G_{\mathbb{R}}\times S).
\end{equation}
Next, 
$$G_{\mathbb{R}}\times_{L_{\mathbb{R}}}S:= (G_{\mathbb{R}}\times S)/L_{\mathbb{R}}$$
which induces an equivalence of categories (see \ref{eq:equivariant_equivalence})
\begin{equation}\label{eq:Wong_functor_2.2}
D^b_{G_{\mathbb{R}}\times L_{\mathbb{R}}, \mathbb{R}\mathchar`-c}(G_{\mathbb{R}}\times S)\stackrel{\sim}\rightarrow D^b_{G_{\mathbb{R}}, \mathbb{R}\mathchar`-c}(G_{\mathbb{R}}\times_{L_{\mathbb{R}}} S).
\end{equation}
Next, one can push forward by $a$, the action map, to obtain 
\begin{equation}\label{eq:Wong_functor_2.3}
a_*\colon D^b_{G_{\mathbb{R}}, \mathbb{R}\mathchar`-c}(G_{\mathbb{R}}\times_{L_{\mathbb{R}}} S)\rightarrow D^b_{G_{\mathbb{R}}, \mathbb{R}\mathchar`-c}(S).
\end{equation}
Composing the three functors (\ref{eq:Wong_functor_2.1}), (\ref{eq:Wong_functor_2.2}), (\ref{eq:Wong_functor_2.3}), one obtains a ``change of groups'' functor
\begin{equation}\label{eq:Wong_functor_2}
\Gamma_{L_{\mathbb{R}}}^{G_{\mathbb{R}}}\colon D^b_{L_{\mathbb{R}}, \mathbb{R}\mathchar`-c}(S)\rightarrow D^b_{G_{\mathbb{R}}, \mathbb{R}\mathchar`-c}(S).
\end{equation}

Recall $S$ was defined as an open subset of $X$; hence, we have a natural inclusion $j\colon S\hookrightarrow X$. Applying the corresponding proper pushforward functor, we obtain
\begin{equation}\label{eq:Wong_functor_3}
j_!\colon D^b_{G_{\mathbb{R}}, \mathbb{R}\mathchar`-c}(S)\rightarrow D^b_{G_{\mathbb{R}}, \mathbb{R}\mathchar`-c}(X).
\end{equation}

Composing the maps (\ref{eq:Wong_functor_1}), (\ref{eq:Wong_functor_2}), and (\ref{eq:Wong_functor_3}), one obtains the functor
\begin{equation}\label{eq:Wong_functor}
\op{Ind}_{L_{\mathbb{R}}}^{G_{\mathbb{R}}}:=j_!\Gamma_{L_{\mathbb{R}}}^{G_{\mathbb{R}}}\iota_*\colon D^b_{L_{\mathbb{R}}, \mathbb{R}\mathchar`-c}(X_L)\rightarrow D^b_{G_{\mathbb{R}}, \mathbb{R}\mathchar`-c}(X).
\end{equation}
\bigskip

Now, consider the constant sheaf $\mathbb{C}_{X_L}$ on $X_L$. One immediately checks
\begin{equation}\label{eq:constant_sheaf_1}
\op{Ext}^0(\mathbb{C}_{X_L}, \mathcal{O}_{X_L}(\mathcal{L}_{\mu}))\simeq \op{Hom}(\mathbb{C}_{X_L}, \mathcal{O}_{X_L}(\mathcal{L}_{\mu}))\simeq \Gamma(X_L, \mathcal{L}_{\mu})
\end{equation}
for all $\mu\in\mathfrak{c}^*$ integral. Moreover, if $l=\dim_{\mathbb{C}}X_L$ and $w_0$ is the longest element of the Weyl group (of the root system of the universal Cartan $\mathfrak{c}$), then by the Borel-Weil-Bott Theorem
\begin{equation}\label{eq:constant_sheaf_2}
\op{Ext}^{l}(\mathbb{C}_{X_L}, \mathcal{O}_{X_L}(\mathcal{L}_{\mu}))\simeq H^{l}(X_L, \mathcal{O}_{X_L}(\mathcal{L}_{\mu}))\simeq H^0(X_L, \mathcal{O}_{X_L}(\mathcal{L}_{w_0(\mu-\rho_{\mathfrak{l}})+\rho_{\mathfrak{l}}})).
\end{equation}

Fix a choice of positive roots $\Delta^+(\mathfrak{l},\mathfrak{h})\subset \Delta(\mathfrak{l},\mathfrak{h})$, and define
\[\rho_{\mathfrak{l}}=\frac{1}{2}\sum_{\alpha\in \Delta^+(\mathfrak{l},\mathfrak{h})}\alpha.\] 
Write $\mathfrak{q}_{\lambda}=\mathfrak{l}\oplus \mathfrak{n}$ as in Section \ref{subsec:parabolics}, and define 
\[\Delta^+(\mathfrak{g},\mathfrak{h})=\Delta^+(\mathfrak{l},\mathfrak{h})\cup \Delta(\mathfrak{n},\mathfrak{h})\]
using our (arbitrarily chosen) fixed choice of positive roots $\Delta^+(\mathfrak{l},\mathfrak{h})$. There exists a unique isomorphism $\mathfrak{h}\simeq \mathfrak{c}$ between $\mathfrak{h}$ and the universal Cartan subalgebra $\mathfrak{c}$ such that the positive roots $\Delta^+(\mathfrak{g},\mathfrak{c})$ correspond to the positive roots $\Delta^+(\mathfrak{g},\mathfrak{h})$.

If $\mu=\lambda+\rho_{\mathfrak{l}}+\rho=\lambda+\rho(\mathfrak{n})+2\rho_{\mathfrak{l}}$, then 
\begin{align*}
 w_0\cdot (\mu-\rho_{\mathfrak{l}})+\rho_{\mathfrak{l}}
& = w_0\cdot (\lambda+\rho(\mathfrak{n})+\rho_{\mathfrak{l}})+\rho_{\mathfrak{l}}\\
& = \lambda+\rho(\mathfrak{n})+w_0\cdot \rho_{\mathfrak{l}}+\rho_{\mathfrak{l}}
 = \lambda+\rho(\mathfrak{n})
\end{align*}
where we used that $\lambda+\rho(\mathfrak{n})$ vanishes on the semisimple part of $\mathfrak{l}$ to deduce that it is fixed by $w_0$.

Combining this calculation with (\ref{eq:constant_sheaf_1}) and (\ref{eq:constant_sheaf_2}), we obtain
\begin{align*}
&\op{Ext}^{l}(\mathbb{C}_{X_L}, \mathcal{O}_{X_L}(\lambda+\rho(\mathfrak{n})+\rho_{\mathfrak{l}}))\\
& =\op{Ext}^{l}(\mathbb{C}_{X_L}, \mathcal{O}_{X_L}(\mathcal{L}_{\lambda+\rho(\mathfrak{n})+2\rho_{\mathfrak{l}}}))\\
& =\op{Ext}^0(\mathbb{C}_{X_L}, \mathcal{O}_{X_L}(\mathcal{L}_{\lambda+\rho(\mathfrak{n})}))
 =\Gamma(X_L, \mathcal{L}_{\lambda+\rho(\mathfrak{n})})
 =\Gamma_{\lambda}\otimes e^{\rho(\mathfrak{n})}
\end{align*}
where the last line denotes the one-dimensional, unitary representation of $L_{\mathbb{R}}$. A similar application of Borel-Weil-Bott yields
\[\op{Ext}^{p}(\mathbb{C}_{X_L},\mathcal{O}_{X_L}(\lambda+\rho(\mathfrak{n})+\rho_{\mathfrak{l}}))=0\ \text{if}\ p\neq l.\]
Notice that the integrality condition (\ref{eq:integral}) is necessary for the sheaf $\mathcal{O}_{X_L}(\lambda+\rho(\mathfrak{n})+\rho_{\mathfrak{l}})$ in the above calculation to be well-defined.

In particular, we have shown that the Kashiwara-Schmid correspondence with ``twist'' $\lambda+\rho(\mathfrak{n})+\rho_{\mathfrak{l}}$ maps the constant sheaf $\mathbb{C}_{X_L}$ to an object in the derived category whose cohomology in degree $l$ is the one-dimensional representation $\Gamma_{\lambda}\otimes e^{\rho(\mathfrak{n})}$ and whose cohomology vanishes in all other degrees. Then by Proposition 4.5 of \cite{Won99}, we deduce 
$$\op{Ext}^{l+s}(\op{Ind}_{L_{\mathbb{R}}}^{G_{\mathbb{R}}} \mathbb{C}_{X_L}, \mathcal{O}_X(\lambda+\rho_{\mathfrak{l}}))\simeq H^s(\mathcal{O},\mathcal{L}_{\Gamma})$$
where $s=\dim_{\mathbb{C}} (\mathfrak{n}\cap \mathfrak{k})$. As noted in Section \ref{subsec:max_globalization}, it follows from results in \cite{Won95} that
$H^s(\mathcal{O},\mathcal{L}_{\Gamma})\simeq \pi^{-\omega}(\mathcal{O},\Gamma)$. Therefore, to place the representations $\pi(\mathcal{O},\Gamma)$ within the Kashiwara-Schmid correspondence, it is enough to compute $\op{Ind}_{L_{\mathbb{R}}}^{G_{\mathbb{R}}}\mathbb{C}_{X_L}$. Now, $\iota_*\mathbb{C}_{X_L}=\mathbb{C}_S$ and the constant sheaf $\mathbb{C}_S$ remains unchanged under the change of groups functor. Therefore, $\op{Ind}_{L_{\mathbb{R}}}^{G_{\mathbb{R}}}\mathbb{C}_{X_L}=j_{!}\mathbb{C}_S$ where $j\colon S\hookrightarrow X$ is the inclusion.

Hence, we conclude
\begin{equation}\label{eq:ext1}
\op{Ext}^{l+s}(j_!\mathbb{C}_S, \mathcal{O}_X(\lambda+\rho_{\mathfrak{l}}))\simeq \pi^{-\omega}(\mathcal{O},\mathcal{L}_{\Gamma})
\end{equation}
if $s=\dim_{\mathbb{C}}(\mathfrak{n}\cap \mathfrak{k})$ and $l=\dim_{\mathbb{C}}X_L$. Whereas 
\begin{equation}\label{eq:ext2}
\op{Ext}^{p}(j_!\mathbb{C}_S, \mathcal{O}_X(\lambda+\rho_{\mathfrak{l}}))=0
\end{equation}
if $p\neq l+s$.

In passing, we also note that one can deduce the above equalities from the results of Section 6 of \cite{Won99} together with knowledge of a $K_{\mathbb{C}}$-equivariant sheaf on a partial flag variety corresponding to the Harish-Chandra module of $\pi(\mathcal{O},\Gamma)$ (see for instance page 81 of \cite{Bie90}).

\subsection{A Character Formula of Schmid-Vilonen}
\label{subsec:SV}

In the seminal paper \cite{SV98}, Schmid and Vilonen give a geometric character formula for every irreducible, admissible representation $\pi$ of a real, reductive group $G_{\mathbb{R}}$ with regular infinitesimal character. In order to state their result, we must recall a couple of additional notions.

If $X$ is a complex manifold of complex dimension $n$ with complex coordinates $z_1,\ldots,z_n$
 with $z_i=x_i+\sqrt{-1} y_i$, then 
\[dx_1\wedge dy_1 \wedge dx_2\wedge dy_2 \wedge \cdots \wedge dx_n \wedge dy_n\]
defines an orientation on $X$.
If $Z\subset X$ is a closed (real) submanifold, the orientation on $X$ induces an orientation
 on the conormal bundle $T_Z^* X$ (see \cite{SV96}).

Now, let $X$ be the flag variety for $G$, assume $\mathcal{S}\in D^b_{G_{\mathbb{R}},\mathbb{R}\mathchar`-c}(X)$ and apply the forgetful functor to obtain an object in $D^b_{\mathbb{R}\mathchar`-c}(X)$ that, by abuse of notation, we also call $\mathcal{S}$. All of the $\mathbb{R}$-constructible sheaves $\op{H}^p(\mathcal{S})$ are constructible with respect to the stratification of $X$ defined by the $G_{\mathbb{R}}$-orbits on $X$. If $Z\subset X$ is a $G_{\mathbb{R}}$-orbit, then the conormal bundle to $Z$ in $X$ is 
$$T_Z^*X=\{(z,\xi)\mid z\in Z,\ \xi\in (\mathfrak{g}/(\mathfrak{g}_{\mathbb{R}}+\mathfrak{b}_z))^*\}$$
where $\mathfrak{b}_z\subset \mathfrak{g}$ is the Borel subalgebra corresponding to $z\in Z$. And
\[T_{G_{\mathbb{R}}}^*X=\bigcup_{Z\in X/G_{\mathbb{R}}} T^*_ZX\]
denotes the union of the conormal bundles to $G_{\mathbb{R}}$-orbits $Z\subset X$. For each $Z\in X/G_{\mathbb{R}}$, write
\[T_Z^*X \setminus \bigcup_{\substack{Z'\in X/G_{\mathbb{R}}\\ Z'\neq Z}}\overline{T_{Z'}^*X}=\bigcup_{\alpha} \Lambda_{Z,\alpha}\]
as a finite union of connected components. Then the characteristic cycle of $\mathcal{S}$, denoted $\op{CC}(\mathcal{S})$, is a subanalytic, Lagrangian cycle in $T^*_{G_{\mathbb{R}}}X$. More precisely it is a linear combination of components of the form $\Lambda_{Z,\alpha}$ for $Z\in X/G_{\mathbb{R}}$. Since $\Lambda_{Z,\alpha}\subset T_Z^*X$ is open, by the above discussion it inherits an orientation if $\dim_{\mathbb{C}} X$ is even and, one must divide by a non-canonical choice of $\sqrt{-1}$ to orient $\Lambda_{Z,\alpha}$ if $\dim_{\mathbb{C}}X$ is odd. Therefore, $\op{CC}(\mathcal{S})$ is equipped with an orientation if $\dim_{\mathbb{C}}X$ is even and one must divide by a non-canonical choice of $\sqrt{-1}$ to obtain an orientation if $\dim_{\mathbb{C}}X$ is odd. The characteristic cycle was introduced by Kashiwara \cite{Kas85}. See Section 9.4 of \cite{KS02} and Section 2 of \cite{SV96} for expositions. 

If $\mathfrak{c}$ is the universal Cartan subalgebra, then, for each pair $(\mathfrak{h},\Delta^+)$ where $\mathfrak{h}\subset \mathfrak{g}$ is a Cartan subalgebra and $\Delta^+\subset \Delta(\mathfrak{g},\mathfrak{h})$ is a subset of positive roots, we obtain a canonical isomorphism $\iota_{\mathfrak{h},\Delta^+}\colon \mathfrak{h}\simeq \mathfrak{c}$ which preserves positivity. If $\xi \in \mathfrak{c}^*$ is a regular element in the dual of the universal Cartan subalgebra, then we obtain an element $\xi_{\mathfrak{h},\Delta^+}=\iota^{-1}_{\mathfrak{h},\Delta^+}(\xi)$. Define $\Omega_{\xi}$ to be  the regular coadjoint $G$-orbit consisting of all elements of the form $\xi_{\mathfrak{h},\Delta^+}$ for some pair $(\mathfrak{h},\Delta^+)$. The set $\Omega_{\xi}$ is also the inverse image of $\xi$ under the fibration $q\colon \mathfrak{g}^*\rightarrow \mathfrak{c}^*/W$ defined in (\ref{eq:fibration}).

Fix a maximal compact subgroup $U:=G^{\sigma_c}\subset G$ where $\sigma_c$ is an anti-holomorphic involution of $G$ commuting with $\sigma$. For every $x\in X$, there is a unique Cartan subalgebra $\mathfrak{h}_x\subset \mathfrak{b}_x$ fixed by $\sigma_c$. If $\Delta^+\subset \Delta(\mathfrak{g},\mathfrak{h}_x)$ is the collection of positive roots determined by $\mathfrak{b}_x$, then we define $\xi_x:=\xi_{\mathfrak{h}_x,\Delta^+}\in \mathfrak{h}_x^*\subset \mathfrak{g}^*$. Rossmann defined the \emph{twisted momentum map} (see for instance (3) on page 265 of \cite{Ro90})
$$\mu_{\xi}\colon T^*X\longrightarrow \Omega_{\xi}\subset \mathfrak{g}^*$$
by
$$(x,\eta)\mapsto \xi_x+\eta.$$
The map $\mu_{\xi}$ is a $U$-equivariant real analytic diffeomorphism.

Let $\mathbb{D}\colon D^b_{G_{\mathbb{R}},\mathbb{R}\mathchar`-c}(X)\rightarrow D^b_{G_{\mathbb{R}},\mathbb{R}\mathchar`-c}(X)$ denote the Verdier duality functor (see page 37 of \cite{BL94}). Suppose $\mathcal{S}\in D^b_{G_{\mathbb{R}},\mathbb{R}\mathchar`-c}(X)$ and $\xi+\rho$ is integral,
Assume that there exists $q\in\mathbb{Z}$ such that
\[\op{Ext}^q(\mathbb{D}\mathcal{S},\mathcal{O}_X(\xi))\simeq \pi^{-\omega}\]
and
\[\op{Ext}^p(\mathbb{D}\mathcal{S},\mathcal{O}_X(\xi))=0\]
if $p\neq q$. Then Schmid and Vilonen show (see (1.8) on page 4 of \cite{SV98}) 
$$\mathcal{F}[\mu_{\xi}(\op{CC}(\mathcal{S}))]=(-1)^q \theta(\pi).$$

Notice $\mu_{\xi}(\op{CC}(\mathcal{S}))$ inherits a canonical orientation from the orientation of $\op{CC}(\mathcal{S})$ if $\dim_{\mathbb{C}}X$ is even. On the other hand, the orientation on $\mu_{\xi}(\op{CC}(\mathcal{S}))$ depends on a non-canonical choice of $\sqrt{-1}$ if $\dim_{\mathbb{C}}X$ is odd. Therefore, as explained in Section \ref{subsec:Rossmann_characters}, the Fourier transform is well-defined.
       
Let us consider the special case where $\pi=\pi(\mathcal{O},\Gamma)$ and $(\mathcal{O},\Gamma)$ is an elliptic orbital parameter satisfying the integrality condition (\ref{eq:integral}) and for which $\mathcal{O}$ is in the good range (\ref{eq:good_range}). Fix $\lambda\in \mathcal{O}$ with $\lambda=-\sigma_c(\lambda)$.  Let $\mathfrak{q}_{\lambda}=\mathfrak{g}(\lambda)\oplus \mathfrak{n}$ be an admissible polarization.
Fix a $\sigma$ and $\sigma_c$-stable Cartan subalgebra $\mathfrak{h}\subset \mathfrak{g}(\lambda)=\mathfrak{l}$. Fix a choice of positive roots $\Delta^+(\mathfrak{l},\mathfrak{h})$ and define $\Delta^+:=\Delta(\mathfrak{n},\mathfrak{h})\cup \Delta^+(\mathfrak{l},\mathfrak{h})$. By abuse of notation, we will often write $\lambda+\rho_{\mathfrak{l}}$ for the element $\iota_{\mathfrak{h},\Delta^+}(\lambda+\rho_{\mathfrak{l}})\in \mathfrak{c}^*$. This element is independent of the above choices.

By (\ref{eq:ext1}) and (\ref{eq:ext2}), we deduce that we may take $\mathcal{S}=j_!\mathbb{C}_S$ to be the sheaf corresponding to $\pi(\mathcal{O},\Gamma)$ where $q=s+l$. Applying Verdier duality, we obtain 
\[\mathbb{D}j_!\mathbb{C}_S\simeq Rj_*\mathbb{C}_S[2\dim_{\mathbb{C}}X]\]
(see Theorem 3.5.2 on page 37 of \cite{BL94}), and we deduce
\begin{equation}\label{eq:SV_character}
\mathcal{F}[\mu_{\lambda+\rho_{\mathfrak{l}}}(\op{CC}(Rj_*\mathbb{C}_S))]=(-1)^q\theta(\mathcal{O},\Gamma).
\end{equation}
Observe that the shift by the even number $2\dim_{\mathbb{C}}X$ does not change the characteristic cycle since the characteristic cycle is defined as an alternating sum of dimensions of certain local cohomologies in each degree (see Chapter 9 of \cite{KS02} and page 4 of \cite{SV96}). 

\section{Proof of the Character Formula}
\label{sec:proof}

In this section, we give a proof of Theorem \ref{thm:main}. In order to do so, we utilize ideas from Schmid-Vilonen \cite{SV96} and Bozicevic \cite{Boz02}, \cite{Boz08} to show that the characteristic cycle of the sheaf $Rj_*\mathbb{C}_S$ described in the previous section is homologous (up to sign) to the inverse image of our cycle $\mathcal{C}(\mathcal{O},\Gamma,\mathfrak{q},\sigma_c)$ under the twisted momentum map. Finally, since we work with untwisted sheaves, we need a coherent continuation argument to obtain the formula in full generality. 

\subsection{Integral case}
\label{subsec:homotopies}


The argument in this section is analogous to \cite[Section 7]{SV98},
  \cite[Section 3]{Boz02}, and \cite{Boz08}.
We retain the setting at the end of Section~\ref{subsec:SV}.
In particular, $(\mathcal{O},\Gamma)$ is an elliptic orbital parameter
 with $\mathcal{O}$ in the good range (\ref{eq:good_range}).
Let $\lambda\in \mathcal{O}$ such that $\lambda=-\sigma_c(\lambda)$
 and assume the integrality condition (\ref{eq:integral}). 
Put $\xi:=\lambda+\rho_{\mathfrak{l}}$.
To prove (\ref{eq:elliptic_character_formula}),
 it is enough to show that the cycle
 $\mu_{\xi}(\op{CC}(Rj_*\mathbb{C}_S))$ is homologous to
 $(-1)^q \mathcal{C}(\mathcal{O},\Gamma,\mathfrak{q},\sigma_c)$
 as an admissible cycle.

We say a chain $\mathcal{C}$ in $T^* X$ is \emph{$\mathbb{R}$-bounded}
 if $\op{Re} \mu_{\xi}(\op{supp} \mathcal{C})$ is bounded.

\begin{lemma}[{\cite[Lemma 3.19]{SV98}}]\label{lem:homologous_chains}
Let $\mathcal{C}_0$ and $\mathcal{C}_1$ be $\mathbb{R}$-bounded
 semi-algebraic $2m$-cycles in $T^* X$.
Suppose that there exists an $\mathbb{R}$-bounded
 semi-algebraic $(2m+1)$-chain $\widetilde{\mathcal{C}}$ in $T^* X$ such that
 $\partial \widetilde{\mathcal{C}}=\mathcal{C}_1-\mathcal{C}_0$.
Then 
\begin{align*}
\int_{\mathcal{C}_0} \mu_{\xi}^* \omega^m
= \int_{\mathcal{C}_1} \mu_{\xi}^* \omega^m.
\end{align*}
\end{lemma}

We will apply the lemma by constructing
 a $(2m+1)$-chain $\widetilde{\mathcal{C}}$
 such that $\partial \widetilde{\mathcal{C}}=\mathcal{C}_1-\mathcal{C}_0$,
 where $\mathcal{C}_0=\op{CC}(Rj_*\mathbb{C}_S)$ and 
 $\mathcal{C}_1=(-1)^q
 \mu_\xi^{-1}(\mathcal{C}(\mathcal{O},\Gamma,\mathfrak{q},\sigma_c))$.

The cycle $\op{CC}(Rj_*\mathbb{C}_S)$ can be described by
 using the open embedding theorem by Schmid-Vilonen~\cite[Theorem 4.2]{SV96}.
Let $V$ be the irreducible finite-dimensional
 representation of $G$ with highest weight
 $\lambda+\rho(\mathfrak{n})$.
The representation space $V$ is realized by the Borel-Weil theorem
 as regular functions $F: G\to \mathbb{C}$ satisfying
 $F(gp)=e^{-\lambda-\rho(\mathfrak{n})}(p)F(g)$
 for $g\in G$ and $p\in \overline{Q}_\lambda$.
Here, $\overline{Q}_\lambda$ is the opposite parabolic subgroup
 of $Q_\lambda=N_{G}(\mathfrak{q}_\lambda)$.
Since $Q_\lambda$ is stable by the involution $\theta=\sigma\sigma_c$, 
 the action of $\theta$ on $G$ induces an involution on $V$,
 which we also denote by $\theta$.
Then $\theta(g)\cdot \theta(v) = \theta(g\cdot v)$ for $g\in G$ and $v\in V$.
Fix a $U$-invariant Hermitian inner product $h$ on $V$ and 
 define another hermitian form $h_r(v,w):=h(v,\theta(w))$.
Then $h_r$ is $G_{\mathbb{R}}$-invariant:
\begin{align*}
h_r(gv,gw)=h(gv,\theta(g)\theta(w))
=h(v,\sigma_c(g)^{-1}\theta(g)\theta(w))
=h(v,\theta(w))
=h_r(v,w)
\end{align*}
for $g\in G_{\mathbb{R}}$.
Define a real algebraic function $f$ on $X$ by
\begin{align*}
 f(x)=\frac{h_r(v,v)}{h(v,v)},
 \quad v\in V^{[\mathfrak{b}_x,\mathfrak{b}_x]}\setminus 0.
\end{align*}
Since the highest weight space $V^{[\mathfrak{b}_x,\mathfrak{b}_x]}$ is one-dimensional, $f$ is well-defined. Let $y\in Y$ and $\mathfrak{q}_y=\mathfrak{l}_y\oplus\mathfrak{n}_y$ the corresponding parabolic subalgebra of $\mathfrak{g}$. By \cite[Theorem 5.104(a)]{Kna05} $V^{\mathfrak{n}_y}$ is an irreducible $\mathfrak{l}_y$-module with highest weight $\lambda+\rho(\mathfrak{n})$. But, since $\lambda+\rho(\mathfrak{n})$ vanishes on $[\mathfrak{l}_y,\mathfrak{l}_y]$, $V^{\mathfrak{n}_y}$ is one-dimensional and isomorphic to $V^{[\mathfrak{q}_y,\mathfrak{q}_y]}$. Let $f_Y$ be the function on $Y$ given by
\begin{align*}
 f_Y(y)=\frac{h_r(v,v)}{h(v,v)},
 \quad v\in V^{[\mathfrak{q}_y,\mathfrak{q}_y]}\setminus 0.
\end{align*}
Recall that we have a natural map $\varpi:X\to Y$. If $\varpi(x)=y$, then $V^{[\mathfrak{b}_x,\mathfrak{b}_x]}=V^{[\mathfrak{q}_y,\mathfrak{q}_y]}$. 
Hence $f=f_Y\circ \varpi$. 

Let $y_0\in Y$ be the base point corresponding to $\mathfrak{q}_{\lambda}$.
A vector $v_0\in V^{[\mathfrak{q}_{\lambda},\mathfrak{q}_{\lambda}]}$
 is realized as a function $F$ on $G$ such that
 $F(pp')=e^{\lambda+\rho(\mathfrak{n})}(p)e^{-\lambda-\rho(\mathfrak{n})}(p')$
 for $p\in Q_{\lambda}$ and $p'\in\overline{Q}_{\lambda}$.
Therefore $\theta (v_0)=v_0$ and $f_Y(y_0)=1$.
For $g\in G_{\mathbb{R}}$,
\begin{align*}
f_Y(g y_0)
 = \frac{h_r(gv_0,gv_0)}{h(gv_0,gv_0)}
 = \frac{h(v_0,v_0)}{h(gv_0,gv_0)}>0.
\end{align*}

We next claim $f_Y(y)=0$ for $y\in \partial S$.
Replacing $y$ by $g y$ for $g\in G_\mathbb{R}$ if necessary,
 we may assume that $\mathfrak{q}_y$ contains
 a $\theta$-stable Cartan subalgebra
 $\mathfrak{h}'_{\mathbb{R}}\subset \mathfrak{g}_\mathbb{R}$.
Let $\lambda'\in (\mathfrak{h}')^*$ be the weight
 corresponding to $\lambda\in \mathfrak{c}^*$ by
 the isomorphism $\mathfrak{h}'\simeq \mathfrak{c}$.
Then $\mathfrak{h}'$ acts by $\lambda'$ on
 $V^{[\mathfrak{q}_y,\mathfrak{q}_y]}$.
Since $\lambda$ is elliptic, $\lambda'=-\sigma_c(\lambda')$.
If moreover $\lambda'=\theta(\lambda')$, then $\lambda'$
 takes purely imaginary values on $\frak{h}'_{\mathbb{R}}$.
But this implies $\mathfrak{g}_\mathbb{R}+\mathfrak{q}_y=\mathfrak{g}$
 and $G_{\mathbb{R}}\cdot y $ is open in $Y$,
 which contradicts $y\in \partial S$.
Therefore, $\lambda'\neq \theta(\lambda')$ and
 $V^{[\mathfrak{q}_y,\mathfrak{q}_y]}$
 and $\theta (V^{[\mathfrak{q}_y,\mathfrak{q}_y]})$ are
 in the different $\mathfrak{h}'$-weight spaces.
Hence $h_r(v,v)=h(v,\theta(v))=0$
 for $v\in V^{[\mathfrak{q}_y,\mathfrak{q}_y]}$, proving the claim.

The above statements on $f_Y$ imply 
 $f=f_Y\circ \varpi$ is positive on $S$
 and zero on the boundary $\partial S$.

Write $X^{\mathbb{R}}$ for the underlying real manifold of $X$.
In what follows, we identify the complex cotangent bundle $T^*X$ with
 the real contangent bundle $T^*X^{\mathbb{R}}$ by 
 the correspondence
 of $\partial \phi (x) \in T_x^*X$ and $d\phi(x) \in T^* X^{\mathbb{R}}$
 for a real-valued function $\phi$ on $X$.
Also, $T^*Y$ and $T^*Y^{\mathbb{R}}$ are identified in a similar way.
Define a map $\varphi:(0,1)\times S \to T^* X$ by $(t,x)\mapsto t (d \log f)_x$ and define a $(2m+1)$-chain $\widetilde{\mathcal{C}}$ in $T^* X$ by the image of $\varphi$. For the orientation, if $\dim S$ is even, then we take a product of the positive orientation on $(0,1)$ and the natural orientation on $S$ coming from the complex structure (see Section \ref{subsec:SV} for a definition). If $\dim S$ is odd, then we take the product of the positive orientation on $(0,1)$ with the non-canonical choice of orientation on $S$ which depends on a choice of $\sqrt{-1}$. By \cite[Proposition 3.25 and Theorem 4.2]{SV96}, we have 
\[\partial \widetilde{\mathcal{C}} =[d \log (f|_S)]-\op{CC}(Rj_*\mathbb{C}_S).\]
As a result, the following lemma implies (\ref{eq:elliptic_character_formula}) under the integrality condition (\ref{eq:integral}). In the next section, we will remove the assumption (\ref{eq:integral}) and then Lemma \ref{lem:elliptic_reduction} will imply our main result, Theorem \ref{thm:main}. 

Recall that we put $s:=\dim (\mathfrak{n}\cap\mathfrak{k})$,
 $l:=\dim_{\mathbb{C}} X_L$, and $q:=s+l$.

\begin{lemma}\label{lem:dlogf}
In the setting above, the twisted momentum map $\mu_\xi : T^*X\to \Omega_{\xi}$ restricted to $d \log (f|_S)$ gives a real algebraic isomorphism $d \log (f|_S)\to \mathcal{C}(\mathcal{O},\Gamma,\mathfrak{q},\sigma_c)$. Moreover, the orientation is preserved if $q$ is even
 and reversed if $q$ is odd. 
\end{lemma}

\begin{proof}
The proof is analogous to \cite[(7.27a)]{SV98},
 \cite[Lemma 3.2]{Boz02}, and \cite{Boz08},
 but we include this for completeness. 
For $Z\in \mathfrak{g}$, we also write $Z$ for the corresponding vector field on $Y$.
Then by an isomorphism $T^*Y^{\mathbb{R}}\simeq T^*Y$, 
\begin{align*}
\langle d \log f_Y, Z\rangle (y)
&=\langle \partial \log f_Y, Z\rangle (y) \\
&= \frac{1}{2}\frac{d}{dt}
 \Bigl(\log f_Y (\exp (tZ) y) 
 -\sqrt{-1} \log f_Y (\exp (\sqrt{-1}tZ) y) \Bigr)\Bigr|_{t=0}
\end{align*}
for $y\in S'$.
If $y=g y_0 = uy_0$ with $g\in G_{\mathbb{R}}$ and $u\in U$,
 then 
\begin{align*}
&\frac{d}{dt} \log f_Y (\exp (tZ) y) \Bigr|_{t=0} \\
&= \frac{h_r(Zgv_0,gv_0)+h_r(gv_0,Zgv_0)}{h_r(gv_0,gv_0)}
 -\frac{h(Zuv_0,uv_0)+h(Zuv_0,uv_0)}{h(uv_0,uv_0)}\\
&=\frac{2 \op{Re} h_r(\op{Ad}(g^{-1})(Z)v_0,v_0)}{h_r(v_0,v_0)}
 -\frac{2 \op{Re} h(\op{Ad}(u^{-1})(Z)v_0,v_0)}{h(v_0,v_0)}\\
&=2 \op{Re} \langle \lambda, \op{Ad}(g^{-1})Z \rangle
 - 2 \op{Re} \langle \lambda, \op{Ad}(u^{-1})Z \rangle\\
&=2 \op{Re} \langle \op{Ad}^*(g)\lambda, Z \rangle
 - 2 \op{Re}  \langle \op{Ad}^*(u)\lambda, Z \rangle.
\end{align*}
Similarly,
\begin{align*}
&-\sqrt{-1}\frac{d}{dt} \log f_Y (\exp (\sqrt{-1}tZ) y) \Bigr|_{t=0}\\
&= 2 \sqrt{-1} \op{Im} \langle \op{Ad}^*(g)\lambda, Z \rangle
 - 2 \sqrt{-1} \op{Im}  \langle \op{Ad}^*(u)\lambda, Z \rangle.
\end{align*}
Therefore, 
\begin{align*}
\langle d \log f_Y, Z\rangle (y)= \langle g\cdot \lambda-u\cdot \lambda, Z\rangle.
\end{align*}
Under isomorphisms $T^*Y^{\mathbb{R}}\simeq T^*Y
 \simeq \{(y,\xi)\mid y\in Y,\ \xi\in (\mathfrak{g}/\mathfrak{q}_y)^*\}$,
 we obtain
\begin{align*}
d \log (f_Y|_{S'})
=\left\{(u\cdot y_0,\, g\cdot \lambda-u\cdot \lambda)
 \mid g\in G_{\mathbb{R}},\ u\in U,\ 
 g\cdot \mathfrak{q}_{\lambda}=u\cdot \mathfrak{q}_{\lambda}\right\}.
\end{align*}
Fix $x_0\in S$ such that $\varpi(x_0)=y_0$.
Then it is easy to see that
\begin{align*}
d \log (f|_{S})
=\left\{(u\cdot x_0,\, g\cdot \lambda-u\cdot \lambda)
 \mid g\in G_{\mathbb{R}},\ u\in U,\ 
 g\cdot \mathfrak{q}_{\lambda}=u\cdot \mathfrak{q}_{\lambda}\right\}.
\end{align*}
Since
 $\mu_{\xi}(u\cdot x_0,\, \eta)
 = \eta+u\cdot (\lambda+\rho_{\mathfrak{l}})$, we have
\begin{align*}
\mu_{\xi}(d \log (f|_{S}))
=\left\{g\cdot \lambda+u\cdot \rho_{\mathfrak{l}}
 \mid g\in G_{\mathbb{R}},\ u\in U,\ 
 g\cdot \mathfrak{q}_{\lambda}=u\cdot \mathfrak{q}_{\lambda}\right\}
=\mathcal{C}(\mathcal{O},\Gamma,\mathfrak{q},\sigma_c).
\end{align*}
Hence the first assertion is proved.

To consider the orientation,
 let $\mu_{Y,\lambda}:T^*Y \to \mathfrak{g}^*$
 be the twisted momentum map for the partial flag variety given by
\begin{align*}
\mu_{Y,\lambda}(u\cdot \mathfrak{q}_{\lambda},\eta)
 =u\cdot \lambda+\eta
\quad  \text{ for $u\in U$, $\eta\in (\mathfrak{g}/(u\cdot \mathfrak{q}_{\lambda}))^*$}.
\end{align*}
There is a commutative diagram
\[\begin{CD}
d\log (f|_S) @>\mu_{\xi}>>
  \mathcal{C}(\mathcal{O},\Gamma,\mathfrak{q},\sigma_c) \\
@V\varpi VV @VVV \\
d\log (f_Y|_{S'}) @>>\mu_{Y,\lambda}> \mathcal{O}  
\end{CD}\]
where the horizontal maps are isomorphisms. 
The isomorphisms $d\log(f|_S)\simeq S$ and $d\log(f_Y|_{S'})\simeq S'$
 together with the complex structures on $S$ and $S'$ orient $d\log(f|_S)$ and $d\log(f_Y|_{S'})$. 
The coadjoint orbit $\mathcal{O}$ is oriented by
 $\frac{\omega^k}{\sqrt{-1}^{k}}$,
 where $\omega$ is the canonical symplectic form and
 $2k=\dim_{\mathbb{R}} \mathcal{O}$.

We show that the orientation is preserved by the map $\mu_{Y,\lambda}$ exactly
 when $s$ is even.
For this, it is enough to compare the orientation at
 $y_0$ and at $\lambda \in \mathcal{O}$.
The tangent spaces $T_{y_0}S$ and $T_{\lambda}\mathcal{O}$ are identified with 
 $\mathfrak{g}_{\mathbb{R}}/\mathfrak{l}_{\mathbb{R}}$.
The complex structure on this space is given by the isomorphism
 $\mathfrak{g}_{\mathbb{R}}/\mathfrak{l}_{\mathbb{R}}
 \simeq \mathfrak{g}/\mathfrak{q}_{\lambda}$.
Let $\mathfrak{g}=\mathfrak{k}\oplus\mathfrak{g}^{-\theta}$
 and $\mathfrak{h}=\mathfrak{h}^{\theta} \oplus\mathfrak{h}^{-\theta}$ be 
 Cartan decompositions.
Since $\mathfrak{h}$ is a fundamental Cartan subalgebra,
 $\mathfrak{h}^{\theta}$ is a Cartan subalgebra of $\mathfrak{k}$.
We have root decompositions:
\begin{align*}
\mathfrak{k}=
 \mathfrak{h}^{\theta}\oplus
 \bigoplus_{\alpha\in\Delta(\mathfrak{k},\mathfrak{h}^{\theta})}
 \mathfrak{k}_{\alpha}
,\qquad
\mathfrak{g}^{-\theta}=
 \mathfrak{h}^{-\theta}\oplus
 \bigoplus_{\alpha\in\Delta(\mathfrak{g}^{-\theta},\mathfrak{h}^{\theta})}
 \mathfrak{g}^{-\theta}_{\alpha}.
\end{align*}
Correspondingly, the isomorphism 
$\mathfrak{g}_{\mathbb{R}}/\mathfrak{l}_{\mathbb{R}}
 \simeq \mathfrak{g}/\mathfrak{q}_{\lambda}$ is decomposed into
\begin{align*}
&(\mathfrak{k}_{\alpha} \oplus \mathfrak{k}_{-\alpha})
  \cap \mathfrak{k}_{\mathbb{R}}
\simeq \mathfrak{k}_{-\alpha}
 \text{\ for\ }
 \alpha\in\Delta(\mathfrak{n}\cap\mathfrak{k},\mathfrak{h}^{\theta})
\text{\ \ and }\\
&(\mathfrak{g}^{-\theta}_{\alpha} \oplus \mathfrak{g}^{-\theta}_{-\alpha})
  \cap \mathfrak{g}^{-\theta}_{\mathbb{R}}
\simeq \mathfrak{g}^{-\theta}_{-\alpha}
 \text{\ for\ }
 \alpha\in\Delta(\mathfrak{n}\cap\mathfrak{g}^{-\theta},\mathfrak{h}^{\theta}).
\end{align*}
This decomposition is orthogonal with respect to the symplectic form
$\omega(Z_1,Z_2)=\lambda([Z_1,Z_2])$.
Therefore, the two orientations on
 $\mathfrak{g}_{\mathbb{R}}/\mathfrak{l}_{\mathbb{R}}$
 induce those on each summand.
Then it is easy to see that the two orientations differ on 
$(\mathfrak{k}_{\alpha} \oplus \mathfrak{k}_{-\alpha})
 \cap \mathfrak{k}_{\mathbb{R}}$ and 
agree on
 $(\mathfrak{g}_{\alpha}^{-\theta} \oplus \mathfrak{g}_{-\alpha}^{-\theta})
 \cap \mathfrak{g}_{\mathbb{R}}$.
Hence $\mu_{Y,\lambda}$ preserves the orientation exactly
 when $s=\# \Delta(\mathfrak{n}\cap \mathfrak{k},\mathfrak{h}^{\theta})$
 is even.

Recall that $\mathcal{C}(\mathcal{O},\Gamma,\mathfrak{q},\sigma_c)$
 has an orientation defined by the orientations of $\mathcal{O}$
 and the fiber $\lambda+\mathcal{O}_{\rho}^{U_{\lambda}}$. 
The fiber of $\varpi : X\to Y$ is isomorphic to $X_L$ and the restriction of
 $\mu_{\xi}$ gives an isomorphism
 $X_L\xrightarrow{\mu_{\xi}} \lambda+\mathcal{O}_{\rho}^{U_{\lambda}}$. 
Recall from Section~\ref{subsec:SV} that the complex manifold $X_L$ has a natural orientation.
The map $\mu_{\xi}|_{X_L}$ is identified with the twisted moment map for $U_{\lambda}$
 and an argument as above shows
 it preserves orientation if and only if $\dim_{\mathbb{C}} X_L$ is even. 
Therefore, the map between total spaces
 $d\log (f|_S) \xrightarrow{\mu_{\xi}} \mathcal{C}(\mathcal{O},\Gamma,\mathfrak{q},\sigma_c)$
 preserves orientation if and only if $q=s+l$ is even. 
\end{proof}

One can drop the integrality condition (\ref{eq:integral}) on $\lambda$ 
 for the claim that
 $\mu_{\lambda+\rho_{\mathfrak{l}}}(\op{CC}(Rj_*\mathbb{C}_S))$ is homologous to
 $(-1)^q \mathcal{C}(\mathcal{O},\Gamma,\mathfrak{q},\sigma_c)$
 as an admissible cycle.
Indeed, since $\lambda$ is elliptic, one can choose a positive integer $N$ such that
 $N(\lambda+\rho(\mathfrak{n}))$ is integral.
Let $f$ be the function on $X$ constructed from the irreducible finite-dimensional representation
 of $G$ with highest weight $N(\lambda+\rho(\mathfrak{n}))$.
Define $\widetilde{\mathcal{C}}$ to be the image of the map
 $\varphi:(0,\frac{1}{N})\times S \to T^* X$ by $(t,x)\mapsto t (d \log f)_x$.  
Then we have
 $\partial \widetilde{\mathcal{C}}
 =\frac{1}{N}[d \log (f|_S)]-\op{CC}(Rj_*\mathbb{C}_S)$ and 
 $\mu_{\lambda+\rho_{\mathfrak{l}}}(\frac{1}{N}[d \log (f|_S)])
 =\mathcal{C}(\mathcal{O},\Gamma,\mathfrak{q},\sigma_c)$,
 which prove the claim.

\subsection{Coherent Continuation}
\label{subsec:coherent_continuation}

In this subsection, we will prove (\ref{eq:elliptic_character_formula}). More precisely, if $(\mathcal{O},\Gamma)$ is an elliptic orbital parameter for $G_{\mathbb{R}}$ in the good range (\ref{eq:good_range}), $\mathfrak{q}$ is a maximally real, admissible polarization, and $\sigma_c$ is a compact form of $G$ commuting with $\sigma$, then 
\begin{equation}\label{eq:final_equation} 
\theta(\mathcal{O},\Gamma)=\mathcal{F}[\mathcal{C}(\mathcal{O},\Gamma,\mathfrak{q},\sigma_c)].
\end{equation}
By Lemma \ref{lem:elliptic_reduction}, (\ref{eq:elliptic_character_formula}) implies the main result of this paper, Theorem \ref{thm:main}. In the previous subsection, we showed (\ref{eq:final_equation}) under the additional integrality assumption (\ref{eq:integral}). In this subsection, we will remove this assumption using a coherent continuation argument.  We note in passing that one can verify (\ref{eq:final_equation}) by rewriting Section \ref{sec:KSV} and Section \ref{subsec:homotopies} in greater generality using the theory of twisted sheaves. However, we have chosen not to do this in an effort to make this article accessible to a wider audience.

Recall from Section \ref{subsec:Rossmann_characters} that the infinitesimal character of an invariant eigendistribution $\theta$ on $\mathfrak{g}_{\mathbb{R}}$ is a $W$-orbit in $\mathfrak{c}^*$. In this section, we will abuse notation and say $\theta$ has infinitesimal character $\eta\in \mathfrak{c}^*$ whenever $\eta$ is contained in the correct $W$-orbit. Let $(\mathfrak{c}^*)'\subset \mathfrak{c}^*$ denote the subset of regular elements.

If $\theta$ is an invariant eigendistribution on $\mathfrak{g}_{\mathbb{R}}$, then $\theta$ is given by integration against a $G_{\mathbb{R}}$-invariant, locally $L^1$, analytic function on the collection of regular, semisimple elements $\mathfrak{g}_{\mathbb{R}}'\subset \mathfrak{g}_{\mathbb{R}}$ (this is the main result of \cite{HC65a}). By abuse of notation, we write $\theta$ for both the invariant eigendistribution and the analytic function. Further, if we fix a Cartan subalgebra $\mathfrak{h}_{\mathbb{R}}\subset \mathfrak{g}_{\mathbb{R}}$ and a choice of positive roots $\Delta^+\subset \Delta(\mathfrak{g},\mathfrak{h})$, then we define the \emph{Weyl denominator} on $\mathfrak{h}_{\mathbb{R}}$ to be the function
\[D(\mathfrak{h}_{\mathbb{R}},\Delta^+)(X)=\prod_{\alpha\in \Delta^+} (e^{\alpha(X)/2}-e^{-\alpha(X)/2}).\]
Harish-Chandra proved that if $u$ is an invariant eigendistribution with infinitesimal character $\eta\in (\mathfrak{c}^*)'$, then one can write 
\[\theta|_{(\mathfrak{h}_{\mathbb{R}})'}=\frac{\sum_{w\in W(\Delta(\mathfrak{g},\mathfrak{h}))} a_we^{w\eta(\mathfrak{h}_{\mathbb{R}},\Delta^+)}}{D(\mathfrak{h}_{\mathbb{R}},\Delta^+)}\]
where $W(\Delta(\mathfrak{g},\mathfrak{h}))$ denotes the Weyl group of the roots of $\mathfrak{g}$ with respect to $\mathfrak{h}$ and $a_w$ is a complex valued function that is constant on the connected components of $(\mathfrak{h}_{\mathbb{R}})'$ for every $w\in W(\Delta(\mathfrak{g},\mathfrak{h}))$ (See Section 8 of \cite{HC56}; see also Chapter 10 of \cite{Kna86} for an exposition).

\begin{definition} A \emph{coherent family of invariant eigendistributions} is a family $\{\theta_{\eta}\}$ of invariant eigendistributions depending on a parameter $\eta\in (\mathfrak{c}^*)'$ such that
\begin{enumerate}
\item $\theta_{\eta}$ has infinitesimal character $\eta$ for every $\eta\in (\mathfrak{c}^*)'$
\item for every real Cartan subalgebra $\mathfrak{h}_{\mathbb{R}}\subset \mathfrak{g}_{\mathbb{R}}$ and every choice of positive roots $\Delta^+\subset \Delta(\mathfrak{g},\mathfrak{h})$, we have an expression
\[\theta_{\eta}|_{(\mathfrak{h}_{\mathbb{R}})'}=\frac{\sum_{w\in W(\Delta(\mathfrak{g},\mathfrak{h}))} a_we^{w\eta(\mathfrak{h}_{\mathbb{R}},\Delta^+)}}{D(\mathfrak{h}_{\mathbb{R}},\Delta^+)}\]
where the locally constant functions $a_w$ are independent of $\eta$.
\end{enumerate}
\end{definition}

\noindent (Traditionally, one discusses coherent families of invariant eigendistributions on $G_{\mathbb{R}}$, which depend on a parameter varying over a certain translate of a lattice in $\mathfrak{c}^*$ \cite{Sch77}. However, for our applications it is useful to study invariant eigendistributions on $\mathfrak{g}_{\mathbb{R}}$ and let the parameter $\eta$ vary over all of $(\mathfrak{c}^*)'$ as on page 266 of \cite{Ro90}.)

The following fact is clear.

\begin{lemma} \label{lem:coherent_equal} If $\{\theta_{\eta}\}_{\eta\in (\mathfrak{c}^*)'}$ and $\{\tau_{\eta}\}_{\eta\in (\mathfrak{c}^*)'}$ are two coherent families of invariant eigendistributions and $\theta_{\eta_0}=\tau_{\eta_0}$ for some $\eta_0\in (\mathfrak{c}^*)'$, then for every $\eta\in (\mathfrak{c}^*)'$, $\theta_{\eta}=\tau_{\eta}$.
\end{lemma}

Armed with this terminology, let us return to our previous setting with fixed data $(\mathcal{O},\Gamma,\mathfrak{q},\sigma_c)$. Fix a fundamental Cartan subalgebra $\mathfrak{h}_{\mathbb{R}}\subset \mathfrak{g}_{\mathbb{R}}$, choose $\lambda\in \sqrt{-1}\mathfrak{h}_{\mathbb{R}}^*\cap \mathcal{O}$, decompose $\mathfrak{q}_{\lambda}=\mathfrak{l}\oplus \mathfrak{n}$ with $\mathfrak{l}:=\mathfrak{g}(\lambda)$, and fix a choice of positive roots $\Delta^+\subset \Delta(\mathfrak{g},\mathfrak{h})$ such that $\Delta(\mathfrak{n},\mathfrak{h})\subset \Delta^+$. Then we may identify $\xi:=\lambda+\rho_{\mathfrak{l}}\in (\mathfrak{h}^*)'\simeq (\mathfrak{c}^*)'$ using our choice of positive roots. As in the proof of Lemma \ref{lem:dlogf}, we have
\[\mu_{\xi}^{-1}\mathcal{C}(\mathcal{O},\Gamma,\mathfrak{q},\sigma_c)=\left\{(\mathfrak{b},g\cdot \lambda-u\cdot \lambda)\mid g\in G_{\mathbb{R}}, u\in U, g\cdot \mathfrak{q}_{\lambda}=u\cdot \mathfrak{q}_{\lambda}, \mathfrak{b}\subset u\cdot \mathfrak{q}_{\lambda}\right\}.\]
Now, define $\mathcal{O}_{\rho(\mathfrak{n})}:=\op{Ad}^*(G_{\mathbb{R}})\cdot \rho(\mathfrak{n})$, define $\Gamma^{\rho(\mathfrak{n})}$ such that $\Gamma^{\rho(\mathfrak{n})}_{\rho(\mathfrak{n})}:=\rho(\mathfrak{n})$, and leave $\mathfrak{q}$, $\sigma_c$ as before. Then similar reasoning leads to
\begin{align*}
&\mu_{\rho}^{-1}\mathcal{C}(\mathcal{O}_{\rho(\mathfrak{n})},\Gamma^{\rho(\mathfrak{n})},\mathfrak{q},\sigma_c)\\
&=\left\{(\mathfrak{b},g\cdot \rho(\mathfrak{n})-u\cdot \rho(\mathfrak{n}))\mid g\in G_{\mathbb{R}}, u\in U, g\cdot \mathfrak{q}_{\lambda}=u\cdot \mathfrak{q}_{\lambda}, \mathfrak{b}\subset u\cdot \mathfrak{q}_{\lambda}\right\}.
\end{align*}

Note that 
\[\tau_{\eta}:=\mathcal{F}[\mu_{\eta}(\op{CC}(Rj_*\mathbb{C}_S))]\]
is a coherent family of invariant eigendistributions (see Theorem 1.4 and Remark 1.5 on pages 268-270 of \cite{Ro90}).

In Lemma \ref{lem:dlogf} (see also the discussion at the end of Section~\ref{subsec:homotopies}), we showed 
\[\mathcal{F}[\mathcal{C}(\mathcal{O},\Gamma,\mathfrak{q},\sigma_c)]=\mathcal{F}[\mu_{\xi}(\op{CC}(Rj_*\mathbb{C}_S))].\]
Replacing $\lambda$ by zero, the same argument yields
\[\mathcal{F}[\mathcal{C}(\mathcal{O}_{\rho(\mathfrak{n})},\Gamma^{\rho(\mathfrak{n})},\mathfrak{q},\sigma_c)]=\mathcal{F}[\mu_{\rho}(\op{CC}(Rj_*\mathbb{C}_S))].\]

Let us record what we have shown.

\begin{lemma} \label{lem:coherent1} There exists a coherent family of $G_{\mathbb{R}}$-invariant eigendistributions $\{\tau_{\eta}\}_{\eta\in (\mathfrak{c}^*)'}$ on $\mathfrak{g}_{\mathbb{R}}$ with 
\[\tau_{\rho}=\mathcal{F}[\mathcal{C}(\mathcal{O}_{\rho(\mathfrak{n})},\Gamma^{\rho(\mathfrak{n})},\mathfrak{q},\sigma_c)]\ \text{and}\ \tau_{\xi}=\mathcal{F}[\mathcal{C}(\mathcal{O},\Gamma,\mathfrak{q},\sigma_c)].\]
\end{lemma}

Next, we will show the following result.

\begin{lemma} \label{lem:coherent2} There exists a coherent family of $G_{\mathbb{R}}$-invariant eigendistributions $\{\theta_{\eta}\}_{\eta\in (\mathfrak{c}^*)'}$ on $\mathfrak{g}_{\mathbb{R}}$ with
\[\theta_{\rho}=\theta(\mathcal{O}_{\rho(\mathfrak{n})},\Gamma^{\rho(\mathfrak{n})})\ \text{and}\ \theta_{\xi}=\theta(\mathcal{O},\Gamma).\]
\end{lemma}

Notice that $(\mathcal{O}_{\rho(\mathfrak{n})},\Gamma^{\rho(\mathfrak{n})},\mathfrak{q},\sigma_c)$ satisfies the integrality condition (\ref{eq:integral}). In particular, the results of Section \ref{subsec:homotopies} imply $\tau_{\rho}=\theta_{\rho}$. Lemma \ref{lem:coherent1}, Lemma \ref{lem:coherent2}, and Lemma \ref{lem:coherent_equal} together imply $\theta_{\eta}=\tau_{\eta}$ for all $\eta\in (\mathfrak{c}^*)'$. In particular, $\theta_{\xi}=\tau_{\xi}$ and we deduce (\ref{eq:final_equation}). Therefore, in order to establish Theorem \ref{thm:main}, all we have left to do is to prove Lemma \ref{lem:coherent2}.
\bigskip

Next, we give a proof of Lemma \ref{lem:coherent2}. For this, we need to state a formula of Zuckerman; we roughly follow the treatment of Adams \cite{Ada11}. Let $H\subset L$ be a $\sigma$-stable Cartan subgroup. If $\Delta^+\subset \Delta(\mathfrak{l},\mathfrak{h})$ is a collection of positive roots, define
\[\rho(\Delta^+):=\frac{1}{2}\sum_{\alpha\in \Delta^+} \alpha.\]
Define $\mathcal{P}(H,\rho)$ to be the set of pairs $(\Lambda,\Delta^+)$ where $\Delta^+\subset \Delta(\mathfrak{l},\mathfrak{h})$ is a collection of positive roots and $\Lambda\colon \widetilde{H}^{\rho(\Delta^+)}\rightarrow \mathbb{C}^{\times}$ is a one-dimensional, holomorphic representation of the $\rho(\Delta^+)$-double cover of $H$ with differential $d\Lambda=\rho(\Delta^+)$. 

Given an element $(\Lambda,\Delta^+)\in \mathcal{P}(H,\rho)$, one associates to $(\Lambda,\Delta^+)$ a reductive subgroup $M_{\mathbb{R}}\subset L_{\mathbb{R}}$ that is $\sigma_c$-stable and a relative Harish-Chandra discrete series representation $I_{M_{\mathbb{R}}}(H,\widetilde{\Lambda})$ where $\widetilde{\Lambda}$ (denote $\Lambda_{M_{\mathbb{R}}}$ in \cite{Ada11}) is obtained from $\Lambda$ by tensoring with a certain character taking values $\pm 1, \pm \sqrt{-1}$ as in (4.4) of \cite{Ada11}. Then, for a suitable real parabolic $P_{\mathbb{R}}\subset L_{\mathbb{R}}$ with Levi factor $M_{\mathbb{R}}$, one defines
\[I(H,\Lambda):=L^2(L_{\mathbb{R}}/P_{\mathbb{R}},I_{M_{\mathbb{R}}}(H,\widetilde{\Lambda})\otimes \mathcal{D}^{1/2})\]
\noindent where $\mathcal{D}^{1/2}$ denotes the bundle of complex half densities on $L_{\mathbb{R}}/P_{\mathbb{R}}$ (see Definition 4.7 on page 8 of \cite{Ada11} for a more thorough discussion). We write $I(H,\Lambda)_{\text{HC}}$ for the Harish-Chandra module of the admissible representation $I(H,\Lambda)$.

Zuckerman gave a formula for the trivial representation $\mathbb{C}^{L_{\mathbb{R}}}$ of the group $L_{\mathbb{R}}$ as an alternating sum of standard modules in the Grothendieck group of Harish-Chandra modules of $L_{\mathbb{R}}$ (see pages 710-711 of \cite{Vog81} and Section 10 of \cite{Ada11}). Zuckerman's formula reads
\begin{equation}\label{eq:Zuckerman_formula_trivial}
\mathbb{C}^{L_{\mathbb{R}}}=\sum_{H} \sum_{(\Lambda,\Delta^+)\in \mathcal{P}(H,\rho)}\epsilon(H,\Lambda) I(H,\Lambda)_{\text{HC}}
\end{equation}
where the first sum is over a set of representatives of the $G_{\mathbb{R}}$-conjugacy classes of $\sigma$-stable Cartan subgroups $H\subset L$ and $\epsilon(H,\Lambda)\in\{\pm 1\}$. 

Next, fix a choice of positive roots $\Delta^+\subset \Delta(\mathfrak{l},\mathfrak{h})$ and fix $(\Lambda,\Delta^+)\in \mathcal{P}(H,\rho)$. Let $\Delta^+(\mathfrak{m},\mathfrak{h}):=\Delta^+\cap \Delta(\mathfrak{m},\mathfrak{h})$ be the collection of positive roots determined by $d\widetilde{\Lambda}$, and form the Borel subalgebra $\mathfrak{d}_{\mathfrak{m}}$ by
\[\mathfrak{d}_{\mathfrak{m}}:=\mathfrak{h}\oplus \sum_{\alpha\in \Delta^+(\mathfrak{m},\mathfrak{h})} \mathfrak{g}_{\alpha}=\mathfrak{h}\oplus \mathfrak{n}_{\mathfrak{m}}.\]
Further, define
\[\rho_{\mathfrak{m}}=\frac{1}{2}\sum_{\alpha\in \Delta^+(\mathfrak{m},\mathfrak{h})}\alpha.\]
Then we have
\[I_{M_{\mathbb{R}}}(H,\widetilde{\Lambda})_{\text{HC}}=(I_{\mathfrak{d}_{\mathfrak{m}},H_{\mathbb{R}}\cap K_{\mathbb{R}}}^{\mathfrak{m},M_{\mathbb{R}}\cap K_{\mathbb{R}}})^{r(\Delta^+)}(\widetilde{\Lambda}\otimes e^{\rho_{\mathfrak{m}}})\]
where $r(\Delta^+)=\dim(\mathfrak{n}_{\mathfrak{m}}\cap \mathfrak{k})$ (see Section 11.8 of \cite{KV95}). Let $\mathfrak{p}$ denote the complexification of the parabolic subalgebra $\mathfrak{p}_{\mathbb{R}}$ and note (see Section 11.2 of \cite{KV95})
\[I(H,\Lambda)_{\text{HC}}:=I_{\mathfrak{p},M_{\mathbb{R}}\cap K_{\mathbb{R}}}^{\mathfrak{l},L_{\mathbb{R}}\cap K_{\mathbb{R}}}(I_{M_{\mathbb{R}}}(H,\widetilde{\Lambda})_{\text{HC}}\otimes e^{\rho(\mathfrak{n}_{\mathfrak{p}})})\]
where
\[\mathfrak{p}=\mathfrak{m}\oplus \mathfrak{n}_{\mathfrak{p}}\]
is a Levi decomposition and
\[\rho(\mathfrak{n}_{\mathfrak{p}}):=\frac{1}{2}\sum_{\alpha\in \Delta(\mathfrak{n}_{\mathfrak{p}},\mathfrak{h})}\alpha.\]
Moreover, define the Borel subalgebra $\mathfrak{d}_{\mathfrak{l}}$ by
\[\mathfrak{d}_{\mathfrak{l}}:=
 \mathfrak{h}\oplus \mathfrak{n}_{\mathfrak{m}}
 \oplus \mathfrak{n}_{\mathfrak{p}}.\]

By Theorem 5.35 and Theorem 5.109 of \cite{KV95}, we have
\[(I_{\mathfrak{d}_{\mathfrak{m}},H_{\mathbb{R}}\cap K_{\mathbb{R}}}^{\mathfrak{m},M_{\mathbb{R}}\cap K_{\mathbb{R}}})^p(\widetilde{\Lambda}\otimes e^{\rho_{\mathfrak{m}}})=0\ \text{if}\ p\neq r(\Delta^+).\]
Therefore, we may apply induction in stages (see Corollary 11.86 on page 683 of \cite{KV95}) to obtain
\[I(H,\Lambda)_{\text{HC}}=(I_{\mathfrak{d}_{\mathfrak{l}}, H_{\mathbb{R}}\cap K_{\mathbb{R}}}^{\mathfrak{l},L_{\mathbb{R}}\cap K_{\mathbb{R}}})^{r(\Delta^+)}(\widetilde{\Lambda}\otimes e^{\rho_{\mathfrak{m}}+\rho(\mathfrak{n}_{\mathfrak{p}})}).\]
Next, for each $(\Lambda,\Delta^+)\in \mathcal{P}(H,\Lambda)$,
 we form the Borel subalgebra 
\[\mathfrak{d}(\Delta^+)=\mathfrak{h}\oplus\mathfrak{n}_{\mathfrak{m}}
 \oplus \mathfrak{n}_{\mathfrak{p}}\oplus \mathfrak{n}\]
where $\mathfrak{n}\subset \mathfrak{q}_{\lambda}$ denotes the nilradical of $\mathfrak{q}_{\lambda}$. 

Now, assume $\eta\in \sqrt{-1}\mathfrak{h}_{\mathbb{R}}^*$ with
\begin{equation}\label{eq:centralizer_condition}
\mathfrak{g}(\eta)=\mathfrak{g}(\lambda)=\mathfrak{l}
\end{equation}
and 
\begin{equation}\label{eq:dominance_condition}
\langle \eta+\rho(\Delta^+),\alpha^{\vee}\rangle> 0
\end{equation}
for all $\alpha\in \Delta(\mathfrak{n},\mathfrak{h})$ and some (equivalently any) choice of positive roots $\Delta^+\subset \Delta(\mathfrak{l},\mathfrak{h})$. Observe that $\eta=\lambda$ and $\eta=\rho(\mathfrak{n})$ both satisfy these conditions. Consider an elliptic orbital parameter $(\mathcal{O}_{\eta},\Gamma^{\eta})$ where $\eta$ satisfies (\ref{eq:centralizer_condition}) and (\ref{eq:dominance_condition}), $\mathcal{O}_{\eta}:=\op{Ad}^*(G_{\mathbb{R}})\cdot \eta$, and $\Gamma^{\eta}_{\eta}=e^{\eta}$. 

A second application of induction in stages yields 
\begin{align*}
&(I_{\mathfrak{q}_{\lambda},L_{\mathbb{R}}
 \cap K_{\mathbb{R}}}^{\mathfrak{g},K_{\mathbb{R}}})^{s}
(I(H,\Lambda)_{\text{HC}}\otimes
 (\Gamma_{\eta}^{\eta}\otimes e^{\rho(\mathfrak{n})})) \\
\simeq\,
 &(I_{\mathfrak{q}_{\lambda},L_{\mathbb{R}}\cap
 K_{\mathbb{R}}}^{\mathfrak{g},K_{\mathbb{R}}})^{s}
((I_{\mathfrak{d}_{\mathfrak{l}}, H_{\mathbb{R}}\cap
 K_{\mathbb{R}}}^{\mathfrak{l},
 L_{\mathbb{R}}\cap K_{\mathbb{R}}})^{r(\Delta^+)}
(\widetilde{\Lambda}\otimes e^{\rho_{\mathfrak{m}}
 +\rho(\mathfrak{n}_{\mathfrak{p}})}
\otimes \Gamma_{\eta}^{\eta}\otimes e^{\rho(\mathfrak{n})})) \\
\simeq\,
 &(I_{\mathfrak{d}(\Delta^+),H_{\mathbb{R}}\cap
 K_{\mathbb{R}}}^{\mathfrak{g},K_{\mathbb{R}}})^{s+r(\Delta^+)}
(\widetilde{\Lambda}\otimes \Gamma_{\eta}^{\eta}\otimes
 e^{\rho_{\mathfrak{m}}+\rho(\mathfrak{n}_{\mathfrak{p}})
 +\rho(\mathfrak{n})}).
\end{align*}
Therefore, tensoring each representation of (\ref{eq:Zuckerman_formula_trivial}) with the character $\Gamma_{\eta}^{\eta}\otimes e^{\rho(\mathfrak{n})}$ and applying the functor $I_{\mathfrak{q}_{\lambda},L_{\mathbb{R}}\cap K_{\mathbb{R}}}^{\mathfrak{g},K_{\mathbb{R}}}$, one obtains an identity in the Grothendieck group of finite length $(\mathfrak{g},K_{\mathbb{R}})$-modules
\begin{align*}
&(I_{\mathfrak{q}_{\lambda},L_{\mathbb{R}}\cap K_{\mathbb{R}}}^{\mathfrak{g},K_{\mathbb{R}}})^s(\Gamma_{\eta}^{\eta}\otimes e^{\rho(\mathfrak{n})}) \numberthis \label{eq:Zuckerman_formula} \\
=&\sum_H \sum_{(\Lambda,\Delta^+)\in \mathcal{P}(H,\Lambda)} \epsilon(H,\Lambda) (I_{\mathfrak{d}(\Delta^+),H_{\mathbb{R}}\cap K_{\mathbb{R}}}^{\mathfrak{g},K_{\mathbb{R}}})^{s+r(\Delta^+)}(\widetilde{\Lambda}\otimes \Gamma_{\eta}^{\eta}\otimes e^{\rho_{\mathfrak{m}}+\rho(\mathfrak{n}_{\mathfrak{p}})
 +\rho(\mathfrak{n})}).
\end{align*}

For every $\sigma$-stable Cartan subgroup $H\subset G$ and each pair $(H,\Delta^+)\in \mathcal{P}(H,\rho)$, we obtain a choice of positive roots $\Delta^+\cup \Delta(\mathfrak{n},\mathfrak{h})\subset \Delta(\mathfrak{g},\mathfrak{h})$.
Let $\Delta^+_{\text{real}}$ and $\Delta^+_{\text{imag}}$ denote the subsets of $\Delta^+\cup \Delta(\mathfrak{n},\mathfrak{h})$ consisting of real and imaginary roots respectively. Proposition 11.126 on page 705 of \cite{KV95} guarantees the existence of a Borel subalgebra $\mathfrak{b}(\Delta^+)$ of \emph{strict type L} associated to $\Delta^+_{\text{real}}$ and $\Delta^+_{\text{imag}}$ (see pages 704-705 of \cite{KV95} for a definition). Define $\Delta^+(\mathfrak{g},\mathfrak{h})\subset \Delta(\mathfrak{g},\mathfrak{h})$ to be the collection of roots for which
\[\mathfrak{b}(\Delta^+)=\mathfrak{h}\oplus \sum_{\alpha\in \Delta^+(\mathfrak{g},\mathfrak{h})}\mathfrak{g}_{\alpha}=\mathfrak{h}\oplus \mathfrak{n}(\Delta^+),\]
and define
\[\delta:=\frac{1}{2}\sum_{\alpha\in \Delta^+(\mathfrak{g},\mathfrak{h})}\alpha,\ s(\Delta^+)=\dim(\mathfrak{n}(\Delta^+)\cap \mathfrak{k}).\] 
By Proposition 11.128 on page 706 of \cite{KV95}, one has an isomorphism 
\begin{equation}\label{eq:transfer}
(I_{\mathfrak{d}(\Delta^+),H_{\mathbb{R}}\cap K_{\mathbb{R}}}^{\mathfrak{g},K_{\mathbb{R}}})^{s+r(\Delta^+)}(\widetilde{\Lambda}\otimes \Gamma_{\eta}^{\eta}\otimes e^{\rho_{\mathfrak{m}}+\rho(\mathfrak{n}_{\mathfrak{p}})
 +\rho(\mathfrak{n})})\simeq (I_{\mathfrak{b}(\Delta^+),H_{\mathbb{R}}\cap K_{\mathbb{R}}}^{\mathfrak{g},K_{\mathbb{R}}})^{s(\Delta^+)}(\widetilde{\Lambda}\otimes \Gamma_{\eta}^{\eta}\otimes e^{\delta}).
\end{equation}
Now, write $\mathfrak{h}=\mathfrak{t}\oplus \mathfrak{a}$,
where $\mathfrak{t}$ is the +1-eigenspace of the action of $\theta:=\sigma\sigma_c$ on $\mathfrak{h}$ and $\mathfrak{a}$ is the $-1$-eigenspace of $\theta$. Let $S$ denote the centralizer of $\mathfrak{a}$ in $G$ with Lie algebra $\mathfrak{s}$, let $\Delta^+(\mathfrak{s},\mathfrak{h}):=\Delta^+(\mathfrak{g},\mathfrak{h})\cap \Delta(\mathfrak{s},\mathfrak{h})$, and let 
\[\mathfrak{p}(\Delta^+):=\mathfrak{s}\oplus \sum_{\alpha\in \Delta^+(\mathfrak{g},\mathfrak{h})\setminus \Delta^+(\mathfrak{s},\mathfrak{h})}\mathfrak{g}_{\alpha}.\]
Then $\mathfrak{p}(\Delta^+)$ is a $\sigma$-stable parabolic subalgebra of $\mathfrak{g}$. Next, we define 
a Borel subalgebra of $\mathfrak{s}$, 
\[\mathfrak{b}(\mathfrak{s},\Delta^+):=\mathfrak{h}\oplus \sum_{\alpha\in \Delta^+(\mathfrak{s},\mathfrak{h})}\mathfrak{g}_{\alpha}.\]
Using induction in stages, we have 
\begin{align*}
& (I_{\mathfrak{b}(\Delta^+),H_{\mathbb{R}}\cap K_{\mathbb{R}}}^{\mathfrak{g},K_{\mathbb{R}}})^{s(\Delta^+)}(\widetilde{\Lambda}\otimes \Gamma_{\eta}^{\eta}\otimes e^{\delta}) \numberthis \label{eq:stages} \\
& \simeq I_{\mathfrak{p}(\Delta^+),S_{\mathbb{R}}\cap K_{\mathbb{R}}}^{\mathfrak{g},K_{\mathbb{R}}}(I_{\mathfrak{b}(\mathfrak{s},\Delta^+),H_{\mathbb{R}}\cap K_{\mathbb{R}}}^{\mathfrak{s},S_{\mathbb{R}}\cap K_{\mathbb{R}}})^{s(\Delta^+)}(\widetilde{\Lambda}\otimes \Gamma_{\eta}^{\eta}\otimes e^{\delta}).
\end{align*}

The representations $(I_{\mathfrak{b}(\mathfrak{s},\Delta^+),H_{\mathbb{R}}\cap K_{\mathbb{R}}}^{\mathfrak{s},S_{\mathbb{R}}\cap K_{\mathbb{R}}})^{s(\Delta^+)}(\widetilde{\Lambda}\otimes \Gamma_{\eta}^{\eta}\otimes e^{\delta})$ are Harish-Chandra modules of relative discrete series representations (see Section XI.8 of \cite{KV95}). Let 
\[\theta(I_{\mathfrak{b}(\mathfrak{s},\Delta^+),H_{\mathbb{R}}\cap K_{\mathbb{R}}}^{\mathfrak{s},S_{\mathbb{R}}\cap K_{\mathbb{R}}})^{s(\Delta^+)}(\widetilde{\Lambda}\otimes \Gamma_{\eta}^{\eta}\otimes e^{\delta})\]
denote the Lie algebra analogue of the character of the corresponding unitary representation. It is implicit in Harish-Chandra's work on discrete series \cite{HC65c}, \cite{HC66} that these relative discrete series characters vary coherently in the parameter $\eta$ (as long as $\eta$ satisfies (\ref{eq:centralizer_condition}) and (\ref{eq:dominance_condition})); this observation was used by Schmid in \cite{Sch77} to motivate the definition of coherent family. Next, the functors $I_{\mathfrak{p}(\Delta^+),S_{\mathbb{R}}\cap K_{\mathbb{R}}}^{\mathfrak{g},K_{\mathbb{R}}}$ are infinitesimal versions of the global real parabolic induction functor. By the induced character formula (see for instance Section X.3 of \cite{Kna86}), real parabolic induction takes coherent families to coherent families. Therefore, if 
\[\theta((I_{\mathfrak{b}(\Delta^+),H_{\mathbb{R}}\cap K_{\mathbb{R}}}^{\mathfrak{g},K_{\mathbb{R}}})^{s(\Delta^+)}(\widetilde{\Lambda}\otimes \Gamma_{\eta}^{\eta}\otimes e^{\delta}))\]
denotes the Lie algebra analogue of the character of the unitary representation with Harish-Chandra module $(I_{\mathfrak{b}(\Delta^+),H_{\mathbb{R}}\cap K_{\mathbb{R}}}^{\mathfrak{g},K_{\mathbb{R}}})^{s(\Delta^+)}(\widetilde{\Lambda}\otimes \Gamma_{\eta}^{\eta}\otimes e^{\delta})$, then this character varies coherently in $\eta$ (as long as $\eta$ satisfies (\ref{eq:centralizer_condition}) and (\ref{eq:dominance_condition})).

Combining (\ref{eq:Zuckerman_formula}), (\ref{eq:transfer}), and the last remark, we have
\[\theta(\mathcal{O}_{\eta},\Gamma^{\eta})=\sum_H \sum_{(\Lambda,\Delta^+)\in \mathcal{P}(H,\Lambda)} \epsilon(H,\Lambda) \theta((I_{\mathfrak{b}(\Delta^+),H_{\mathbb{R}}\cap K_{\mathbb{R}}}^{\mathfrak{g},K_{\mathbb{R}}})^{s(\Delta^+)}(\widetilde{\Lambda}\otimes \Gamma_{\eta}^{\eta}\otimes e^{\delta})).\]
This exhibits $\theta(\mathcal{O}_{\eta},\Gamma^{\eta})$ as a linear combination of characters which vary coherently in the parameter $\eta$ (as long as $\eta$ satisfies (\ref{eq:centralizer_condition}) and (\ref{eq:dominance_condition})). In particular, 
\[\eta\mapsto \theta(\mathcal{O}_{\eta},\Gamma^{\eta})\] with $\eta$ satisfying (\ref{eq:centralizer_condition}) and (\ref{eq:dominance_condition}) is part of a coherent family of invariant eigendistributions. Taking $\eta=\lambda$ and $\eta=\rho(\mathfrak{n})$, Lemma \ref{lem:coherent2} follows. 

\section{Further Remarks and Conjectures}
\label{sec:remarks}

We begin this section by discussing the philosophy of the formula (\ref{eq:character_formula}) in greater detail. Then we give conjectures concerning generalizations of this formula. 

\subsection{Philosophy of the Formula}
\label{subsec:philosophy}

Let $(\mathcal{O},\Gamma)$ be a semisimple orbital parameter with $\mathcal{O}$ in the good range ((\ref{eq:good_range}) and Section \ref{subsec:parabolic_induction}). Fix $\lambda\in \mathcal{O}$ such that $L_{\mathbb{R}}:=G_{\mathbb{R}}(\lambda)$ is $\sigma_c$-stable.
Fix a fundamental Cartan subalgebra $\mathfrak{h}_{\mathbb{R}}\subset \mathfrak{l}_{\mathbb{R}}:=\mathfrak{g}_{\mathbb{R}}(\lambda)$ with complexification $\mathfrak{h}$ and fix a compact form $U_L:=L^{\sigma_c}\subset L$ for which $\mathfrak{h}$ is $\sigma_c$-stable. If $\Delta^+\subset \Delta(\mathfrak{l},\mathfrak{h})$ is a choice of positive roots, define
\[\rho_{\mathfrak{l},\mathfrak{h},\Delta^+}:=\frac{1}{2}\sum_{\alpha\in \Delta^+}\alpha\in \mathfrak{h}^*\subset \mathfrak{l}^*\]
and define
\[\mathcal{O}^{U_L}_{\rho}:=\op{Ad}^*(U_L)\cdot \rho_{\mathfrak{l},\mathfrak{h},\Delta^+}\subset \mathfrak{l}^*.\]

By Kirillov's character formula for the compact group $U_L:=L^{\sigma_c}$ \cite{HC57a}, \cite{Kir68}, we have 
\begin{equation}\label{eq:fiber_formula}
\mathcal{F}[\lambda+\mathcal{O}_{\rho}^{U_L}]=\theta_{\lambda}^{U_L}:=j_{U_L}^{1/2}\cdot \lambda.
\end{equation}
In reality, $\lambda$ may not lift to a character of $U_L$; hence, one may have to apply a coherent continuation argument (as in Section \ref{subsec:coherent_continuation}) in order to obtain (\ref{eq:fiber_formula}). Now, one may write down explicit formulas for $j_{U_L}^{1/2}$ and $j_{L_{\mathbb{R}}}^{1/2}$ (see for instance Theorem 1.14 on page 96 of \cite{Hel00}). From these formulas, one observes that both of these functions extend to identical holomorphic functions on $\mathfrak{l}$. In addition, observe that the definition of the Fourier transform is also independent of the real form of $L$ (see (\ref{eq:Fourier_transform_2})). Hence, we deduce
\begin{equation}\label{eq:fiber_formula2}
\mathcal{F}[\lambda+\mathcal{O}_{\rho}^{U_L}]=\theta(\{\lambda\},\Gamma_{\lambda}):=j_{L_{\mathbb{R}}}^{1/2}\cdot \lambda.
\end{equation}
Therefore, the Fourier transform of the contour $\lambda+\mathcal{O}_{\rho}^{U_L}$ is the Lie algebra analogue of the character of $\Gamma_{\lambda}$.

Philosophically, one might wish to define a contour $\varpi\colon \mathcal{C}(\mathcal{O},\Gamma)\rightarrow \mathcal{O}$ where the fiber over $\lambda\in \mathcal{O}$ is $\lambda+\mathcal{O}^{U_L}_{\rho}$. Defining such a fiber for every $\lambda\in \mathcal{O}$ is the same as choosing a maximal compact subgroup of $G_{\mathbb{R}}(\lambda)$ for every $\lambda\in \mathcal{O}$. One checks that there are many possible ways to make a continuous choice varying over $\lambda\in \mathcal{O}$. However, every such choice produces a contour $\mathcal{C}(\mathcal{O},\Gamma)$ whose real part is unbounded in $\mathfrak{g}_{\mathbb{R}}^*$; in particular $\mathcal{C}(\mathcal{O},\Gamma)$ is not an admissible contour in the sense of Rossmann (see Section \ref{subsec:Rossmann_characters}).  

In order to deform $\mathcal{C}(\mathcal{O},\Gamma)$ into an admissible contour, we require a choice of maximally real admissible polarization $\left\{\mathfrak{q}_{\lambda}\right\}_{\lambda\in \mathcal{O}}$ of $\mathcal{O}$. In addition, we choose an involution $\sigma_c$ of $G$ for which $U:=G^{\sigma_c}\subset G$ is a maximal compact subgroup with $\sigma\sigma_c=\sigma_c\sigma$. For every $\lambda\in \mathcal{O}$, there exists a unique $\sigma_c$-stable, Levi subalgebra $\mathfrak{l}_{\lambda,\sigma_c}\subset \mathfrak{q}_{\lambda}$. To check this, one first fixes $\lambda_0\in \mathcal{O}$ and a fundamental Cartan subalgebra $\mathfrak{h}_{\mathbb{R}}\subset \mathfrak{g}_{\mathbb{R}}(\lambda_0)$. After conjugating by an element of $G_{\mathbb{R}}$, one may assume that $\mathfrak{h}_{\mathbb{R}}$ is $\sigma_c$-stable. Then  $\mathfrak{l}_{\lambda_0,\sigma_c}:=\mathfrak{g}(\lambda_0)\subset \mathfrak{q}_{\lambda_0}$ is a $\sigma_c$-stable Levi factor, and any other $\sigma_c$-stable Levi factor would have to be conjugate by $Q_{\lambda}\cap U=L_{\lambda_0,\sigma_c}\cap U$ where $Q_{\lambda}$ is the normalizer of $\mathfrak{q}_{\lambda}$ and $L_{\lambda_0,\sigma_c}$ is the normalizer of $\mathfrak{l}_{\lambda_0,\sigma_c}$. But, $L_{\lambda_0,\sigma_c}\cap U$ normalizes $\mathfrak{l}_{\lambda_0,\sigma_c}$; hence, $\mathfrak{l}_{\lambda_0,\sigma_c}$ is the unique $\sigma_c$-stable Levi factor. Now, one uses the transitive action of $U$ on the partial flag variety $ Y:=\op{Ad}(G)\cdot \mathfrak{q}_{\lambda}$ to prove that $\mathfrak{l}_{\lambda,\sigma_c}:=\operatorname{Ad}(u)\cdot \mathfrak{l}_{\lambda_0,\sigma_c}$ is the unique $\sigma_c$-stable Levi factor of $\mathfrak{g}_{\lambda}$ if $\operatorname{Ad}(u)\cdot \mathfrak{q}_{\lambda_0}=\mathfrak{q}_{\lambda}$.

Now, if $\mathfrak{h}\subset \mathfrak{l}_{\lambda,\sigma_c}$ is a $\sigma_c$-stable Cartan subalgebra and $\Delta^+\subset \Delta(\mathfrak{g},\mathfrak{h})$ is a choice of positive roots, then we define 
\[\rho_{\mathfrak{l}_{\lambda,\sigma_c},\mathfrak{h},\Delta^+}=\frac{1}{2}\sum_{\alpha\in \Delta^+}\alpha.\]

We put $U_{\lambda}:=L_{\lambda,\sigma_c}\cap U$, and we put
\[\mathcal{O}_{\rho}^{U_{\lambda}}:=\bigcup_{\substack{\sigma_c(\mathfrak{h})=\mathfrak{h}\\ \Delta^+\subset \Delta(\mathfrak{l}_{\lambda,\sigma_c},\mathfrak{h})}}\rho_{\mathfrak{l}_{\lambda,\sigma_c},\mathfrak{h},\Delta^+}\subset \mathfrak{l}_{\lambda,\sigma_c}^*\]
where the union is over Cartan subalgebras $\mathfrak{h}\subset \mathfrak{l}_{\lambda,\sigma_c}$ that are $\sigma_c$-stable and choices of positive roots $\Delta^+\subset \Delta(\mathfrak{g},\mathfrak{h})$. Alternately, $\mathcal{O}_{\rho}^{U_{\lambda}}$ is the $U_{\lambda}$-orbit through any single $\rho_{\mathfrak{l}_{\lambda,\sigma_c},\mathfrak{h},\Delta^+}$. Observe that $\mathfrak{l}_{\lambda,\sigma_c}=\op{Ad}(g)\cdot \mathfrak{g}(\lambda)$ for some $g\in G$ since all Levi factors of $\mathfrak{q}_{\lambda}$ are conjugate. Moreover, given a maximal compact subgroup $U_L\subset G(\lambda)$, one may choose $g\in G$ such that $g(U_{\lambda})g^{-1}=U_L$. Then one obtains

\[\lambda+\op{Ad}^*(g)\cdot \mathcal{O}_{\rho}^{U_{\lambda}}=\lambda+\mathcal{O}_{\rho}^{U_L}.\]

Since $G$ is a connected group, one can find a path connecting the identity $e$ to $g$ which then induces a homotopy between 
the fiber $\lambda+\mathcal{O}_{\rho}^{U_{\lambda}}$ and the fiber we desired (because of (\ref{eq:fiber_formula2})) $\lambda+\mathcal{O}_{\rho}^{U_L}$. As in the introduction, we put

\[\mathcal{C}(\mathcal{O},\Gamma,\mathfrak{q},\sigma_c)=\bigcup_{\lambda\in \mathcal{O}} \lambda+\mathcal{O}_{\rho}^{U_{\lambda}}\]

\noindent and we observe that the fibers of this contour are individually homotopic to the fibers of our desired contour $\mathcal{C}(\mathcal{O},\Gamma)$. Philosophically, one should expect
\[\mathcal{F}[\mathcal{C}(\mathcal{O},\Gamma)]=\theta(\mathcal{O},\Gamma)\]
in some generalized sense. We have seen that the addition of the polarization $\mathfrak{q}=\{\mathfrak{q}_{\lambda}\}_{\lambda\in \mathcal{O}}$ produces homotopies on the fibers of $\mathcal{C}(\mathcal{O},\Gamma)$ so that the contour is deformed to a contour $\mathcal{C}(\mathcal{O},\Gamma,\mathfrak{q},\sigma_c)$. Since this contour is admissible in the sense of Rossmann (see Section \ref{subsec:Rossmann_characters}), we can precisely define what we mean by
\[\mathcal{F}[\mathcal{C}(\mathcal{O},\Gamma,\mathfrak{q},\sigma_c)]=\theta(\mathcal{O},\Gamma).\]
This is the main result of this article.

\subsection{Conjectures}
\label{subsec:conjectures}

In this section, we enumerate several conjectures suggesting future work beyond Theorem \ref{thm:main}.

\begin{conjecture}\label{conj:admissible} The character formula (\ref{eq:character_formula}) holds for any admissible polarization $\mathfrak{q}=\{\mathfrak{q}_{\lambda}\}_{\lambda\in \mathcal{O}}$. In particular, the polarization need not be maximally real.
\end{conjecture}

One can give a construction of the Harish-Chandra module of $\pi(\mathcal{O},\Gamma)$ utilizing any admissible polarization $\mathfrak{q}$. In Chapter XI of \cite{KV95}, this construction is carried out and shown to be independent of admissible polarization in the special case where $\mathcal{O}$ is of maximal dimension (or equivalently $\mathfrak{q}$ is a Borel subalgebra). An analogous statement can be proved for arbitrary semisimple $\mathcal{O}$ (see Appendix~\ref{appendix}).

This leads the authors to believe that Theorem \ref{thm:main} is also true for an arbitrary admissible polarization. The authors feel it is likely that a proof of Conjecture \ref{conj:admissible} can be given by generalizing techniques of this paper.
\bigskip

Second, we motivated our main formula (\ref{eq:character_formula}) by pointing out that it is a generalization of Rossmann's formula (\ref{eq:character_tempered}) for the characters of irreducible, tempered representations with regular infinitesimal character. Due to work of Harish-Chandra \cite{HC76}, Langlands \cite{La89}, and Knapp-Zuckerman \cite{KZ77}, \cite{KZ82a}, \cite{KZ82b}, it is known that any irreducible, tempered representation can be realized as a ``limit'' of irreducible, tempered representations with regular infinitesimal character and it is known how these limits decompose. In \cite{Ro80}, \cite{Ro82a}, \cite{Ro82b}, Rossmann studied the corresponding limits of character formulas. He found that the limit of coadjoint orbits becomes degenerate if, and only if the limit of representations is zero (this follows from the main theorem of \cite{Ro82a} combined with work of Knapp-Zuckerman \cite{KZ77}). In addition, if the limit of representations decomposes, then the limit of orbits must have an analogous decomposition (this is the main theorem of \cite{Ro82b}). These results give an elegant geometric picture of the collection of irreducible tempered representations.

One might expect analogues of these results for ``limits'' of representations of the form $\pi(\mathcal{O},\Gamma)$ where $\mathcal{O}$ is semisimple but not necessarily of maximal dimension. In particular, we might expect that the limit of contours $\mathcal{C}(\mathcal{O},\Gamma,\mathfrak{q},\sigma_c)$ becomes degenerate if, and only if the limit representation is zero. And we might expect that a decomposition of the limit representation must be accompanied by a natural decomposition of the limit contour. These questions are especially interesting since the behavior of the limit representation is more complicated in the case where $\mathcal{O}$ is not of maximal dimension.
\bigskip

Finally, it has long been conjectured that a large and important part of the collection of irreducible, unitary representations of $G_{\mathbb{R}}$ can be constructed from orbital parameters that generalize the semisimple orbital parameters introduced in this paper. Over the past 35 years, a vast literature has emerged that, in a variety of special cases, associates irreducible, unitary representations to coadjoint orbits (see for instance \cite{Vog81b}, \cite{Tor83}, \cite{BV85}, \cite{BZ91}, \cite{Kna04}, \cite{Tra04}, \cite{HKM14}, \cite{Won16}, \cite{BT} for a small sampling of this literature). However, despite these efforts, no general procedure has been found that associates a unitary representation to an arbitrary (not necessarily semisimple) orbital parameter. It has been previously suggested that one might solve this problem by constructing the character of the representation $\pi(\mathcal{O},\Gamma)$ associated to the orbital parameter $(\mathcal{O},\Gamma)$ rather than by constructing the representation itself \cite{Ve94}. This article gives a generalization of Rossmann's formula that is different than the ones in \cite{Ve94}; hence, it might breathe new life into this project.

Let us finish by roughly sketching what a solution might look like. First, one must define the notion of an orbital parameter $(\mathcal{O},\Gamma)$ in full generality. This is more subtle than it looks; for instance, one already sees in the work of Rossmann \cite{Ro82b} that $\mathcal{O}$ may be the union of several coadjoint orbits even in the case where $\mathcal{O}$ is of maximal dimension. Second, one must construct an appropriate contour $\mathcal{C}(\mathcal{O},\Gamma)\rightarrow \mathcal{O}$ contained in $\mathfrak{g}^*$. A correct solution to this problem should involve an \emph{explicit} formula for this contour and it should directly relate the fiber over $\lambda\in \mathcal{O}$ to a character formula for $\Gamma_{\lambda}$. Third, one must show 
\[\theta(\mathcal{O},\Gamma):=\mathcal{F}[\mathcal{C}(\mathcal{O},\Gamma)]\]
is the Lie algebra analogue of the character of a unitary representation $\pi(\mathcal{O},\Gamma)$ that is then associated to $(\mathcal{O},\Gamma)$.

\appendix
\section{Independence of Polarization}
\label{appendix}

Let $(\mathcal{O}, \Gamma)$ be a semisimple orbital parameter.
For a maximally real admissible polarization $\mathfrak{q}=\mathfrak{l}+\mathfrak{n}$,
 we defined 
 a unitary representation $\pi(\mathcal{O}, \Gamma)$ of $G_{\mathbb{R}}$
 combining Vogan-Zuckerman construction and Mackey parabolic induction.
On the $(\mathfrak{g},K)$-module level, both of these two constuctions are given by cohomological induction (see \cite[Chapter XI]{KV95}).
Therefore, the underlying $(\mathfrak{g},K)$-module of 
$\pi(\mathcal{O}, \Gamma)$ is isomorphic to
 $(I_{\mathfrak{q}, 
 L_{\mathbb{R}}\cap K_{\mathbb{R}}}^{\mathfrak{g},K_{\mathbb{R}}})^s
 (\Gamma_{\lambda}\otimes\mathbb{C}_{\rho(\mathfrak{n})})$,
 where $s=\dim (\mathfrak{n}\cap\mathfrak{k})$.
Also, 
 $(I_{\mathfrak{q}, 
 L_{\mathbb{R}}\cap K_{\mathbb{R}}}^{\mathfrak{g},K_{\mathbb{R}}})^j
 (\Gamma_{\lambda}\otimes\mathbb{C}_{\rho(\mathfrak{n})})=0$ for $j\neq s$.
In light of this, let us more generally define
 a virtual $(\mathfrak{g},K)$-module 
\begin{align*}
\pi(\mathcal{O}, \Gamma,\mathfrak{q})
:=\sum_{j}(-1)^{j+s}
(I_{\mathfrak{q}, 
 L_{\mathbb{R}}\cap K_{\mathbb{R}}}^{\mathfrak{g},K_{\mathbb{R}}})^j
 (\Gamma_{\lambda}\otimes\mathbb{C}_{\rho(\mathfrak{n})})
\end{align*}
for any polarization $\mathfrak{q}$.
The purpose of this appendix is to show that 
 $\pi(\mathcal{O}, \Gamma,\mathfrak{q})$ is independent of the choice of $\mathfrak{q}$
 as long as $\mathfrak{q}$ is admissible:
\begin{theorem} \label{thm:transfer} 
If $\mathfrak{q}_1$ and $\mathfrak{q}_2$ are admissible, then 
$\pi(\mathcal{O}, \Gamma, \mathfrak{q}_1)=\pi(\mathcal{O}, \Gamma, \mathfrak{q}_2)$.
\end{theorem} 

We note that if $\mathfrak{q}$ is maximally real, admissible, then 
$\pi(\mathcal{O}, \Gamma,\mathfrak{q})=\pi(\mathcal{O}, \Gamma)$
 and we already saw in Section~\ref{sec:representations}
 that this does not depend on the choice of $\mathfrak{q}$.

A key result to prove Theorem~\ref{thm:transfer} is the transfer theorem
 (\cite[Theorem 11.87]{KV95}).
The following lemma is an easy consequence of the transfer theorem.
Let $\mathfrak{h}$ be a Cartan subalgebra stable by the Cartan involution
 $\theta$ of $\mathfrak{g}$
 and let $\mathfrak{h}=\mathfrak{h}^{\theta}\oplus\mathfrak{h}^{-\theta}$
 be the decomposition into $\theta$-eigenspaces.
We say a root $\alpha\in\Delta(\mathfrak{g},\mathfrak{h})$ is
 real, imaginary, and complex, respectively if 
 $\alpha|_{\mathfrak{h}^{\theta}}=0$, $\alpha|_{\mathfrak{h}^{-\theta}}=0$ and
 otherwise.

\begin{lemma} \label{lemma:transfer1} 
Suppose that
 $\mathfrak{b}_1=\mathfrak{h}+\mathfrak{n}_1$ and $\mathfrak{b}_2=\mathfrak{h}+\mathfrak{n}_2$
 are Borel subalgebras of $\mathfrak{g}$
 containing a Cartan subalgebra $\mathfrak{h}$.
Let $V$ be a one-dimensional
 $(\mathfrak{h}, H_{\mathbb{R}}\cap K_{\mathbb{R}})$-module on which $\mathfrak{h}$ acts by scalars according to $\lambda\in\mathfrak{h}^*$.
We assume that all roots $\alpha$ in 
 $\Delta(\mathfrak{n}_1,\mathfrak{h})\setminus
  \Delta(\mathfrak{n}_2,\mathfrak{h})$ are complex roots and 
 satisfy $\langle \lambda, \alpha\rangle \not\in {\mathbb R}$.
Then we have the following equation of virtual $(\mathfrak{g},K)$-modules
\begin{align*}
\sum_j
(-1)^j (I_{\frak{b}_1, H_{\mathbb{R}}\cap K_{\mathbb{R}}}^{\frak{g},K_{\mathbb{R}}})^j(V)
= \sum_j (-1)^{j+t}
 (I_{\frak{b}_2, H_{\mathbb{R}}\cap K_{\mathbb{R}}}^{\frak{g},K_{\mathbb{R}}})^j
 (V\otimes \bigwedge^{\rm top}(\frak{n}_2/(\frak{n}_1\cap\frak{n}_2))),
\end{align*}
where $t=\dim(\frak{n}_1\cap\frak{k})-\dim(\frak{n}_2\cap\frak{k})$.
\end{lemma}

\begin{proof}
We use induction on the dimension of
 $\frak{n}_2/(\frak{n}_1\cap\frak{n}_2)$.
If $\dim (\frak{n}_2/(\frak{n}_1\cap\frak{n}_2))=0$,
 then $\frak{n}_1=\frak{n}_2$ and the statement follows.

Suppose $\dim (\frak{n}_2/(\frak{n}_1\cap\frak{n}_2)) > 0$.
Then there exists a simple root $\alpha$ with respect to the
 positive roots $\Delta(\frak{n}_1,\frak{h})$ such that
 $\alpha \not\in \Delta(\frak{n}_2,\frak{h})$.
Let 
\begin{align*}
\frak{n}_3=\frak{g}_{-\alpha} \oplus \bigoplus_{\substack{\beta\in \Delta(\frak{n}_1,\frak{h})\\ \beta\neq \alpha}} \frak{g}_{\beta},
\qquad \frak{b}_3=\frak{h}\oplus \frak{n}_3.
\end{align*}
Since $\alpha$ is a complex root,
 the Borel subalgebras $\frak{b}_1$ and $\frak{b}_3$ satisfy the conditions
 in \cite[Theorem 11.87]{KV95}.
As a result, we get 
\begin{align*}
(I_{\frak{b}_1, H_{\mathbb{R}}\cap K_{\mathbb{R}}}^{\frak{g},K_{\mathbb{R}}})^j(V)
\simeq (I_{\frak{b}_3, H_{\mathbb{R}}\cap K_{\mathbb{R}}}^{\frak{g},K_{\mathbb{R}}})^{j\pm 1}(V\otimes \mathbb{C}_\alpha)
\end{align*}
for all $j$, where the degree on the right hand side is $j-1$ (resp.\ $j+1$)
 if $\theta\alpha\in\Delta(\frak{n}_1,\frak{h})$
 (resp.\ $\theta\alpha\not\in\Delta(\frak{n}_1,\frak{h})$).
Hence 
\begin{align*}
\sum_j
(-1)^j (I_{\frak{b}_1, H_{\mathbb{R}}\cap K_{\mathbb{R}}}^{\frak{g},K_{\mathbb{R}}})^j(V)
= \sum_j (-1)^{j+1}
 (I_{\frak{b}_3, H_{\mathbb{R}}\cap K_{\mathbb{R}}}^{\frak{g},K_{\mathbb{R}}})^j (V\otimes \mathbb{C}_\alpha).
\end{align*}
Since $\dim (\frak{n}_2/(\frak{n}_3\cap\frak{n}_2))=\dim (\frak{n}_2/(\frak{n}_1\cap\frak{n}_2))-1$, the induction hypothesis implies 
\begin{multline*}
\sum_j
(-1)^j (I_{\frak{b}_3, H_{\mathbb{R}}\cap K_{\mathbb{R}}}^{\frak{g},K_{\mathbb{R}}})^j(V\otimes \mathbb{C}_\alpha)\\
= \sum_j (-1)^{j+t+1}
 (I_{\frak{b}_2, H_{\mathbb{R}}\cap K_{\mathbb{R}}}^{\frak{g},K_{\mathbb{R}}})^j
 (V\otimes \mathbb{C}_\alpha \otimes \bigwedge^{\rm top}(\frak{n}_2/(\frak{n}_3\cap\frak{n}_2))).
\end{multline*}
These two equations and
 $\frak{n}_2/(\frak{n}_1\cap\frak{n}_2)\simeq
 \frak{n}_2/(\frak{n}_3\cap\frak{n}_2) \oplus \mathbb{C}_\alpha$
 give
\begin{align*}
\sum_j
(-1)^j (I_{\frak{b}_1, H_{\mathbb{R}}\cap K_{\mathbb{R}}}^{\frak{g},K_{\mathbb{R}}})^j(V)
= \sum_j (-1)^{j+t}
 (I_{\frak{b}_2, H_{\mathbb{R}}\cap K_{\mathbb{R}}}^{\frak{g},K_{\mathbb{R}}})^j
 (V \otimes \bigwedge^{\rm top}(\frak{n}_2/(\frak{n}_1\cap\frak{n}_2))),
\end{align*}
 proving the lemma.
\end{proof}

We also have a similar lemma for parabolic subalgebras.
\begin{lemma} \label{lemma:transfer2} 
Suppose that
 $\frak{q}_1=\frak{l}+\frak{n}_1$ and $\frak{q}_2=\frak{l}+\frak{n}_2$
 are parabolic subalgebras of $\frak{g}$
 which have the same Levi factor $\frak{l}$.
We assume that $\frak{l}$ is $\theta$-stable.
Let $V$ be a one-dimensional
 $(\frak{l}, L_{\mathbb{R}}\cap K_{\mathbb{R}})$-module on which $\frak{h}$ acts by scalars according to $\lambda\in\frak{h}^*$.
We assume that all roots $\alpha$ in 
 $\Delta(\frak{n}_1,\frak{h})\setminus
  \Delta(\frak{n}_2,\frak{h})$ are complex roots and 
 satisfy $\langle \lambda, \alpha\rangle \not\in {\mathbb R}$.
Then we have the following equation of virtual $(\frak{g},K)$-modules
\begin{align*}
\sum_j
(-1)^j (I_{\frak{q}_1, L_{\mathbb{R}}\cap K_{\mathbb{R}}}^{\frak{g},K_{\mathbb{R}}})^j(V)
= \sum_j (-1)^{j+t}
 (I_{\frak{q}_2, L_{\mathbb{R}}\cap K_{\mathbb{R}}}^{\frak{g},K_{\mathbb{R}}})^j
 (V\otimes \bigwedge^{\rm top}(\frak{n}_2/(\frak{n}_1\cap\frak{n}_2))),
\end{align*}
where $t=\dim(\frak{n}_1\cap\frak{k})-\dim(\frak{n}_2\cap\frak{k})$.
\end{lemma}

\begin{proof}
As in \eqref{eq:Zuckerman_formula_trivial}, write the trivial representation $\mathbb{C}^{L_{\mathbb{R}}}$ of $L_{\mathbb{R}}$ as an alternating sum of standard modules:
\begin{align*}
\mathbb{C}^{L_{\mathbb{R}}}=\sum_{H'} \sum_{(\Lambda,\Delta^+)\in \mathcal{P}(H',\rho)}\epsilon(H',\Lambda) I(H',\Lambda)_{\text{HC}}
\end{align*}
The standard module 
$I(H',\Lambda)_{\text{HC}}$ is isomorphic to a certain
 cohomological induced module.
We can take a Borel subalgebra $\frak{b}_L$, a one-dimensional module $W$, 
 and a degree $s$ such that
\begin{align*}
I(H',\Lambda)_{\text{HC}}\simeq 
(I_{\frak{b}_L, H'_{\mathbb{R}}\cap K_{\mathbb{R}}}^{\frak{l},L_{\mathbb{R}}\cap K_{\mathbb{R}}})^s(W)
\end{align*}
and 
$(I_{\frak{b}_L, H'_{\mathbb{R}}\cap K_{\mathbb{R}}}^{\frak{l},L_{\mathbb{R}}\cap K_{\mathbb{R}}})^j(W)=0$ for $j\neq s$.
Moreover, it follows that the $\frak{h}$-action on $W$ is given
 by an element in the root lattice.
As a result, we get an equation 
\begin{align*}
\mathbb{C}^{L_{\mathbb{R}}}=\sum_{H'} \sum_{W,\frak{b}_L}
 \epsilon(H',W)
 \sum_j (-1)^j (I_{\frak{b}_L, H'_{\mathbb{R}}\cap K_{\mathbb{R}}}^{\frak{l},L_{\mathbb{R}}\cap K_{\mathbb{R}}})^j(W).
\end{align*}
By tensoring with $V$, we have 
\begin{align*}
V=\sum_{H'} \sum_{W,\frak{b}_L}
 \epsilon(H',W)
 \sum_j (-1)^j (I_{\frak{b}_L, H'_{\mathbb{R}}\cap K_{\mathbb{R}}}^{\frak{l},L_{\mathbb{R}}\cap K_{\mathbb{R}}})^j(W\otimes V).
\end{align*}
By applying
 $(-1)^i (I_{\frak{q}_1, L_{\mathbb{R}}\cap K_{\mathbb{R}}}^{\frak{g},K_{\mathbb{R}}})^i$ to this equation:
\begin{align*}
&\sum_i
(-1)^i (I_{\frak{q}_1, L_{\mathbb{R}}\cap K_{\mathbb{R}}}^{\frak{g},K_{\mathbb{R}}})^i(V)\\
&=\sum_{H'} \sum_{W,\frak{b}_L}
 \epsilon(H',W)
 \sum_{i,j}
(-1)^{i+j} (I_{\frak{q}_1, L_{\mathbb{R}}\cap K_{\mathbb{R}}}^{\frak{g},K_{\mathbb{R}}})^i
 (I_{\frak{b}_L, H'_{\mathbb{R}}\cap K_{\mathbb{R}}}^{\frak{l},L_{\mathbb{R}}\cap K_{\mathbb{R}}})^j(W\otimes V)\\
&=\sum_{H'} \sum_{W,\frak{b}_L}
 \epsilon(H',W)
 \sum_j (-1)^j (I_{\frak{b}_L\oplus \frak{n}_1, H'_{\mathbb{R}}\cap K_{\mathbb{R}}}^{\frak{g}, K_{\mathbb{R}}})^j(W\otimes V).
\end{align*}
Here, the last equation is a consequence of a spectral sequence
 (see \cite[p.680, Theorem 11.77]{KV95})
\begin{align*}
(I_{\frak{q}_1, L_{\mathbb{R}}\cap K_{\mathbb{R}}}^{\frak{g},K_{\mathbb{R}}})^i
 (I_{\frak{b}_L, H'_{\mathbb{R}}\cap K_{\mathbb{R}}}^{\frak{l},L_{\mathbb{R}}\cap K_{\mathbb{R}}})^j(W\otimes V)
 \Longrightarrow
 (I_{\frak{b}_L\oplus \frak{n}_1, H'_{\mathbb{R}}\cap K_{\mathbb{R}}}^{\frak{g}, K_{\mathbb{R}}})^{i+j}(W\otimes V).
\end{align*}
Similarly, 
\begin{multline*}
\sum_j (-1)^{j}
 (I_{\frak{q}_2, L_{\mathbb{R}}\cap K_{\mathbb{R}}}^{\frak{g},K_{\mathbb{R}}})^j
 (V\otimes \bigwedge^{\rm top}(\frak{n}_2/(\frak{n}_1\cap\frak{n}_2)))\\
=\sum_{H'} \sum_{W,\frak{b}_L}
 \epsilon(H',W)
 \sum_j (-1)^j (I_{\frak{b}_L\oplus \frak{n}_2, H'_{\mathbb{R}}\cap K_{\mathbb{R}}}^{\frak{g}, K_{\mathbb{R}}})^j(W\otimes V\otimes \bigwedge^{\rm top}(\frak{n}_2/(\frak{n}_1\cap\frak{n}_2)))
\end{multline*}
with the same index set of $H',W,\frak{b}_L$.
If we put $\frak{b}_1=\frak{b}_L\oplus \frak{n}_1$ and
 $\frak{b}_2=\frak{b}_L\oplus \frak{n}_2$, we can apply
 Lemma~\ref{lemma:transfer1} and conclude that 
\begin{multline*}
\sum_j (-1)^j (I_{\frak{b}_L\oplus \frak{n}_1, H'_{\mathbb{R}}\cap K_{\mathbb{R}}}^{\frak{g}, K_{\mathbb{R}}})^j(W\otimes V) \\
= \sum_j (-1)^j (I_{\frak{b}_L\oplus \frak{n}_2, H'_{\mathbb{R}}\cap K_{\mathbb{R}}}^{\frak{g}, K_{\mathbb{R}}})^j(W\otimes V\otimes \bigwedge^{\rm top}(\frak{n}_2/(\frak{n}_1\cap\frak{n}_2))).
\end{multline*}
This shows the equation in the lemma.
\end{proof}

We now turn to the proof of Theorem~\ref{thm:transfer}.
We say a polarization $\frak{q}=\frak{l}+\frak{n}$ satisfies condition (*) if 
\begin{itemize}
\item 
$\op{Re}\langle \lambda, \alpha \rangle>0$ implies 
 $\alpha\in \Delta(\frak{n},\frak{h})$
 for roots $\alpha\in \Delta(\frak{g},\frak{h})$, and
\item
If $\op{Re}\langle \lambda, \alpha \rangle=0$
 and $\alpha\in\Delta(\frak{n},\frak{h})$, then
 $\sigma\alpha\in\Delta(\frak{n},\frak{h})$.
\end{itemize}
Let $\lambda=\lambda_c+\lambda_n$ be the decomposition
 according to $\mathfrak{h}=\frak{h}^{\theta}+\frak{h}^{-\theta}$.
Let $\frak{q}_{\op{Im}}=\frak{l}_{\op{Im}}\oplus\frak{n}_{\op{Im}}$
 be the parabolic subalgebra defined by
\begin{align*}
\Delta(\frak{n}_{\op{Im}},\frak{h})
 =\{\alpha\in\Delta(\frak{g},\frak{h}):
 \langle\lambda_c,\alpha\rangle > 0\},\qquad 
\frak{l}_{\op{Im}}=\frak{g}(\lambda_c).
\end{align*}
Then if a polarization $\frak{q}$ satisfies condition (*), then
 $\frak{q}\subset \frak{q}_{\op{Im}}$
 and $\frak{q}\cap\frak{l}_{\op{Im}}$ is a $\sigma$-stable
 parabolic subalgebra of $\frak{l}_{\op{Im}}$.

\begin{proof}[proof of Theorem~\ref{thm:transfer}]
For an admissible polarization $\frak{q}_1=\frak{l}+\frak{n}_1$
 define the set $\Delta(\frak{n}'_1,\frak{h})$ to be 
 the collection of the following roots $\alpha$:
\begin{itemize}
\item $\op{Re}\langle \lambda, \alpha \rangle>0$,
\item $\op{Re}\langle \lambda, \alpha \rangle=0$
 and $\langle \rho(\frak{n}_1)|_{\frak{h}^{-\theta}}, \alpha \rangle>0$,
\item
 $\op{Re}\langle \lambda, \alpha \rangle=
 \langle \rho(\frak{n}_1)|_{\frak{h}^{-\theta}}, \alpha \rangle=0$
 and 
 $\op{Im} \langle \lambda, \alpha \rangle>0$.
\end{itemize}
Then the corresponding parabolic subalgebra
 $\frak{q}'_1=\frak{l}+\frak{n}'_1$ is an admissible polarization
 and satisfies condition (*).
If $\alpha\in \Delta(\frak{n}_1,\frak{h})$ is a real root, then 
 $\langle \rho(\frak{n}_1)|_{\frak{h}^{-\theta}}, \alpha \rangle
 =\langle \rho(\frak{n}_1), \alpha \rangle>0$.
Also, if $\alpha\in \Delta(\frak{n}_1,\frak{h})$ and 
 $\op{Im}\langle \lambda, \alpha \rangle=0$, then 
 $\op{Re}\langle \lambda, \alpha \rangle>0$
 by the admissibility of $\frak{n}_1$.
Hence the roots $\alpha$ in 
 $\Delta(\frak{n}_1,\frak{h})\setminus\Delta(\frak{n}'_1,\frak{h})$
 are all complex roots
 and satisfy $\op{Im} \langle \lambda, \alpha \rangle \neq 0$.

By applying Lemma~\ref{lemma:transfer2}, we have 
$\pi(\mathcal{O}, \Lambda, \frak{q}_1)=\pi(\mathcal{O}, \Lambda, \frak{q}'_1)$.
Therefore, we may assume that $\frak{q}_1$ and $\frak{q}_2$
 satisfy condition (*).
Then by the spectral sequence
 $(I_{\frak{q}_{\op{Im}}, L_{\op{Im},\mathbb{R}}\cap K_{\mathbb{R}}}^{\frak{g},K_{\mathbb{R}}})^i
(I_{\frak{q}_1\cap \frak{l_{\op{Im}}}, L_{\mathbb{R}}\cap K_{\mathbb{R}}}^{\frak{l_{\op{Im}}},L_{\op{Im},\mathbb{R}}\cap K_{\mathbb{R}}})^{j}(\cdot)
 \Longrightarrow
 (I_{\frak{q}_{1}, L_{1,\mathbb{R}}\cap K_{\mathbb{R}}}^{\frak{g},K_{\mathbb{R}}})^{i+j}(\cdot)$,
 we have
\begin{align*}
\pi(\mathcal{O}, \Lambda, \frak{q}_1)
= \sum_j (-1)^{i+j}(I_{\frak{q}_{\op{Im}}, L_{\op{Im},\mathbb{R}}\cap K_{\mathbb{R}}}^{\frak{g},K_{\mathbb{R}}})^i
(I_{\frak{q}_1\cap \frak{l_{\op{Im}}}, L_{\mathbb{R}}\cap K_{\mathbb{R}}}^{\frak{l_{\op{Im}}},L_{\op{Im},\mathbb{R}}\cap K_{\mathbb{R}}})^{j}
 (\Gamma_{\lambda}\otimes \mathbb{C}_{\rho(\mathfrak{n})})
\end{align*}
and a similar equation for
 $\pi(\mathcal{O}, \Lambda, \frak{q}_2)$.
Consequently, it is enough to show that 
\begin{align*}
(I_{\frak{q}_1\cap \frak{l_{\op{Im}}}, L_{\mathbb{R}}\cap K_{\mathbb{R}}}^{\frak{l_{\op{Im}}},L_{\op{Im},\mathbb{R}}\cap K_{\mathbb{R}}})^{j}(\Gamma_{\lambda}\otimes \mathbb{C}_{\rho(\frak{n})})
\simeq (I_{\frak{q}_2\cap \frak{l_{\op{Im}}}, L_{\mathbb{R}}\cap K_{\mathbb{R}}}^{\frak{l_{\op{Im}}},L_{\op{Im},\mathbb{R}}\cap K_{\mathbb{R}}})^{j}(\Gamma_{\lambda}\otimes \mathbb{C}_{\rho(\frak{n})}).
\end{align*}
However, the parabolic subalgebra $\frak{q}_1\cap \frak{l_{\op{Im}}}$
 of $\frak{l_{\op{Im}}}$ is defined over $\mathbb{R}$ and hence
the functor $I_{\frak{q}_1\cap \frak{l_{\op{Im}}}, L_{\mathbb{R}}\cap K_{\mathbb{R}}}^{\frak{l_{\op{Im}}},L_{\op{Im},\mathbb{R}}\cap K_{\mathbb{R}}}$ is a real parabolic induction.
Then for parabolic inductions, one can show the independence of the polarization by seeing that they have the same distribution character (see page 352 of \cite{Kna86}).
\end{proof}

\bibliographystyle{amsalpha}
\bibliography{UniversalBibtex}

\end{document}